\documentclass[12pt]{amsart}
\pdfoutput=1
\usepackage{amsfonts,amssymb,amsthm,eucal,amsmath}
\usepackage{pinlabel}
\usepackage{graphicx}
\usepackage[T1]{fontenc}
\usepackage{latexsym,url}
\usepackage{enumerate}

\usepackage{subfig}
\usepackage{listings}

\newtheorem{thm}{Theorem}[section]

\newtheorem{lemma}[thm]{Lemma}

\newtheorem{cor}[thm]{Corollary}
\theoremstyle{definition}
\newtheorem{defn}[thm]{Definition}

\newtheorem{ex}[thm]{Example}
\newtheorem{exercise}[thm]{Exercise}
\theoremstyle{remark}
\newtheorem{rmk}[thm]{Remark}

\newcommand{\co}{\colon\thinspace}

\def\H{\mathbb{H}}
\def\R{\mathbb{R}}
\def\Z{\mathbb{Z}}
\def\E{\mathbb{E}}

\def\tri{\mathcal{T}}
\def\bdry{\partial}

\def \bA{\mathbf{A}} %for point configs
\def \bp{\mathbf{p}}
\def\kk{\boldsymbol{k}}
\def\aa{\boldsymbol{a}}
\def\bb{\boldsymbol{b}}
\def\cc{\boldsymbol{c}}
\def\bnu{\boldsymbol{\nu}}

\usepackage[inner=40mm, outer=40mm, textheight=225mm]{geometry}

\usepackage{url}
\usepackage[headings]{fullpage}
\usepackage[bookmarks=true,%
    colorlinks=true,%
    linkcolor=blue,%
    citecolor=blue,%
    filecolor=blue,%
    menucolor=blue,%
    urlcolor=blue,%
    breaklinks=true]{hyperref}

\def\BN{\mathbb N}
\def\BZ{\mathbb Z}

\def\BR{\mathbb R}
\def\BC{\mathbb C}

\def\BH{\mathbb H}

\def\calF{\mathcal F}

\def\calX{\mathcal X}

\def\a{\alpha}

\def\d{\delta}

\def\ti{\widetilde}

\def\longto{\longrightarrow}

\def\pt{\partial}

\def\lbl{\label}
\newcommand{\CT}{\mathcal{T}}
\def\be{  \begin{equation} }
\def\ee{  \end{equation} }
\def\ID{I_{\Delta}}
\def\EP{\mathrm{EP}}
\newcommand{\wt}{\widetilde}
\def\GA{\mathrm{GA}}
\def\SA{\mathrm{SA}}
\newcommand{\mb}{\mathbf}

\def\bd{\partial}
\def\JD{J_{\Delta}}

\begin{document}

%%%%%%%%%%%%%%%%%%%%%%{page1}

\title[1-efficient triangulations and the index of a cusped hyperbolic 
3-manifold]{1-efficient triangulations and the index of a cusped hyperbolic 
3-manifold}
\author{Stavros Garoufalidis}
\address{School of Mathematics \\
         Georgia Institute of Technology \\
         Atlanta, GA 30332-0160, USA \newline
         {\tt \url{http://www.math.gatech.edu/~stavros }}}
\email{stavros@math.gatech.edu}
\author{Craig D. Hodgson}
\address{Department of Mathematics and Statistics \\
         The University of Melbourne \\
         Parkville, VIC, 3010, AUSTRALIA
         \tt{\url{http://www.ms.unimelb.edu.au/~cdh }}}
\email{craigdh@unimelb.edu.au}
\author{J. Hyam Rubinstein}
\address{Department of Mathematics and Statistics \\
         The University of Melbourne \\
         Parkville, VIC, 3010, AUSTRALIA
         \tt{\url{http://www.ms.unimelb.edu.au/~rubin }}}
\email{rubin@ms.unimelb.edu.au}
\author{Henry Segerman}
\address{Department of Mathematics and Statistics \\
         The University of Melbourne \\
         Parkville, VIC, 3010, AUSTRALIA
         \tt{\url{http://www.ms.unimelb.edu.au/~segerman }}}
\email{segerman@unimelb.edu.au}
\thanks{S.G. was supported in part by grant DMS-0805078 of the US 
National Science Foundation. C.D.H, J.H.R, H.S are supported by the Australian Research Council grant DP1095760.
\\
\newline
2010 {\em Mathematics Classification.} Primary 57N10, 57M50. Secondary 57M25.
}

\date{\today }

\begin{abstract}
In this paper we will promote the 3D index of an ideal triangulation $\CT$ 
of an oriented cusped 3-manifold $M$ 
(a collection of $q$-series with integer coefficients, 
introduced by Dimofte-Gaiotto-Gukov) to a topological invariant of oriented cusped
hyperbolic 3-manifolds. To achieve our goal we show that (a) $\CT$ admits an 
index structure if and only if $\CT$ is 1-efficient and 
(b) if $M$ is hyperbolic, it has a canonical set of 1-efficient ideal 
triangulations related by 2-3 and 0-2 moves which preserve the 3D index. 
We illustrate our results with several examples. 
\end{abstract}

\maketitle

\tableofcontents

%%%%%%%%%%%%%%%%%%%%%%%%%%%%%%%%%%%%%%%%%%%%%%%%%%%%%%%%%%%%%%%%%%%%%%%%%%%%
%%%%%%%%%%%%%%%%%%%%%%%%%%%%%%%%%%%%%%%%%%%%%%%%%%%%%%%%%%%%%%%%%%%%%%%%%%%%

\section{Introduction}
\label{sec.intro}

\subsection{The 3D index of Dimofte-Gaiotto-Gukov}
\label{sub.DGG}

The goal of this paper is to convert the index of an ideal triangulation
$\CT$ (a remarkable collection of Laurent series in $q^{1/2}$ introduced by 
Dimofte-Gaiotto-Gukov \cite{DGG1,DGG2} and further studied in \cite{Ga:index})
to a topological invariant of oriented cusped hyperbolic 3-manifolds $M$. 
Our goal will be achieved in two steps. 

The first step identifies the 
existence of an index 
structure of $\CT$ (a necessary and sufficient condition for the existence
of the index of $\CT$; see \cite{Ga:index}) with the non-existence
of sphere or non vertex-linking torus {\em normal surfaces} of $\CT$; see Theorem
\ref{index_structure_iff_1-efficient} below. Such ideal triangulations are called 1-efficient in 
\cite{JR,KR2}. 
The unexpected connection between the 
index of an ideal triangulation (a recent quantum object) and the 
classical theory of normal surfaces places restrictions
on the topology of $M$; see Remark \ref{rem.1} below. 

The second step
constructs a canonical collection $\calX_M^{\EP}$ of triangulations
of the Epstein-Penner ideal cell decomposition of a cusped hyperbolic 
3-manifold $M$, such that the index behaves well with respect to 2--3 and 0--2 moves that connect any two members of $\calX_M^{\EP}$. The index
of those triangulations then gives the desired topological invariant of $M$;
see Theorem \ref{thm.2} below.

We should point out that normal surfaces were
also used by Frohman-Bartoczynska~\cite{FK} in an attempt to construct 
topological invariants of
3-manifolds, in the style of a Turaev-Viro TQFT. Strict angle structures 
(a stronger form of an index structure) play a role in quantum hyperbolic 
geometry studied by Baseilhac-Benedetti \cite{BB1,BB2}. In the recent work 
of Andersen-Kashaev~\cite{AK}, strict angle structures
were used as sufficient conditions for convergence of analytic state-integral 
invariants of ideal triangulations. The latter invariants are 
expected to depend on the underlying cusped 3-manifold and to form
a generalization of the Kashaev invariant 
\cite{Ka}. The $q$-series of Theorem \ref{thm.2} below are $q$-holonomic,
of Nahm-type and, apart from a meromorphic singularity at $q=0$, admit analytic
continuation in the punctured unit disc. 

Before we get to the details, we should stress that the origin
of the 3D index is the exciting work of Dimofte-Gaiotto-Gukov 
\cite{DGG1,DGG2} (see also \cite{holo-blocks})
who studied gauge theories with $N=2$ supersymmetry that are 
associated to an ideal triangulation $\CT$ of an oriented 3-manifold $M$ 
with at least one cusp. The low-energy limit of these gauge theories gives
rise to a partially defined function, the so-called {\em 3D index}
\be
\lbl{eq.ICT}
I: \{\text{ideal triangulations}\}
\longto \BZ((q^{1/2}))^{H_1(\pt M;\BZ)}, \qquad \CT \mapsto 
I_{\CT}([\varpi]) \in \BZ((q^{1/2}))
\ee
for $[\varpi] \in H_1(\pt M;\BZ)$.\footnote{Here and below we will  
use the notation $M$ for both a cusped hyperbolic 3-manifold 
and the corresponding compact manifold with boundary $\bd M$ consisting of a disjoint union of tori; 
the intended meaning should be clear from the context.}
The function $I$ is only partially defined because the expression for the 3D index may not converge.
The above gauge theories provide an analytic
continuation of the coloured Jones polynomial and play an important role
in Chern-Simons perturbation theory and in categorification.
Although the gauge theory depends on the ideal triangulation $\CT$, 
and the 3D index in general may not be defined, physics predicts that the 
gauge theory ought to be a topological invariant of the 
underlying 3-manifold $M$. Recall that any two ideal triangulations 
of a cusped 3-manifold are related by a sequence of 2-3 moves 
\cite{Ma1,Ma2,Pi}. In \cite{Ga:index} the following was shown.
For the definition of an index structure, see Section \ref{sec.definitions}.

\begin{thm}
\lbl{thm.00}
\rm{(a)}
$I_{\CT}$ is well-defined if and only if $\CT$ admits an {\em index structure}.
\newline
\rm{(b)} If $\CT$ and $\CT'$ are related by a 2--3 move and both admit an
index structure, then $I_{\CT}=I_{\CT'}$.
\end{thm}

\subsection{Index structures and 1-efficiency}
\lbl{sub.1eff}

\begin{thm}
\lbl{index_structure_iff_1-efficient}
An ideal triangulation $\CT$ of an oriented 3-manifold with cusps admits 
an index structure if and only if $\CT$ is 1-efficient.
\end{thm}

The above theorem has some consequences for our sought topological invariants.

\begin{rmk}
\lbl{rem.1}
1-efficiency of $\CT$ implies restrictions on the topology of $M$: it follows
that $M$ is irreducible and atoroidal. 
Note that here by \emph{atoroidal}, we mean that any \emph{embedded} torus 
is either compressible or boundary parallel. It follows by 
Thurston's Hyperbolization Theorem 
in dimension $3$ that $M$ is hyperbolic or 
small Seifert-fibred.
\end{rmk}

\begin{rmk}
\lbl{rem.2}
If $K$ is the connected sum of the $4_1$ and $5_2$ knots, or $K'$ is the
Whitehead double of the $4_1$ knot and $\CT$ is any ideal triangulation of
the complement of $K$ or $K'$, then $\CT$ is not 1-efficient, thus
$I_{\CT}$ never exists. On the other hand, the (coloured) Jones polynomial, 
the Kashaev invariant and the $\mathrm{PSL}(2,\BC)$-character variety of 
$K$ and $K'$ happily exist; see \cite{Jo,Ka,CCGLS}.
\end{rmk}

\begin{thm}
\lbl{semi-angle => 1-efficient}
Let $\CT$ be an ideal triangulation of an oriented atoroidal 3--manifold with 
at least one cusp. If $\CT$ admits a semi--angle structure then $\CT$ is 1--efficient.
\end{thm}

\begin{rmk}
\lbl{rem.3}
Taut and strict
angle structures are examples of semi-angle structures, and for these cases this is proved in \cite[Thm.2.6]{KR2}. In Section \ref{sec: index structures and 1-efficiency}, we give a brief outline of the argument for a general semi--angle structure.
\end{rmk}

\begin{rmk}
In Corollary \ref{atoroidal=>1-efficient}, we note that a construction of Lackenby produces triangulations with taut angle structures, which are therefore 1--efficient, on all irreducible an-annular cusped 3--manifolds. However, it is not clear that the triangulations produced by this construction are connected by the appropriate 2--3 and 0--2 moves, so we cannot prove that the 3D index is independent of the choice of taut triangulation for the manifold. 
\end{rmk}

\subsection{Regular ideal triangulations and topological invariance}
\lbl{sub.regular}

In view of Remark \ref{rem.1}, we restrict our attention to 
hyperbolic 3-manifolds $M$ with at least one cusp. 
All we need is a canonical set 
$\calX_M$ of $1$--efficient ideal triangulations of $M$ such that any two of these 
triangulations are related by moves 
that preserve $I_{\CT}$. From Theorem \ref{thm.00}, we know that we can use 2--3 and 3--2 moves for this purpose. Given the choice we will make for $\calX_M$ below, it turns out that we will also need to use 0--2 and 2--0 moves to connect together the triangulations of $\calX_M$. 
Using the dual language of special spines, it is shown in 
\cite[Lem.2.1.11]{Ma2} and \cite{Pi}
(see also \cite[Prop.I.1.13]{Petronio:thesis})
that the 0--2 and 2--0 moves can be derived from the 2--3 and 3--2 moves, as long as the triangulation has at least two tetrahedra. However, the required sequence of 2--3 and 3--2 moves takes us out of our set $\calX_M$, and it is not clear that the triangulations the sequence passes through are 1--efficient.

Every cusped hyperbolic 3--manifold $M$ has
a canonical {\em cell decomposition} \cite{EP} where the cells are convex
ideal polyhedra in $\BH^3$. The cells can be triangulated into ideal 
tetrahedra, with layered flat tetrahedra inserted to form a bridge between two polyhedron faces that are supposed to be glued to each other but whose induced triangulations do not match. Unfortunately, it is not known whether any two triangulations
of a 3-dimensional polyhedron are related by 2--3 and 3--2 moves; the corresponding
result trivially holds in dimension 2 and nontrivially fails in dimension
$5$; \cite{DeLoera:book,Santos:ICM}. Nonetheless, it was shown by
Gelfand-Kapranov-Zelevinsky that any two {\em regular} triangulations 
of a polytope in $\BR^n$ are related by a sequence of geometric bistellar 
flips; \cite{GKZ}. Using the Klein model of $\BH^3$, we define the notion
of a \emph{regular ideal triangulation} of an ideal polyhedron and observe that
every two regular ideal triangulations are related by a sequence of geometric
2--2, 2--3 and 3--2 moves. Our set $\calX_M^{\EP}$ of ideal triangulations of a cusped hyperbolic manifold $M$ consists of all possible choices of regular triangulation for each ideal polyhedron, together with all possible ``bridge regions'' of layered flat tetrahedra joining the induced triangulations of each identified pair of polyhedron faces.
From the geometric structure of the cell decomposition, we obtain a natural semi-angle structure on each triangulation of $\calX_M^{\EP}$, which shows that they are all 1-efficient by Theorem \ref{semi-angle => 1-efficient}, and so the 3D index is defined for each triangulation by Theorems \ref{thm.00}(a) and  \ref{index_structure_iff_1-efficient}.  We show that any two of these triangulations are related to each other by a sequence of 2--3, 3--2, 0--2 and 2--0 moves through 1--efficient triangulations, the moves all preserving the 3D index, using Theorems \ref{thm.00}(b), \ref{index_structure_iff_1-efficient} and \ref{thm.20move.index}. (The intermediate triangulations are mostly also within $\calX_M^{\EP}$, although we sometimes have to venture outside of the set briefly.) Therefore we obtain a topological invariant of cusped hyperbolic 3--manifolds $M$.

\begin{thm}
\lbl{thm.2}
If $M$ is a cusped hyperbolic 3-manifold, and $\CT \in \calX_M^{\EP}$
we have $I_M:=I_{\CT}$ is well-defined. 
\end{thm}

The next theorem is of independent interest, and may be useful for the
problem of contructing topological invariants of cusped hyperbolic 3-manifolds.
For a definition of the gluing equations of an ideal triangulation, see \cite{NZ,Th} 
and also Section \ref{sub.ge} below.

\begin{thm}
\lbl{thm.3}
Fix a cusped hyperbolic 3-manifold $M$. 
\newline
\rm{(a)} For every $\CT \in \calX_M^{\EP}$, there exists a solution $Z_{\CT}$
to the gluing equations of $\CT$ which recovers the complete hyperbolic 
structure on $M$. Moreover, all shapes of $Z_{\CT}$ have non-negative 
imaginary part.
\newline
\rm{(b)} If $\CT, \CT' \in \calX_M^{\EP}$ are related by 2--3, 3--2, 0--2 and 
2--0 moves, then so are $Z_{\CT}$ and $Z_{\CT'}$.
\newline
\rm{(c)} For every $\CT$, the arguments of $Z_{\CT}$ give a semi-angle structure
on $\CT$.  
\end{thm}

\begin{rmk}
\lbl{rmk.5}
In \cite{HRS}, it is shown that a cusped hyperbolic 3--manifold $M$ admits an ideal triangulation with strict angle structure if $H_1(M, \partial M;\BZ_2)=0$. All link complements in the 3--sphere satisfy this condition. Such triangulations admit index structures but it is not known if they can be connected by 2--3 
and 0--2 moves within the class of 1--efficient triangulations. 
\end{rmk}

\begin{rmk}
\lbl{rmk.XEPM}
For a typical cusped hyperbolic manifold, one expects that the Epstein-Penner
ideal cell decomposition consists of ideal tetrahedra, i.e., that 
$\calX_M^{\EP}$ consists of one element. Many examples of such
cusped hyperbolic manifolds appear in the census \cite{SnapPy} and also
in \cite{Akiyoshi,GS}.
\end{rmk}

\begin{rmk}  
In a later paper we will extend this work in the following ways:
\begin{itemize}
\item  extend the domain of the 3D index $I_\CT([\varpi])$
to $[\varpi]\in H_1(\bd M; \frac{1}{2}\Z)$ such that
$2[\varpi] \in\mathrm{Ker}(H_1(\bd M; \Z) \to H_1(M; \Z/2\Z))$,
\item give a definition of the 3D index using singular normal surfaces in $M$.
\end{itemize}
\end{rmk}

\begin{rmk}
Theorem \ref{thm.2} constructs a family of $q$-series $I_M([\varpi])(q)$
(parametrized by $[\varpi] \in H_1(\pt M,\BZ)$) associated to a 
cusped hyperbolic manifold $M$. When $M = S^3 \setminus K$ is the complement of a knot $K$,
we can choose $[\varpi]=\mu$ to be the homology class of the meridian 
and consider the series
\begin{equation}
\lbl{eq.indtotal}
I^{\mathrm{tot}}_K(q)=\sum_{e \in \BZ} I_M(e \mu)(q)
\end{equation}
Since the semi-angle structures of Theorem \ref{thm.3} have zero holonomy
at all peripheral curves, it can be shown that $I^{\mathrm{tot}}_K(q)$ is 
well-defined. It turns out that $I^{\mathrm{tot}}_K(q)$ is closely related to
the state-integral invariants of Andersen-Kashaev and Kashaev-Luo-Vartanov
\cite{AK,KLV}. The relation between state-integrals of the quantum dilogarithm
and $q$-series is explained in detail in \cite{GK}. An empirical study
of the asymptotics of the series $I^{\mathrm{tot}}_{4_1}(q)$ is given in
\cite{GZ}.
\end{rmk}

\subsection{Plan of the paper}
\label{sub.plan}

In Section \ref{sec.definitions} we review the basic definitions of ideal
triangulations, efficiency, angle structures and index 
structures.

In Section \ref{sec: index structures and 1-efficiency} we prove Theorem 
\ref{index_structure_iff_1-efficient}. So for an ideal triangulation, existence of an index structure is equivalent to being
1-efficient. 

In Section \ref{sec.review} we review the basic properties of the tetrahedron
index from \cite{Ga:index}, and give a detailed discussion of the 3D index for an 
ideal triangulation of a cusped 3-manifold. In Section \ref{sec.02move} we study the
behaviour of the 3D index under the 0--2 and 2--0 move.

In Section \ref{sec.XEPM} we discuss the Epstein-Penner ideal cell 
decomposition and its subdivision into regular triangulations. At the end
of Section \ref{sub.moving.paths} we prove Theorems \ref{thm.2} and 
\ref{thm.3}.

In Section \ref{sec.computation} we compute the first terms of the 3D index for some example manifolds.

Finally in the appendix, we give a detailed and self-contained proof 
of the invariance of the 3D index of 1-efficient triangulations
under 2--3 moves, following \cite{Ga:index} and \cite{DGG2}. 

%%%%%%%%%%%%%%%%%%%%%%%%%%%%%%%%%%%%%%%%%%%%%%%%%%%%%%%%%%%%%%%%%%%%%%%%%%%%
%%%%%%%%%%%%%%%%%%%%%%%%%%%%%%%%%%%%%%%%%%%%%%%%%%%%%%%%%%%%%%%%%%%%%%%%%%%%

\section{Definitions}
\label{sec.definitions}

\begin{defn}
Let $M$ be an orientable topologically finite 3-manifold which is the interior of a compact 3-manifold with torus boundary components. An \emph{ideal triangulation} $\tri$ of $M$ consists of a pairwise disjoint union of standard Euclidean 3--simplices, $\widetilde{\Delta} = \cup_{k=1}^{n} \widetilde{\Delta}_k,$ together with a collection $\Phi$ of Euclidean isometries between the 2--simplices in $\widetilde{\Delta},$ called \emph{face pairings}, such that the quotient space 
$(\widetilde{\Delta} \setminus \widetilde{\Delta}^{(0)} )/ \Phi$ is homeomorphic to $M.$
The images of the simplices in $\tri$ may be singular in $M$.
\end{defn}

\begin{defn}
Let $\CT$ be an ideal triangulation with at least 2 distinct tetrahedra. A \emph{2--3 move} can be performed on any pair of distinct tetrahedra of $\CT$ that share a triangular face $t$. We remove $t$ and the two tetrahedra, and replace them with three tetrahedra arranged around a new edge, which has endpoints the two vertices not on $t$. See Figure \ref{2-3}. A \emph{3--2 move} is the reverse of a 2--3 move, and can be performed on any triangulation with a degree 3 edge, where the three tetrahedra incident to that edge are distinct.
\end{defn}

\begin{defn}
Let $\CT$ be an ideal triangulation. A \emph{0--2 move} can be performed on any pair of distinct triangular faces of $\CT$ that share an edge $e$\footnote{Unlike for the 2--3 move, it is possible to make sense of the 0--2 move when the two triangles are not distinct. However, we will not make use of this variant in this paper.}. Around the edge $e$, the tetrahedra of $\CT$ are arranged in a cyclic sequence, which we call a \emph{book of tetrahedra}. (Note that tetrahedra may appear more than once in the book.) The two triangles and $e$ separate the book into two \emph{half--books}. We unglue the tetrahedra that are identified across the two triangles, duplicating the triangles and also duplicating $e$. We glue into the resulting hole a pair of tetrahedra glued to each other in such a way that there is a degree 2 edge between them. See Figure \ref{0-2}. A \emph{2--0 move} is the reverse of a 0--2 move, and can be performed on any triangulation with a degree 2 edge, where the two tetrahedra incident to that edge are distinct, there are no face pairings between the four external faces of the two tetrahedra, and the two edges opposite the degree 2 edge are not identified. 
\end{defn}

\begin{figure}[htb]

\centering
\subfloat[The 2--3 and 3--2 moves.]{
\labellist
\small\hair 2pt
\pinlabel 2--3 at 139 84
\pinlabel 3--2 at 139 44
\endlabellist
\includegraphics[width=0.44\textwidth]{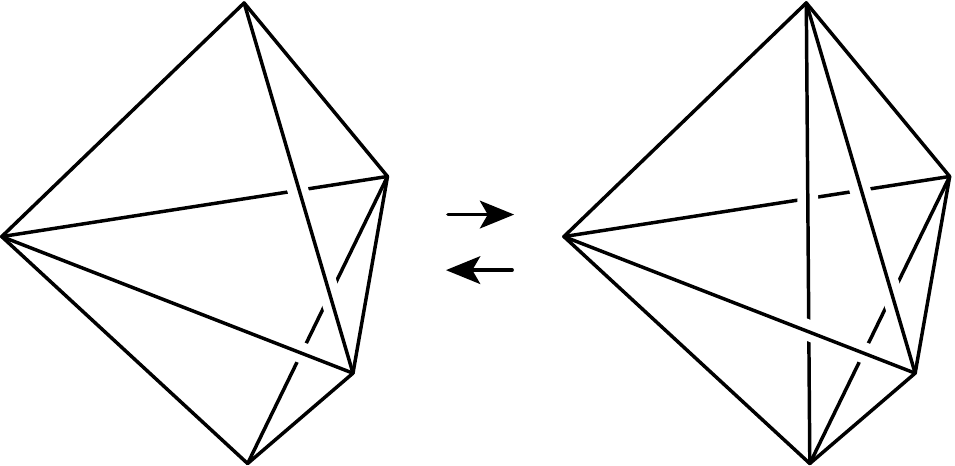}
\label{2-3}}
\qquad
\subfloat[The 0--2 and 2--0 moves.]{
\labellist
\small\hair 2pt
\pinlabel 0--2 at 109 84
\pinlabel 2--0 at 109 44
\endlabellist
\includegraphics[width=0.34\textwidth]{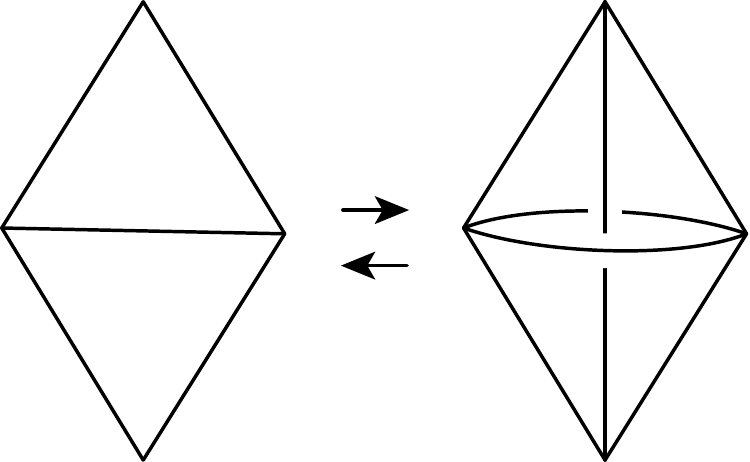}
\label{0-2}}
\caption{Moves on (topological) triangulations.}
\label{2-3 and 0-2 moves}
\end{figure}

\begin{rmk}
A 0--2 move is also called a {\em lune}
move in the dual language of standard spines \cite{Ma1,Ma2,Pi,BP}. In
\cite[Lem.2.1.11]{Ma1} and \cite{Pi} (see also 
\cite[Prop.I.1.13]{Petronio:thesis}) it was shown that a 0--2 move follows 
from a combination of 2--3 moves as long as the initial triangulation 
has at least $2$ ideal tetrahedra.
\end{rmk}

\begin{defn}
Let $\Delta^3$ be the standard 3--simplex with a chosen orientation. Each pair of opposite edges corresponds to a normal isotopy class of quadrilateral discs in $\Delta^3,$ disjoint from the pair of edges. We call such an isotopy class a \emph{normal quadrilateral type}. Each vertex of $\Delta^3$ corresponds to a normal isotopy class of triangular discs in $\Delta^3,$ disjoint from the face of $\Delta^3$ opposite the vertex. We call such an isotopy class a \emph{normal triangle type}. Let $\tri^{(k)}$ be the set of all $k$--simplices in $\tri$. If $\sigma \in \tri^{(3)},$ then there is an orientation preserving map $\Delta^3 \to \sigma$ taking the $k$--simplices in $\Delta^3$ to elements of $\tri^{(k)},$ and which is a bijection between the sets of normal quadrilateral and triangle types in $\Delta^3$ and in $\sigma$. 
Let $\square$ and $\triangle$ denote the sets of all normal quadrilateral and triangle types in $\tri$ respectively.
\end{defn}

\begin{defn}
Given a $3$-manifold $M$ with an ideal triangulation $\tri$,
the \emph{normal surface solution space}  $C(M; \tri)$ is a vector subspace of $\mathbb{R}^{7n},$ where $n$ is the number of tetrahedra in $\tri$, consisting of vectors  satisfying the \emph{compatibility equations} of normal surface theory. The coordinates of $x \in \mathbb{R}^{7n}$ represent weights of the four normal triangle types and the three normal quadrilateral types in each tetrahedron, and the compatibility equations state that normal triangles and quadrilaterals have to meet the 2--simplices of $\tri$ with compatible weights. 

A vector in $\mathbb{R}^{7n}$ is called \emph{admissible} if at most one quadrilateral coordinate from each tetrahedron is non-zero and all coordinates are non-negative. An integral admissible element of $C(M; \tri)$ corresponds to a unique embedded, closed \emph{normal surface} in $(M,\tri)$ and vice versa.
\end{defn}

\begin{defn} (See \cite{JR}, \cite{KR2})
An ideal triangulation $\tri$ of an orientable $3$-manifold is \emph{0-efficient} if there are no embedded normal 2-spheres or one-sided projective planes. An ideal triangulation $\tri$ is \emph{1-efficient} if it is 0-efficient, the only embedded normal tori are vertex-linking and there are no embedded one-sided normal Klein bottles. An ideal triangulation $\tri$ is \emph{strongly 1-efficient} if there are no immersed normal 2--spheres, projective planes or Klein bottles and the only immersed normal tori are coverings of the vertex-linking tori.
\end{defn}

Note that in some contexts, ``atoroidal'' is taken to mean that there is no immersed torus whose fundamental group injects into the fundamental group of the 3--manifold. In our context, we mean that there are no embedded incompressible tori or Klein bottles, other than tori isotopic to boundary components. In Corollary \ref{atoroidal=>1-efficient} and Remark \ref{small SFS lifted may not have index structure} we highlight this distinction.  

Note that if $M$ is orientable, it is sufficient to consider only normal 2-spheres and tori, except in the special case that $M$ is a twisted $I$-bundle over a Klein bottle. For any embedded normal projective plane or Klein bottle must be one-sided, so the boundary of a small regular neighbourhood is a normal 2-sphere or torus. However in the non-orientable case, one must consider two-sided projective planes and Klein bottles. In this paper we will consider only the orientable case. 

\begin{defn}
If $e\in \tri^{(1)}$ is any edge, then there is a sequence $(q_{n_1}, ..., q_{n_k})$ of normal quadrilateral types facing $e,$ which consists of all normal quadrilateral types dual to $e$ listed in sequence as one travels around $e.$ Then $k$ equals the degree of $e,$ and a normal quadrilateral type may appear at most twice in the sequence. This sequence is called the \emph{normal quadrilateral type sequence} for $e$ and is well-defined up to cyclic permutations and reversing the order. 
\end{defn}

\begin{defn}\label{angle structures}
A function $\alpha\co \square \to \mathbb{R}$ is called a \emph{generalised angle structure on $(M,\tri)$} if it satisfies the following two properties:
\begin{enumerate}
\item If $\sigma^3 \in \tri^{(3)}$ and $q, q', q''$ are the three normal quadrilateral types supported by it, then
\begin{equation*}
   \alpha(q) + \alpha(q') + \alpha(q'') =\pi.
\end{equation*}
\item If $e\in \tri^{(1)}$ is any edge and $(q_{n_1}, ..., q_{n_k})$ is its normal quadrilateral type sequence, then
\begin{equation*}
\sum_{i=1}^k \alpha(q_{n_i})  =2\pi.
\end{equation*}
\end{enumerate}
\end{defn}

Dually, one can regard $\alpha$ as assigning angles $\alpha(q)$ to the two edges opposite $q$ in the tetrahedron containing $q$. The triangulations we consider are of oriented manifolds, so we may assume that the triangulation is also oriented. We fix an ordering $q\rightarrow q' \rightarrow q'' \rightarrow q$ on these quad types, well defined up to cyclic permutation. See Figure \ref{quad_types}.\\

\begin{figure}[htbp]
\labellist
\small\hair 2pt
\pinlabel $q$ at 135 20
\pinlabel $q'$ at 430 20
\pinlabel $q''$ at 725 20
\endlabellist
\centering
\includegraphics[width=0.6\textwidth]{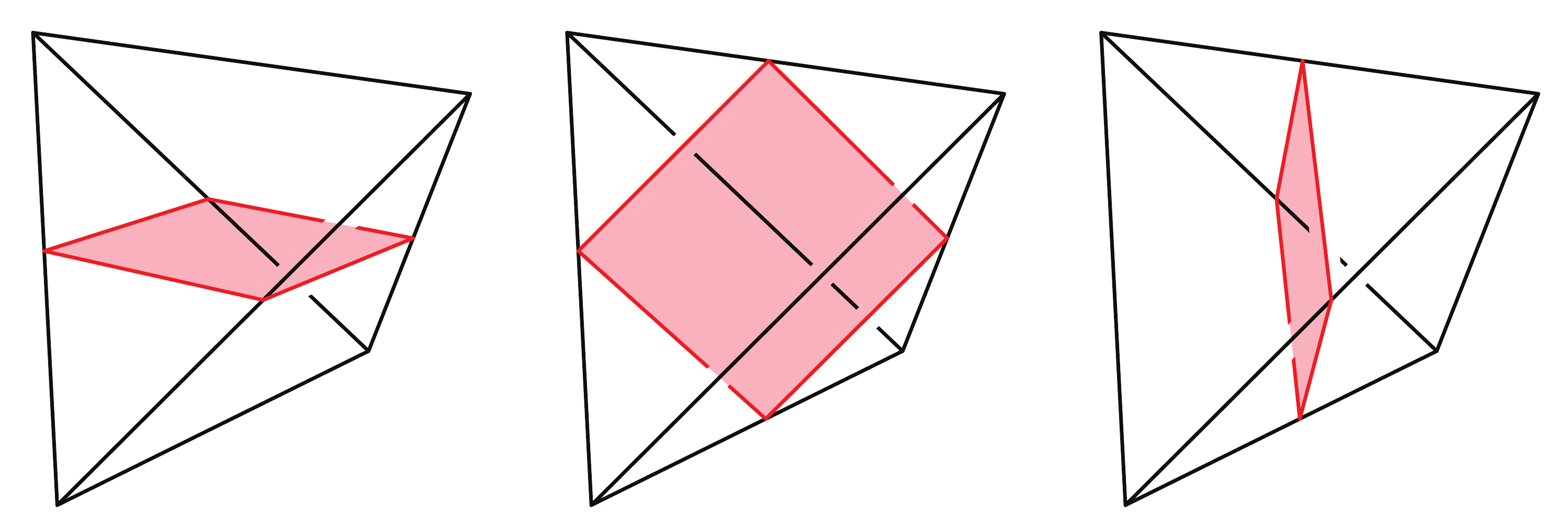}
\caption{The three quad types within an oriented tetrahedron, arranged in our chosen cyclic order.}
\label{quad_types}
\end{figure}

\begin{defn}
If we restrict the angles of a generalised angle structure to be in 
\begin{itemize}
\item $[0,\pi]$, then the generalised angle structure is a \emph{semi-angle structure}.
\item $(0,\pi)$, then the generalised angle structure is a \emph{strict angle structure}.  
\item $\{0,\pi\}$, then the generalised angle structure is a \emph{taut angle structure}.
\end{itemize}
The set of generalised angle structures is denoted by $\GA(\CT)$ and is an affine subspace of $\R^{3N}$, where $N$ is the number of tetrahedra in $\tri$. The subset of semi-angle structures is denoted by $\SA(\CT)$, and is a closed 
polytope in $\GA(\CT)$.
\end{defn}

\begin{rmk}
It is easy to see that a taut angle structure can only happen if every tetrahedron has a pair of opposite edges with angles $\pi$ and the other four edges have angles $0$.
\end{rmk}

\begin{defn}\label{choice_of_quads}
For an ideal triangulation $\tri$ with $N$ tetrahedra, a \emph{quad-choice} is an element 
$Q=(Q_1, \ldots, Q_N) \in \square^N$ such that $Q_n$ is a choice of one of the three quad types in the $n$th tetrahedron. An \emph{index structure} $\alpha$ on $\tri$ consists of $3^N$ generalised angle structures, indexed by the quad-choices $Q$, with the property that $\alpha_Q(Q_n)>0$ for $n=1,\ldots,N$, for each quad-choice $Q.$
\end{defn}

\begin{defn}\label{matrix equations}
The equations determining a generalised angle structure can be read off as three $N\times N$ matrices 
$\overline{\mb A}=(\bar a_{ij})$, $\overline{\mb B}=(\bar b_{ij})$ and $\overline{\mb C}=(\bar c_{ij})$
whose rows are indexed by the $N$ edges of $\CT$ and whose columns are 
indexed by the $\alpha(q_j) ,\alpha(q_j'),\alpha(q_j'')$ variables respectively, where $q_j, q_j', q_j''$ are the quads type in the $j$th tetrahedron.   These are the so-called 
{\em Neumann-Zagier matrices} that encode the exponents of the 
{\em gluing equations} of $\CT$, originally introduced by Thurston 
\cite{NZ,Th}. In terms of these matrices, a generalised angle structure is
a triple of vectors $Z, Z',  Z''  \in \BR^N$ that satisfy the equations
\be
\lbl{eq.angle1}
\overline{\mb A} \, Z + \overline{\mb B} \, Z' 
+ \overline{\mb C} \, Z'' = 2 \pi (1,\dots,1)^T,  \qquad   Z+Z'+Z''=\pi (1,\dots,1)^T \,.
\ee
Note that the matrix entries $\bar a_{ij},\bar b_{ij},\bar c_{ij}$ give the coefficients
of $Z_j,Z'_j,Z''_j$ in the $i$th edge equation corresponding to the edges of tetrahedron $j$
facing quad types $q_j,q'_j,q''_j$ respectively.

We can combine these into a single matrix equation
\be
\label{full_matrix_eqn}
\begin{pmatrix} \overline{\mb A} & \overline{\mb B} & \overline{\mb C} \\ 
\mb I_N & \mb I_N & \mb I_N 
\end{pmatrix} \begin{pmatrix}Z \\ 
Z'\\
Z'' 
\end{pmatrix} = 
\begin{pmatrix} 2 \pi (1,\dots,1)^T \\ \pi (1,\dots,1)^T \end{pmatrix},
\ee
where $\mb I_N$ is the $N\times N$ identity matrix. We call this matrix equation \emph{the matrix form of the generalised angle structure equations}.

\end{defn}

%%%%%%%%%%%%%%%%%%%%%%%%%%%%%%%%%%%%%%%%%%%%%%%%%%%%%%%%%%%%%%%%%%%%%%%%%%%%
%%%%%%%%%%%%%%%%%%%%%%%%%%%%%%%%%%%%%%%%%%%%%%%%%%%%%%%%%%%%%%%%%%%%%%%%%%%%

\section{Index structures and 1--efficiency}
\label{sec: index structures and 1-efficiency}

We first give a sketch proof of Theorem \ref{semi-angle => 1-efficient}, showing that a semi--angle structure implies 1--efficiency. We follow \cite{KR2} and indicate the required small modification. Suppose that
$M$ is oriented with cusps and has an ideal triangulation $\tri$ with a semi-angle structure. Assume that there is an embedded normal torus or Klein bottle or sphere or projective plane, where the normal torus is not a peripheral torus.  Firstly, exactly as in \cite{Lackenby} the latter two cases are excluded by a simple Euler characteristic argument. 
Similarly, if there is a cube with knotted hole bounded by an embedded normal torus, then a barrier argument as in \cite{JR} establishes that there is a normal 2-sphere bounding a ball containing this normal torus, which is a contradiction. Embedded Klein bottles are excluded, so we are reduced to the cases of an embedded essential non peripheral normal torus or a normal torus bounding a solid torus. 

In both cases, there is a sweepout between the normal torus and a peripheral normal torus (for essential tori) or to a core circle of the solid torus. 
By a minimax argument (see \cite{Rub}, \cite{Sto}), there is an almost normal torus associated with this sweepout. This is either obtained by attaching a tube parallel to an edge to a normal 2-sphere or has a single properly embedded octagonal disc in a tetrahedron and a collection of normal triangular and quadrilateral discs. 
The first case is excluded, since we have ruled out such normal 2-spheres. 

The semi-angle structure now implies that a standard combinatorial Gauss-Bonnet argument can be applied. Each polygonal disc in our torus has curvature given by $\Sigma_i \alpha_i  - (n-2)\pi$, where $n$ is the number of edges of the disc and $\alpha_i$ are the interior angles at the vertices of the disc. Gauss-Bonnet then says that the sum of the curvatures of all the discs is zero, since the Euler characteristic of the torus is zero. Every normal triangular disc contributes zero and each normal quadrilateral is non-positive in the curvature sum. On the other hand, any embedding of an octagon into an ideal tetrahedron with a semi-angle structure gives a strictly negative contribution. See Figure \ref{normal_octagon}. Hence the Euler characteristic of such a surface cannot be zero and there could not have been an embedded normal torus to begin with. This completes the sketch proof. 
\qed

\begin{figure}[htbp]
\centering

\labellist
\small\hair 2pt
\pinlabel $\alpha$ at 174 335
\pinlabel $\alpha$ at 249 326
\pinlabel $\alpha$ at 141 76
\pinlabel $\alpha$ at 280 89
\pinlabel $\gamma$ at 84 200
\pinlabel $\gamma$ at 350 184
\pinlabel $\beta$ at 241 212
\pinlabel $\beta$ at 170 152
\endlabellist

\includegraphics[width=0.5\textwidth]{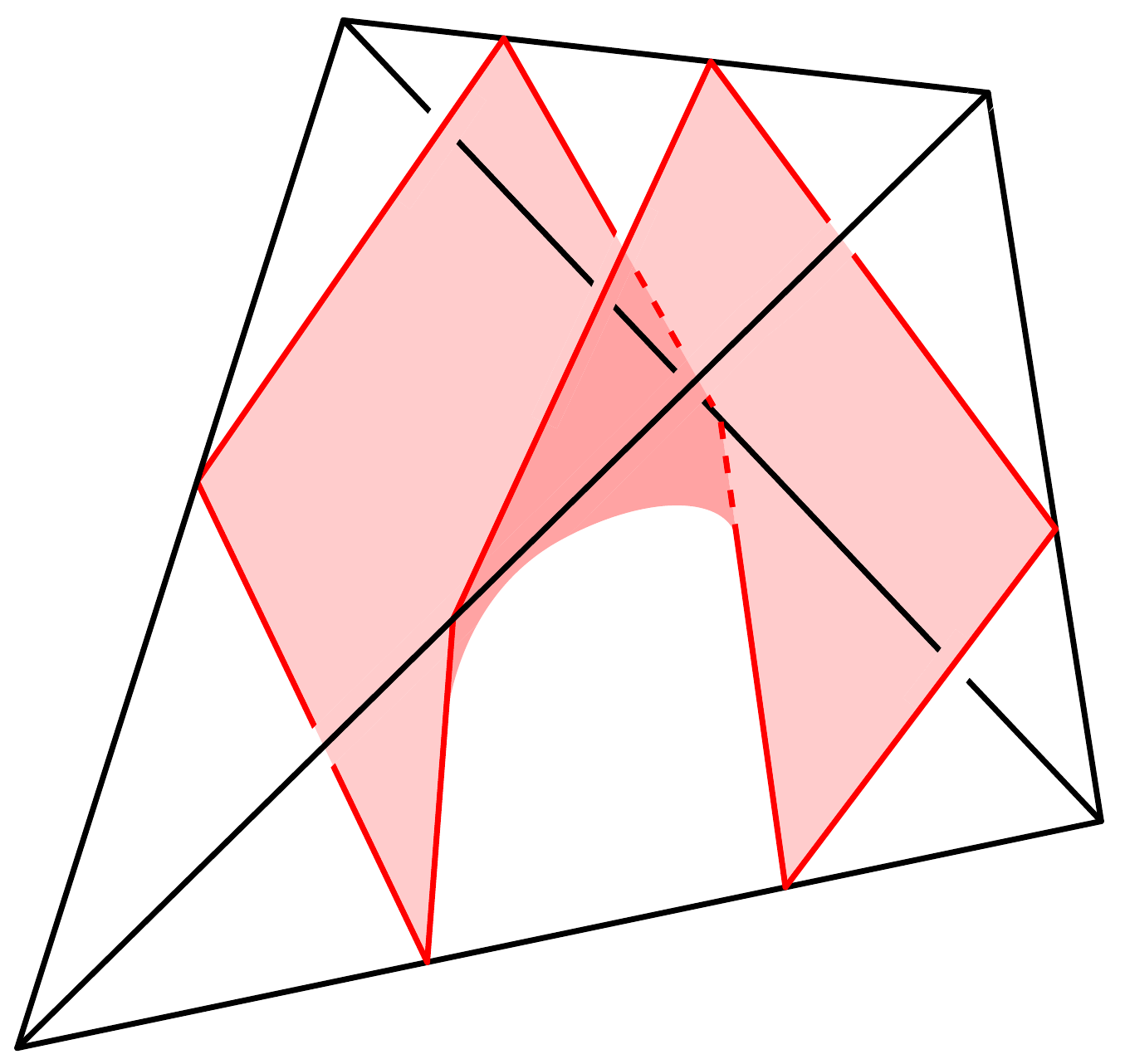}
\caption{A normal octagon in a tetrahedron with a semi-angle structure with angles $\alpha, \beta, \gamma \in [0,\pi]$. The curvature of this octagon is $4\alpha+2\beta+2\gamma - (8-2)\pi = 2\alpha + 2\pi - 6\pi = 2\alpha - 4\pi < 0.$}
\label{normal_octagon}
\end{figure}
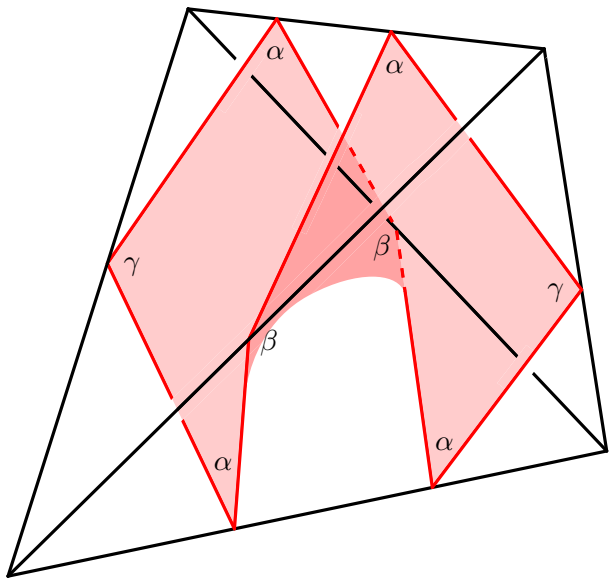

A useful observation (see \cite{KR2}) following from Theorem \ref{semi-angle => 1-efficient} is the following;

\begin{cor}\label{strongly}

Suppose that $\CT$ is an ideal triangulation of an oriented $3$-manifold $M$ with cusps. If $M$ is an-annular and $\CT$ admits a semi-angle structure then $M$ is strongly 1-efficient. 

\end{cor}

\begin{proof}

The key observation is that the semi-angle structure on $\CT$ lifts to a semi-angle structure on the lifted triangulation $\wt \CT$, for any covering space $\wt M$ of $M$. Assume that there is an immersed normal torus $T$ in $M$ which is not a covering of the peripheral torus. If $\wt M$ is chosen as the covering space whose fundamental group corresponds to the image of $\pi_1(T)$, then $T$ lifts to a normal torus $\wt T$ so that the inclusion map induces an onto map $\pi_1(\wt T) \rightarrow \pi_1(\wt M)$. 

We can now use $\wt T$ as a barrier (see \cite{JR}) to produce an embedded normal non-peripheral torus $T^*$, which is either essential and isotopic into a boundary cusp, or bounds a solid torus or cube with knotted hole. (Here the an-annular assumption is used to show that the covering space $\wt M$ is atoroidal). The rest of the argument is exactly the same as in Theorem \ref{semi-angle => 1-efficient}.

\end{proof}

\begin{proof}[Proof of Theorem \ref{index_structure_iff_1-efficient}]
We closely follow Luo--Tillmann \cite{LT}. We use the following version of Farkas' lemma, which is given as Lemma 10 (3) in \cite{LT}:

\begin{lemma}
Let $A$ be a real $K \times L$ matrix, $b\in\R^K$, and $\cdot$ denote the usual Euclidean inner product on $\R^K$. Then $\{x\in\R^L \mid Ax=b, x>0\} \neq \emptyset$ if and only if for all $y\in\R^K$ such that $A^Ty\neq 0$ and $A^Ty\leq 0,$ one has $y\cdot b<0.$
\end{lemma}

For our purposes, $Ax=b$ is the matrix form \eqref{full_matrix_eqn} of the generalised angle structure equations, so $b=(2\pi, \ldots, 2\pi, \pi, \ldots, \pi)^T$. Consider a particular quad-choice $Q$, as in Definition \ref{choice_of_quads}. If there is to be an index structure, then we must be able to find the appropriate generalised angle structure $x$. That is, $x_l>0$ if $l$ corresponds to one of the $Q_n$, and $x_l$ can have any real value if not. We refer to the former as \emph{restricted} variables, and the latter as \emph{unrestricted} variables.

The problem with applying Farkas' lemma directly is that it applies to the set of solutions $\{Ax=b \mid x>0\}$. That is, \emph{all} variables are strictly positive. However, we use a standard trick: for each unrestricted variable $x_l$, introduce a new variable $x_l'$. The
new variable acts precisely like $-x_l,$ so the old $x_l$ can be written in the new coordinates as $x_l - x_l'$. 
This allows both new variables $x_l, x_l' > 0$, making them restricted variables, so that Farkas' lemma can be applied.

The effect that this has on the matrix $A$ is as follows: We get a new column after each unrestricted $x_l$ for $x_l'$,
and the values in the new column are the negatives of the values in the column for $x_l$.

Now we apply Farkas' lemma. We get a solution to our system if and only if for all $y\in\R^K$ such that $A^Ty\neq 0$ and $A^Ty\leq 0,$ we have $y\cdot b<0.$ The transposed matrix $A^T$ has dual variables $(z_1, ..., z_n, w_1, ..., w_t)$, where the $w_i$
correspond to the tetrahedra and the $z_j$ correspond to the edges. The dual system $A^T(z,w)^T \leq 0$ is given by inequalities:
$$w_i + z_j + z_k \leq 0$$
whenever the $i$th tetrahedron contains a quad that faces the edges $j$ and $k$ (which may not be distinct). This holds for all the rows corresponding to the $x_l$, and we get the following for the $x_l'$:
$$-(w_i + z_j + z_k) \leq 0$$
The two of these together imply that $w_i + z_j + z_k = 0$ for the quads
corresponding to unrestricted angles, while $w_i + z_j + z_k \leq 0$ for restricted angles. 
The rest of the argument is the same as in \cite{LT}, as follows.

Kang and Rubinstein \cite{KR1} give a basis of the normal surface solution space  $C(M; \tri)$ which consists of one element for each edge and one element for each tetrahedron of $\CT$. For the edge $e$, the corresponding basis element has each of the quad types in the normal quadrilateral type sequence for $e$ with coefficient $-1$ (or $-2$ if that quad appears twice), and each of the triangle disc types that intersect $e$ with coefficient $+1$. For each tetrahedron $\sigma$, the corresponding basis element has each of the quad types in $\sigma$ with coefficient $-1$, and each of the triangle disc types in $\sigma$ with coefficient $+1$.

If we have a solution to the dual system, then we can form a normal surface solution class $W_{w,z}$ as a sum of tetrahedral and edge basis elements with coefficients given by the $w_i$ and $z_j$ corresponding to their tetrahedra and edges respectively. There is a linear functional $\chi^*$ on $\mathbb{R}^{7n}$ called the \emph{generalised Euler characteristic}, which agrees with the Euler characteristic in the case of an embedded normal surface represented by an element of $C(M; \tri)$. It is shown in \cite{LT} that the generalised Euler characteristic  $\chi^*(W_{w,z})$ is equal to $y\cdot b$, and that the normal quad coordinates of $W_{w,z}$ are given by $-(w_i + z_j + z_k)$. From the above inequalities, we find that the obstruction classes are solutions to the normal surface matching equations with zero quad coordinates for unrestricted angles, non-negative quad coordinates for restricted angles (i.e. the quads specified by the quad-choice $Q$), at least one quad coordinate strictly positive, and generalised Euler characteristic $\chi^*(W_{w,z}) \geq 0$. 

If there are any negative triangle coordinates, we can add vertex linking copies of the boundary tori to the solution until all normal disc coordinates are non-negative. Now, since at most one quad coordinate in each tetrahedron is non-zero, we can in fact realise the normal surface solution class as an embedded normal surface, and so the generalised Euler characteristic is equal to the Euler characteristic. Therefore, an obstruction class to this quad-choice having an associated generalised angle structure is an embedded normal sphere, projective plane, Klein bottle or torus, with the only quads appearing being of the quad types given by the quad-choice. Thus, if the triangulation is 1-efficient, then there can be no such obstruction.

The above argument shows that a 1-efficient triangulation admits an index structure. For the converse, note that if a triangulation is not 1-efficient, then there is an embedded normal sphere, projective plane, Klein bottle or non-vertex linking torus. This must then have at least one non-zero quad coordinate, and since it is embedded, there can be only one non-zero quad coordinate in each tetrahedron. Choosing these quad types in the tetrahedra containing the surface, and arbitrarily choosing quad types in any other tetrahedra, we construct a quad-choice that by the above argument cannot have a suitable generalised angle structure, and so there is no index structure. This completes the proof of Theorem \ref{index_structure_iff_1-efficient}. \end{proof}

\begin{cor}\label{atoroidal=>1-efficient}
Suppose that $M$ is a compact  oriented irreducible 3-manifold with incompressible tori boundary components and no immersed incompressible tori or Klein bottles, except those which are homotopic into the boundary tori. Then $M$ admits an ideal triangulation $\tri$ having an index structure. Moreover if $M$ has no essential annuli (i.e $M$ is an-annular) then for any finite sheeted covering space $\widetilde M$, the lifted triangulation also admits an index structure.  
\end{cor}

\begin{proof}
To construct 1-efficient triangulations, we can use a construction of Lackenby \cite{Lackenby2}. 
He proves that if $M$ is a compact oriented irreducible 3-manifold with incompressible tori boundary components and $M$ has no immersed essential annuli, except those homotopic into the boundary tori, then $M$ admits a taut ideal triangulation $\tri$.
Then by Corollary \ref{strongly}
such triangulations are strongly 1-efficient.  Note that the lift of such a triangulation to any finite sheeted covering space is also 1-efficient. \\

There is a remaining case of small Seifert fibred spaces. For these are precisely the oriented $3$-manifolds with tori boundary components which admit essential annuli, but no embedded incompressible tori which are not homotopic into the boundary components. 

Such examples have 
base orbifold either a disc with two cone points or an annulus with one cone point or  M\"obius band with no cone points. The cone points are the images of the exceptional fibres in the Seifert structure. These manifolds have immersed incompressible tori, but do not have embedded incompressible tori or Klein bottles, except in the case where the base orbifold has orbifold Euler characteristic zero: a disc with two cone points corresponding to exceptional fibres of multiplicity two or orbit surface a M\"obius band with no cone points. This represents two different Seifert fibrations of the same manifold. We exclude this latter case. 

Now to construct a suitable ideal triangulation, note that these Seifert fibred spaces $M$ are bundles over a circle with a punctured surface of negative Euler characteristic as the fibre. To see this, note that $M$ is Seifert fibred over an orientable base orbifold $B$ with $\chi^{orb}(B)<0$. Then $M$ admits a connected horizontal surface $F$ which is orientable with $\chi(F)<0$ since $F$ is an orbifold covering of $B$. (A surface is \emph{horizontal} if it is everywhere transverse to the Seifert fibration.) Since $M$ is orientable it follows that $F$ non-separating, so $M$ fibres over the circle with $F$ as fibre (see, for example, \cite[sections 1.2 and 2.1]{Hat}.) 

After Lemma 6 in \cite{Lackenby2}, it is shown that, starting with any ideal triangulation of the punctured surface $F$, a bundle can be formed as a layered triangulation. This is done by realising a sequence of diagonal flips on the surface triangulation needed to achieve any given monodromy map. Such a triangulation then gives an ideal triangulation with a taut structure. So by Theorem \ref{semi-angle => 1-efficient} these are 1-efficient triangulations and hence admit index structures. \end{proof}

\begin{rmk}\label{small SFS lifted may not have index structure}
The small Seifert fibred spaces from the proof of Corollary \ref{atoroidal=>1-efficient} have finite sheeted coverings with embedded incompressible tori so that the lifted triangulations do not all admit index structures, in contrast with the hyperbolic case. \end{rmk} 

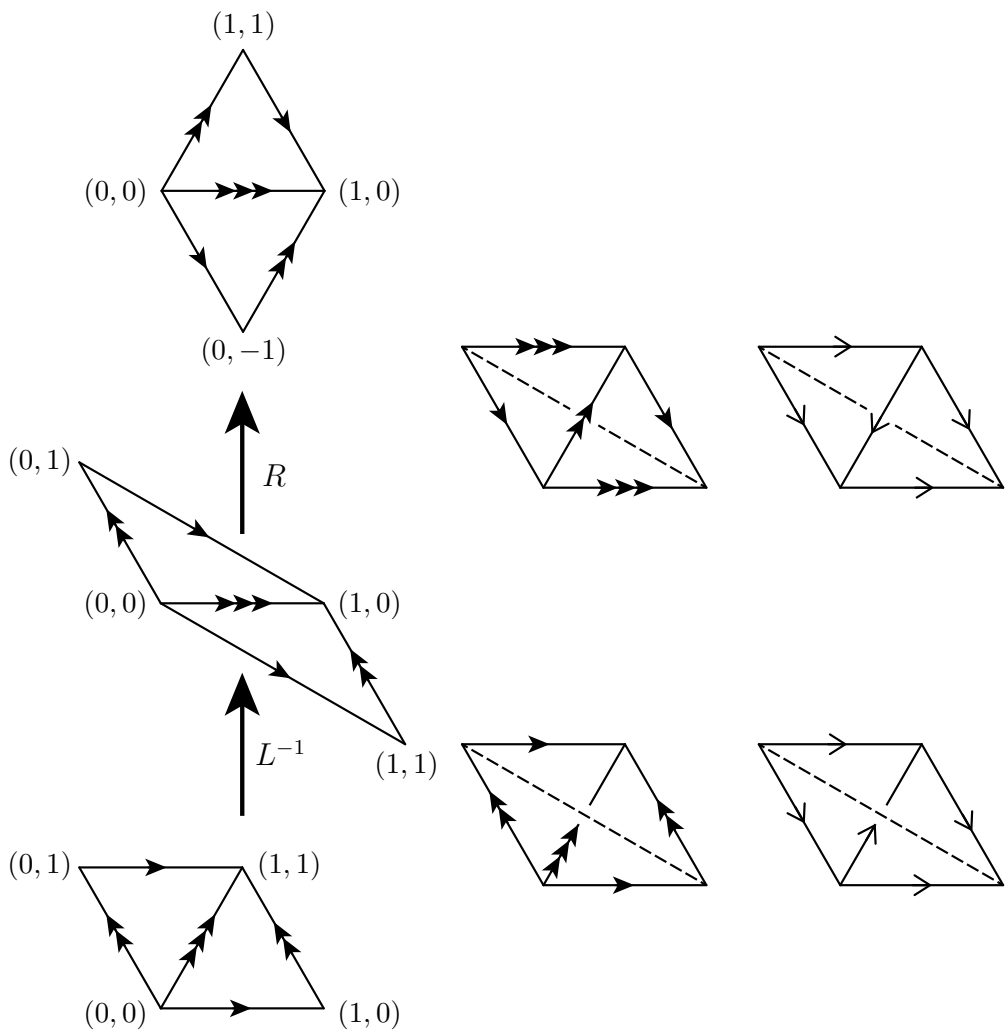
\begin{figure}[htbp]
\centering

\labellist
\pinlabel $L^{-1}$ at 120 140
\pinlabel $R$ at 117 268
\small\hair 2pt
\pinlabel $(0,0)$ at 44 19
\pinlabel $(1,0)$ at 161 19
\pinlabel $(1,1)$ at 124 88
\pinlabel $(0,1)$ at 9 88
\pinlabel $(0,1)$ at 9 274
\pinlabel $(0,0)$ at 44 208
\pinlabel $(1,0)$ at 161 208
\pinlabel $(1,1)$ at 178 134
\pinlabel $(0,-1)$ at 103 325
\pinlabel $(1,1)$ at 103 475
\pinlabel $(0,0)$ at 44 398
\pinlabel $(1,0)$ at 161 398

\endlabellist

\includegraphics[width=0.8\textwidth]{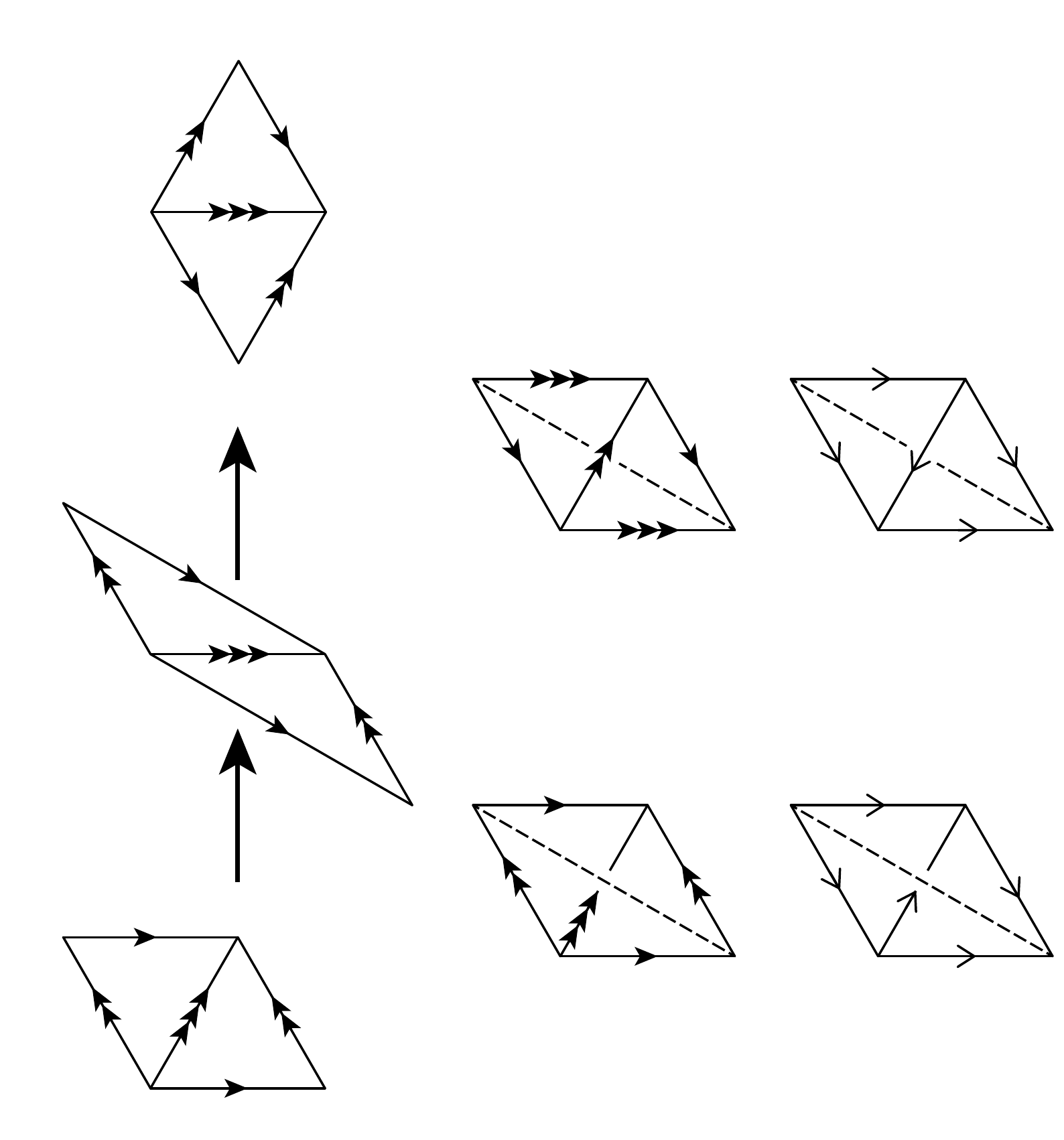}
\caption{A layered triangulation of the complement of the trefoil knot, seen as a punctured torus bundle. On the left, the monodromy is decomposed into generators which act on the punctured torus. The diagrams are shown sheared to highlight the fact that the monodromy has the effect of rotation by $-\pi/3$. The arrows show where edges of a triangulation of the punctured torus map to under the generators. In the middle, we realise each change in the triangulation by layering on a flat tetrahedron. The arrows are shown on the bottom and top of the stack of two tetrahedra to show the gluing. On the right, we see the edges after the identifications induced by gluing the top to the bottom. There are two tetrahedra and two edges in the triangulation, one of degree two (shown with a dashed line) and the other of degree ten.}
\label{trefoil_example}
\end{figure}

\begin{ex}
The trefoil knot complement has an ideal layered triangulation with two tetrahedra and two edges, one of degree $2$ and one of degree $10$. See Figure \ref{trefoil_example}. The complement of the trefoil knot can be seen as a punctured torus bundle with monodromy given by $RL^{-1}$, where 
\[
L=\begin{pmatrix}1&0\\1&1\end{pmatrix} \qquad R=\begin{pmatrix}1&1\\0&1\end{pmatrix}\negthinspace .
\]
Following the caption of Figure \ref{trefoil_example}, we obtain a triangulation of the complement of the trefoil consisting of two tetrahedra. The matrix form of the generalised angle structure equations for this triangulation is
\[
\begin{pmatrix} 
1&1&0&0&0&0\\
1&1&2&2&2&2\\
1&0&1&0&1&0\\
0&1&0&1&0&1
\end{pmatrix} 
\begin{pmatrix}
Z_1 \\ Z_2 \\ Z'_1 \\ Z'_2\\ Z''_1 \\ Z''_2
\end{pmatrix}
=
\begin{pmatrix}
2\pi \\ 2\pi \\ \pi \\ \pi
\end{pmatrix}  \negthinspace.
\]

There is a taut structure given by choosing angles $(\pi, \pi,0,0,0,0)^T$. This assigns the angle $\pi$ to the quad types facing the degree $2$ edge and $0$ to all other angles. This taut structure is compatible with the layering construction. By Theorem \ref{semi-angle => 1-efficient}, this triangulation is 1--efficient. 

It is easy to see that there are no other semi-angle structures for this particular triangulation, because of the degree 2 edge. However, consistent with Theorem \ref{index_structure_iff_1-efficient}, it admits an index structure. To see this, we have to produce a generalised angle structure for each of the $3^2=9$ possible quad-choices. However, by symmetry of the matrix we can reduce this number to three, represented by the following three pairs of conditions that must be satisfied by three generalised angle structures.
\[
(Z_1 > 0, Z_2>0), \qquad (Z_1 > 0, Z'_2>0),  \qquad (Z'_1 > 0, Z'_2>0)
\]
These three representatives are all satisfied by, for example, $(\pi,\pi,x,x,-x,-x)^T$ for any $x>0$.\\

Note that there is a well-known 6-fold cyclic covering 
by the bundle which is a product of a once punctured torus and a circle. This covering is toroidal so we see that there is an index structure on the trefoil knot space but not on this covering space. 
\end{ex}

\begin{figure}[htb]
\centering

\labellist
\small\hair 2pt
\pinlabel $e_1$ at 39 150
\pinlabel $e_2$ at 255 115
\pinlabel $a$ at 278 55
\pinlabel $b$ at 322 160
\pinlabel $q$ at 240 71
\endlabellist

\includegraphics[width=0.45\textwidth]{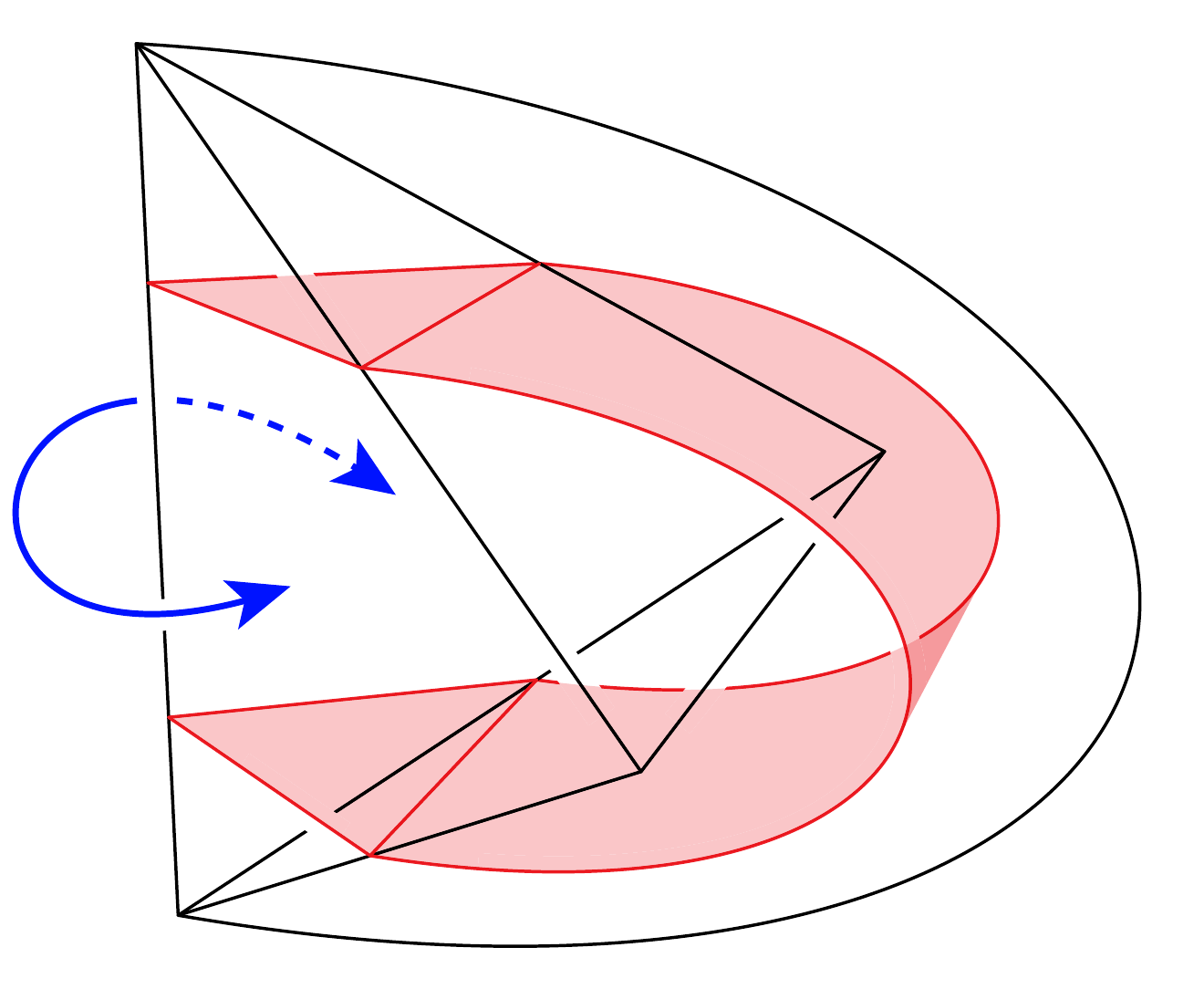}
\caption{Part of a triangulation that does not admit an index structure, and part of the corresponding surface. The edge $e_1$ is degree 1, so the two faces incident to it are identified (indicated by the arrows).}
\label{deg1_deg2_example}
\end{figure}
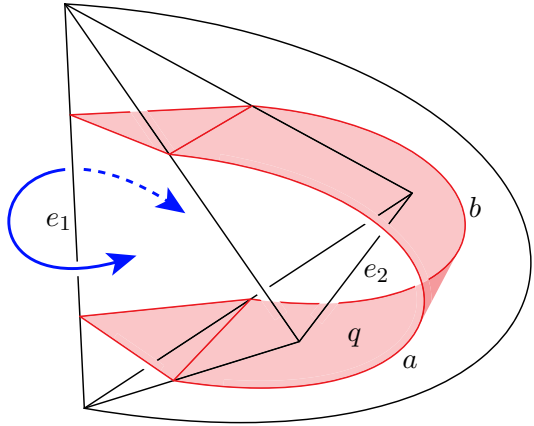

\begin{ex}
We give an example of a subset of a triangulation consisting of two tetrahedra identified in a particular way. Namely, we have a tetrahedron $\sigma_1$ with opposite edges $e_1$ of degree 1 and $e_2$ of degree 2, and another tetrahedron $\sigma_2$ which is the second tetrahedron incident to $e_2$. See Figure \ref{deg1_deg2_example}. If these tetrahedra are part of any ideal triangulation with torus boundary components then that triangulation will not have an index structure, and will have a normal torus that is not vertex-linking, so it is not 1--efficient.

First we show that there is no index structure. Since $e_1$ is degree 1, for any generalised angle structure the angle of $\sigma_1$ at the quad type facing $e_1$ must be $2\pi$. This quad type also faces $e_2$. The angle of the quad type in $\sigma_2$ facing $e_2$ must add to $2\pi$ to give $2\pi$, and so it must be zero. Therefore this angle can never be strictly positive, and so there is no index structure.

Next, we find the corresponding embedded normal torus. It has a single quadrilateral in $\sigma_2$, labelled $q$ in Figure \ref{deg1_deg2_example}. Two of its triangles are in $\sigma_1$, also shown. When the two identified faces of $\sigma_1$ are glued to each other, the boundary of the shown surface consists of the two arcs labelled $a$ and $b$, on two of the boundary faces of $\sigma_2$. Now consider the vertex--linking normal torus $T$, given by the link of the vertex at which the endpoints of $e_2$ meet. We complete our surface into an embedded normal torus by deleting from $T$ the normal triangles in $\sigma_1$ and $\sigma_2$ at the endpoints of $e_2$, and gluing the resulting boundary arcs to $a$ and $b$. The resulting surface is boundary parallel and so is a torus, but is obviously not vertex--linking since it contains a quadrilateral.
\end{ex}

%%%%%%%%%%%%%%%%%%%%%%%%%%%%%%%%%%%%%%%%%%%%%%%%%%%%%%%%%%%%%%%%%%%%%%%%%%%%
%%%%%%%%%%%%%%%%%%%%%%%%%%%%%%%%%%%%%%%%%%%%%%%%%%%%%%%%%%%%%%%%%%%%%%%%%%%%

\section{A review of the index of an ideal triangulation}
\label{sec.review}

\subsection{The tetrahedron index and its properties}

In this section we review the definition and the identities satisfied by
the tetrahedron index of \cite{DGG2}. For a detailed discussion, see 
\cite{Ga:index}.

The building block of the index $I_{\CT}$ of an ideal triangulation $\CT$
is the {\em tetrahedron index $\ID(m,e)(q) \in \BZ[[q^{1/2}]]$} defined by
\be
\lbl{eq.ID}
\ID(m,e)=\sum_{n=(-e)_+}^\infty (-1)^n \frac{q^{\frac{1}{2}n(n+1)
-\left(n+\frac{1}{2}e\right)m}}{(q)_n(q)_{n+e}}
\ee
where
$$
e_+=\max\{0,e\}
$$
and $(q)_n=\prod_{i=1}^n (1-q^i)$. If we wish, we can sum in the above 
equation over the integers, with the understanding that $1/(q)_n=0$ for $n < 0$.

The tetrahedron index satisfies the following {\em linear} recursion relations
\begin{subequations} 
\be
\lbl{eq.rec1}
 q^{\frac{e}{2}} \ID(m + 1, e) + q^{-\frac{m}{2}} \ID(m, e + 1) - \ID(m, e)=0
\ee
\be
\lbl{eq.rec2} 
 q^{\frac{e}{2}} \ID(m - 1, e) + q^{-\frac{m}{2}} \ID(m, e - 1) - \ID(m, e)=0
\ee
\end{subequations}
and
\begin{subequations} 
\be
\lbl{eq.rec1a}
\ID(m, e+1) + (q^{e + \frac{m}{2}} - q^{-\frac{m}{2}} - q^{\frac{m}{2}}) 
\ID(m, e) + \ID(m, e - 1)  = 0
\ee
\be
\lbl{eq.rec2a} 
\ID(m+1, e) + (q^{- \frac{e}{2} - m} - q^{-\frac{e}{2}} - q^{\frac{e}{2}}) 
\ID(m, e) +  \ID(m - 1, e) =   0 
\ee
\end{subequations}
and the {\em duality} identity
\be
\lbl{eq.Z2}
\ID(m,e)(q)=\ID(-e,-m)
\ee
and the {\em triality} identity
\be
\lbl{eq.Z3}
\ID(m,e)(q)=(-q^{\frac{1}{2}})^{-e}\ID(e,-e-m)(q) \,
=(-q^{\frac{1}{2}})^m \ID(-e-m,m)(q)
\ee
and the {\em pentagon} identity
\be
\lbl{eq.pentagon}
\ID(m_1-e_2,e_1)\ID(m_2-e_1,e_2)=\sum_{e_3 \in \BZ}
q^{e_3} \ID(m_1,e_1+e_3)\ID(m_2,e_2+e_3)\ID(m_1+m_2,e_3)\,,
\ee
and the {\em quadratic} identity
\be
\lbl{eq.IDquadratic}
\sum_{e \in \BZ} \ID(m,e)\ID(m,e+c)q^e =\delta_{c,0}=
\left\{\begin{array}{l}
 1 \text{ if } c=0\\  
 0 \text{ if } c\neq0\\  
 \end{array}\right.
\ee
The above relations are valid for all integers $m,e,m_i,e_i,c$. 

\subsection{The degree of the tetrahedron index}
\lbl{sub.degI}

The (minimum) degree $\d(m,e)$ with respect to $q$ of $\ID(m,e)$ is given by

\be
\lbl{eq.degID}
\d(m,e) = \frac{1}{2}\left( m_+(m+e)_+ + (-m)_+ e_+ + (-e)_+(-e-m)_+
+\max\{0,m,-e\} \right) 
\ee
It follows that $\d(m,e)$ is a piecewise quadratic polynomial shown in Figure \ref{degI}.
\begin{figure}[htb]
\centering

\labellist
\pinlabel $e=0$ at 65 145
\pinlabel $m=0$ at 235 295
\pinlabel $e+m=0$ at 405 35
\small\hair 2pt
\pinlabel $-\frac{em}{2}$ at 185 192
\pinlabel $\frac{m(e+m)}{2}+\frac{m}{2}$ at 335 160
\pinlabel $\frac{e(e+m)}{2}-\frac{e}{2}$ at 190 80
\endlabellist

\includegraphics[width=0.4\textwidth]{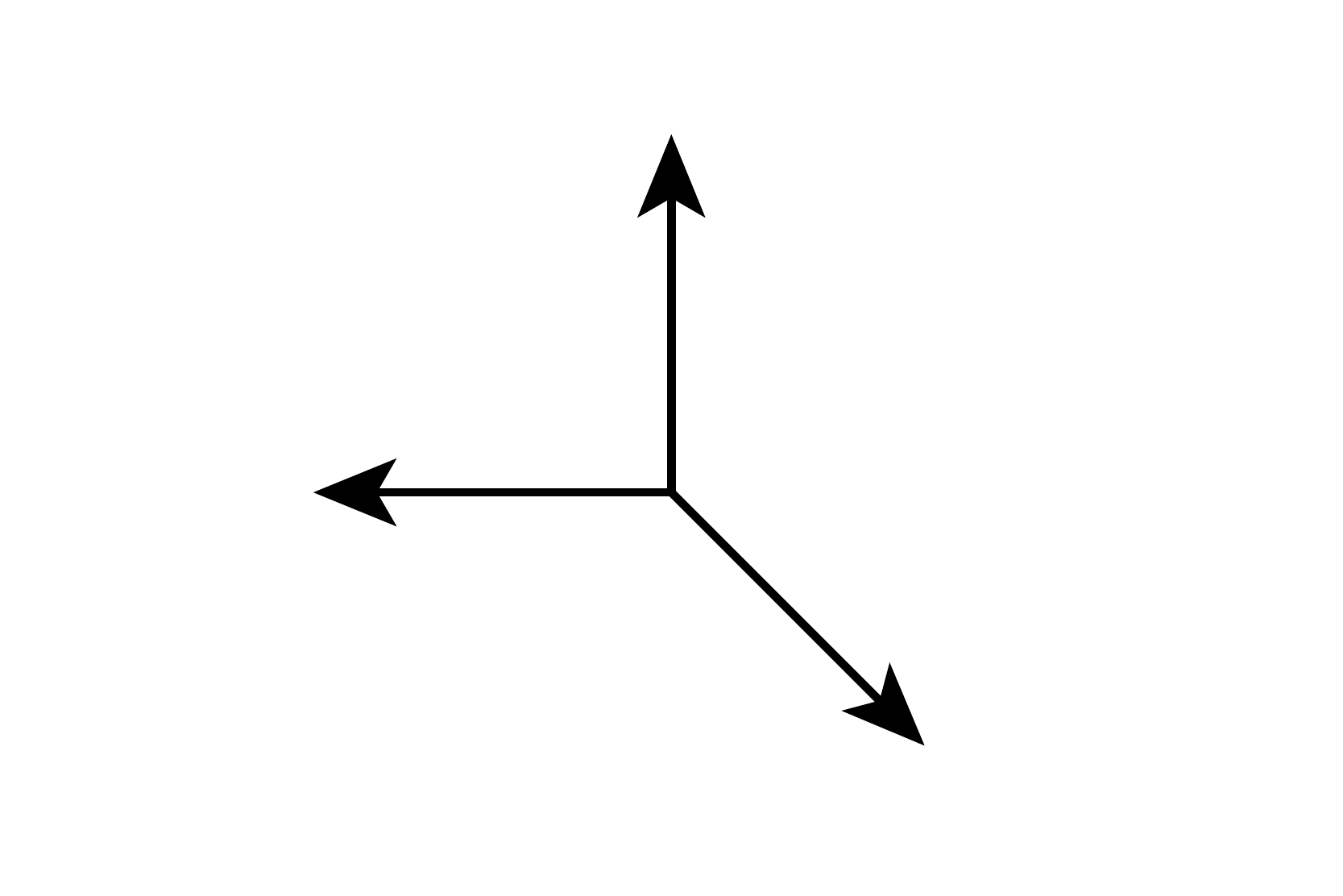}
\caption{The degree of the tetrahedron index $\ID(m,e)$. Here the positive $m$ axis is to the right and the positive $e$ axis is upwards.}
\label{degI}
\end{figure}
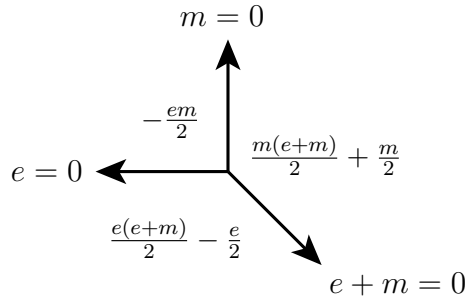

The regions of polynomiality of $\d(m,e)$ give a fan in $\BR^2$ with rays
spanned by the vectors $(0,1)$, $(-1,0)$ and $(1,-1)$. An important feature
of $\d$ is that it is a convex function on rays.

\subsection{Angle structure equations}
\label{sub.ge}

Recall the equations for a generalised angle structure as given in Definition \ref{matrix equations}.  In this section, we will refer to the angle variables within the $i$th tetrahedron, $\alpha(q_i) ,\alpha(q_i'),\alpha(q_i'')$, as  $Z_i,Z_i',Z_i''$ respectively.

We can view a quad-choice 
$Q$ for $\CT$ (as in Definition \ref{choice_of_quads}) as a choice of pair of opposite edges 
at each tetrahedron $\Delta_i$ for $i=1,\dots,N$. The quad-choice $Q$ can be used to eliminate
one of the three variables $Z_i,Z'_i,Z''_i$ at each tetrahedron using the
relation $Z_i+Z'_i+Z''_i=\pi$. Doing so, equations \eqref{eq.angle1} take the form
$$
\mb A \, Z + \mb B \, Z'' = \pi \boldsymbol{\nu} \,,
$$ 
where $\boldsymbol{\nu} \in \BZ^{N}$.
(For example, if we eliminate the variables $Z'_i$ then
$\mb A = \overline{\mb A}-\overline{\mb B}$, $\mb B = \overline{\mb C}-\overline{\mb B}$ 
and $\boldsymbol{\nu} = 2 (1,\dots,1)^T - \overline{\mb B} (1,\dots,1)^T$. 

The matrices $(\mb A|\mb B)$ have some key 
{\em symplectic properties}, discovered by Neumann-Zagier when $M$ is a 
hyperbolic 3-manifold (and $\CT$ is well-adapted to the hyperbolic stucture) 
\cite{NZ}, and later generalised to the case of arbitrary 3-manifolds 
in \cite{Neumann-combi}. Neumann-Zagier show that the rank of 
$(\mb A|\mb  B)$ is $N-r$, where $r$ is the number of boundary 
components of $M$; all assumed to be tori. 

\subsection{Peripheral equations}
\lbl{sub.peripheral}

Assume first, for simplicity, that $\bd M$ consists of a single torus, and let  $\varpi$ be an oriented  simple closed curve  in $\bd M$ that is in {\em normal position} with respect to the induced triangulation $\tri_\bd$ of $\bd M$.
Let 
\be
(\overline{a}_{\varpi}|\overline{b}_{\varpi}|\overline{c}_{\varpi})=
(\bar a_{\varpi,1} \ldots, \bar a_{\varpi,N} \mid  \bar b_{\varpi,1}, \ldots, \bar b_{\varpi,N} \mid \bar c_{\varpi,1}, \ldots, \bar c_{\varpi,N})
\ee
denote the vector in $\BZ^{3N}$
computed as follows.
See Figure \ref{truncated_tetra2}.

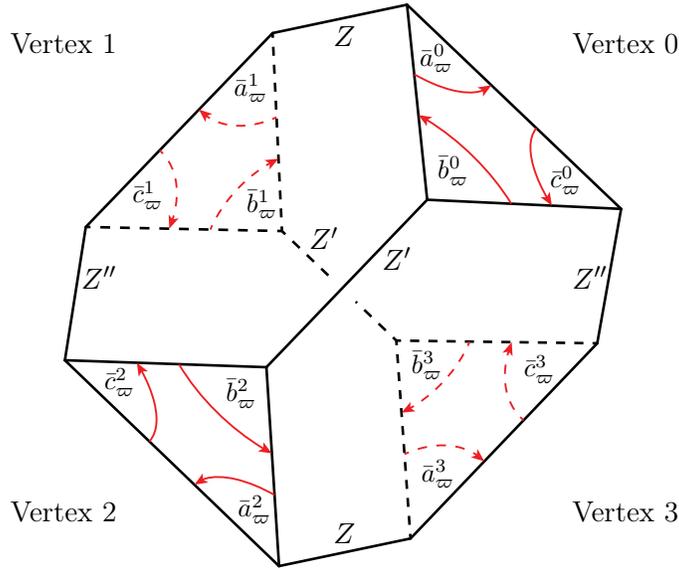
\begin{figure}[htbp]
\centering

\labellist
\small\hair 2pt
\pinlabel {Vertex 0} at 300 280
\pinlabel {Vertex 1} at 0 280
\pinlabel {Vertex 2} at 0 30
\pinlabel {Vertex 3} at 300 30
\pinlabel $Z$ at 150 18
\pinlabel $Z$ at 150 284
\pinlabel $Z'$ at 139 176
\pinlabel $Z'$ at 178 165
\pinlabel $Z''$ at 19 155
\pinlabel $Z''$ at 281 155

\footnotesize

\pinlabel $\bar b^0_\varpi$ at 208 212
\pinlabel $\bar a^0_\varpi$ at 199 272
\pinlabel $\bar c^0_\varpi$ at 268 207

\pinlabel $\bar c^3_\varpi$ at 254 106
\pinlabel $\bar b^3_\varpi$ at 194 108
\pinlabel $\bar a^3_\varpi$ at 200 50

\pinlabel $\bar c^2_\varpi$ at 30 99
\pinlabel $\bar b^2_\varpi$ at 95 93
\pinlabel $\bar a^2_\varpi$ at 102 31

\pinlabel $\bar a^1_\varpi$ at 100 256
\pinlabel $\bar b^1_\varpi$ at 106 194
\pinlabel $\bar c^1_\varpi$ at 45 199

\endlabellist

\includegraphics[width=0.45\textwidth]{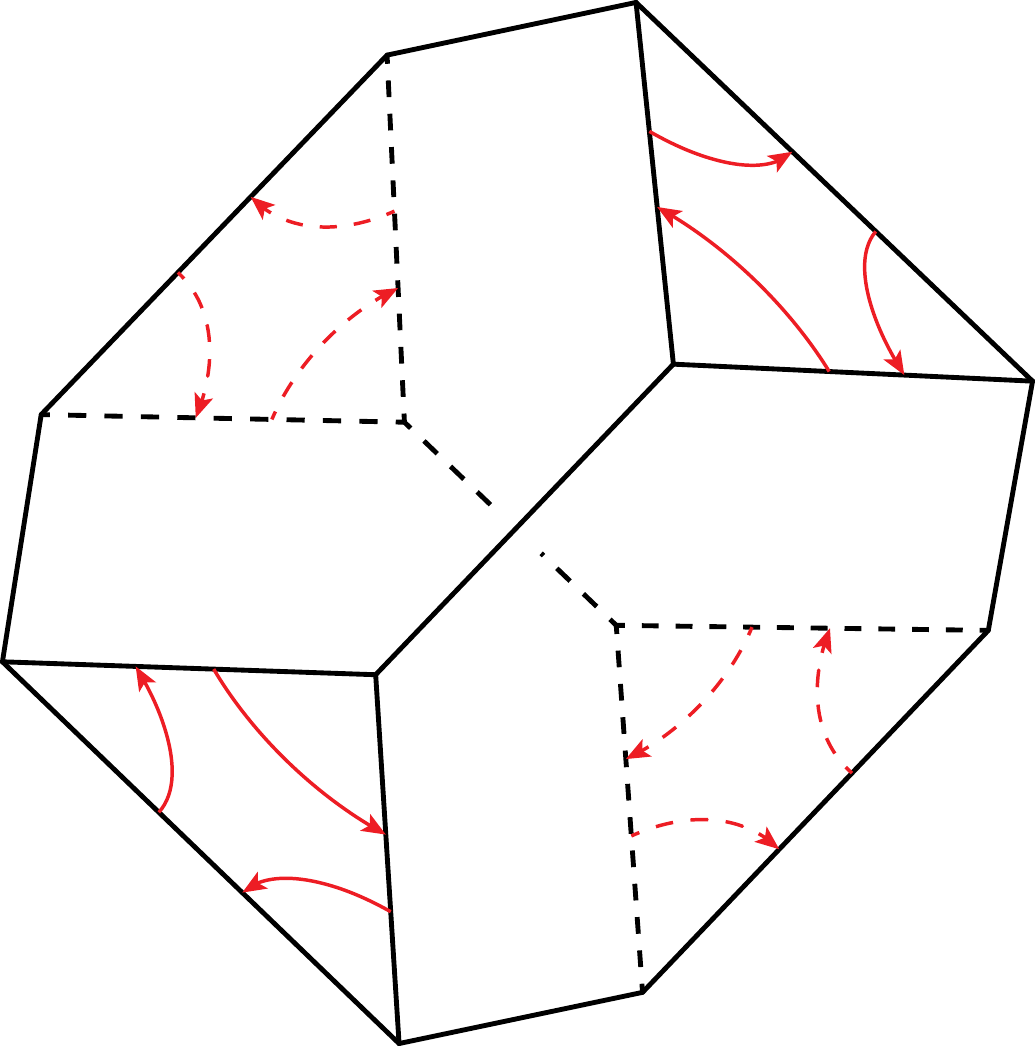}
\caption{Each term of the ``turning number'' vector $(\overline{a}_{\varpi}|\overline{b}_{\varpi}|\overline{c}_{\varpi})$ is calculated as a sum of the signed number of times the curve $\varpi$ turns anticlockwise around the corners of the triangular ends of the truncated tetrahedra. Edges and arcs on the back side of the tetrahedron are drawn with dashed lines.}
\label{truncated_tetra2}
\end{figure}

The term $\bar a_\varpi^l$ counts the signed number of normal arcs of $\varpi$ that turn anticlockwise around the corner of the truncated triangle associated to the variable $Z$, at vertex number $l$ of this tetrahedron. The entry in the vector $\overline{a}_{\varpi}$ for this tetrahedron is $\sum_{l=0}^3 \bar a_\varpi^l$, and similarly for the $\overline{b}_{\varpi}$ and $\overline{c}_{\varpi}$ terms. We suppress the vertex number superscripts from now on, since this data is implied by the location of the labels in the figures.

If we eliminate $Z'_j$ using $Z_j+Z'_j+Z''_j = \pi$, then we obtain the vector in $\Z^{2N}$
\be
(a_{\varpi}|b_{\varpi})=(\overline{a}_{\varpi}-\overline{b}_{\varpi}|
\overline{c}_{\varpi}-\overline{b}_{\varpi})
\ee
as well as the scalar 
\be\nu_{\varpi}=-\sum_{i=1}^N \overline{b}_{\varpi,j}.\ee
Similarly, we can define ``turning number'' vectors $(\overline{a}_{\varpi}|\overline{b}_{\varpi}|\overline{c}_{\varpi}) \in \Z^{3N}$
and $(a_{\varpi}|b_{\varpi}) \in \Z^{2N}$ for any oriented multi-curve $\varpi$ on $\bd M$ (i.e. a disjoint union of oriented simple closed normal curves on $\bd M$).

More generally, suppose that $M$ is a 3-manifold whose boundary $\bd M$ consists of $r \ge 1$ tori
$T_1, \ldots, T_r$. Let $\varpi=(\varpi_1, \ldots, \varpi_r)$ where $\varpi_i$ is an oriented multi-curve on $T_i$ for each $h=1, \ldots , r$.
Then we will use the notation 
\begin{equation}
\lbl{eq.nupi}
(\overline{a}_{\varpi}|\overline{b}_{\varpi}|\overline{c}_{\varpi}) = \sum_{h=1}^r 
(\overline{a}_{\varpi_h}|\overline{b}_{\varpi_h}|\overline{c}_{\varpi_h}),~~
(a_{\varpi}|b_{\varpi})= \sum_{h=1}^r  (a_{\varpi_h}|b_{\varpi_h}), \text{ and }
\nu_{\varpi} = \sum_{h=1}^r \nu_{\varpi_h}.
\end{equation}

\begin{rmk} \label{edge_eqn_as_holonomy}

Suppose that $\varpi=C_i$ is a small linking circle on $\bd M$ around 
one of the two vertices at the ends of the $i$th edge, with $C_i$ oriented 
anticlockwise as viewed from a cusp of $M$. Then 
$$(\overline{a}_{\varpi}|\overline{b}_{\varpi}|\overline{c}_{\varpi}) = 
(\bar a_{i1} \ldots, \bar a_{iN} \mid  \bar b_{i1}, \ldots, \bar b_{iN} \mid \bar c_{i1}, \ldots, \bar c_{iN})$$
 gives the coefficients of the $i$th edge equation as a special case of this construction. 

\end{rmk}

\subsection{The index of an ideal triangulation}
\lbl{sub.defn.index}

Suppose that $M$ is a 3-manifold whose boundary $\bd M$ consists of $r \ge 1$ tori
$T_1, \ldots, T_r$, and let $\CT$ be an ideal triangulation of $M$.
 Let $\varpi=(\varpi_1, \ldots, \varpi_r)$ be a collection of oriented peripheral curves
as above. 
By Theorem \ref{pick_edges}, proved below, 
we can order the edges of $\tri$ so that the first $N-r$ rows of the Neumann-Zagier matrix $(\mb A \mid \mb B)$
form an {\em integer basis} for its integer row space (i.e. the $\Z$-module  of all linear combinations
of its rows with integer coefficients). Then we define
 \be
\lbl{eq.indT}
I_{\CT}(\varpi)(q)=\sum_{\kk \in \BZ^{N-r}\subset \Z^N}
(-q^{\frac{1}{2}})^{\kk \cdot \bnu  + \nu_{\varpi}} 
\prod_{j=1}^N \ID(-b_{\varpi,j} - \kk \cdot \boldsymbol{b}_j  , 
a_{\varpi,j} +\kk \cdot \boldsymbol{a}_j  ) \,.
\ee
where $\boldsymbol{a}_j$ and $\boldsymbol{b}_j$ for $j=1,\dots, N$ denote the columns of $\mb A$ and
$\mb B$, 
and $$\Z^{N-r} = \left\{  (k_1, \ldots, k_N) \in \Z^N : k_j=0 \text{ for } j > N-r \right\} .$$

It can be checked that this definition is independent of the quad choice involved in 
forming $(\mb A \mid \mb B)$; see \eqref{JD_index_eq}.  
It is also independent of the choice of $N-r$ edges used to produce an integer basis for the 
integer row space of the Neumann-Zagier matrix, by Remark \ref{basis_indept}.
In the case of a 1-cusped manifold $M$, any $N-1$ edges can be used; in other words we could replace the domain of summation $\Z^{N-1}$ by any of the coordinate hyperplanes   $\left\{ (k_1, \ldots, k_N) \in \Z^N : k_s=0 \right\}$ with $s \in \{1, \ldots, N \}$. In general, we choose a set $\mathcal B$ of $N-r$ \emph{basic edges} whose corresponding rows we sum over, for example by using Theorem \ref{pick_edges}. Equivalently, we choose the complementary set $\mathcal X$ of \emph{excluded edges}.
 
Theorem \ref{index_isotopy_invariance}
below shows that
the index is unchanged by an isotopy of $\varpi$ so only depends on the homology class
$$[\varpi]=\left[ \sum \varpi_i\right]  \in H_1(\bd M; \Z) = \bigoplus_{i=1}^N  H_1(T_i;\Z).$$  
So the index
gives a well-defined function
$$
I_{\CT}: H_1(\pt M;\BZ) \longto \BZ((q^{1/2})) \text{ where } I_{\CT}([\varpi])=I_{\CT}(\varpi).
$$

If $M$ is a 1-cusped manifold $M$, and
$\mu$ and $\lambda$ in $H_1(\pt M;\Z)$ are a fixed oriented meridian and longitude on $\pt M$
(a canonical choice exists when $M$ is the complement of an oriented knot in $S^3$).
Then we can write 
\begin{equation}
\lbl{eq.varpiem}
[\varpi]=-m \lambda + e \mu
\end{equation} 
for integers $e,m$. 
The naming of  the integers $e$ and $m$ (electric and magnetic charge) 
and the above choice of signs was chosen to make our index compatible with
the definition of \cite{DGG2} and \cite{Ga:index}.

\subsection{Choice of edges in the summation for index}

Let $M$ be an 
orientable 3-manifold with $r \ge 1$ torus cusps and let $\CT$
be an ideal triangulation of $M$ with $N$ tetrahedra and, hence,  $N$ edges which we denote $e_1, \ldots, e_N$.
Let $G$ be 
the 1-skeleton $\CT^{(1)}$ of $\CT$ together with one (ideal) vertex for each cusp of $M$.
Note that $G$ has $r$ vertices and $N$ edges, 
and may contain loops (i.e. edges with both ends at a single vertex)  
or multiple edges between the same two vertices.
The incidence matrix $C = (c_{hi})$ for $G$  is an $r \times N$ matrix whose 
$(h,i)$ entry gives the number of ends of edge $i$ on cusp $h$. Note that
each $c_{hi} \in \{0,1,2\}$ and the sum of entries is $2$ in each column of $C$.
Let $E(e_i)=E_i  \subset  \Z^{2N}$
be the edge equation coefficients
corresponding to edge $e_i$ in $\CT$, and let
\be
\Lambda = \left\{ \sum_{\kk \in \Z^N} k_i E_i \right\} \subset  \Z^{2N}
\ee
be the lattice of all integer linear combinations of these. In other words,
$E_i$ is the $i$th row of the Neumann-Zagier matrix $(\mb A \mid \mb B)$, and
$\Lambda$ is the integer row space of this matrix.

From the work of Neumann and Zagier (see \cite{NZ} and \cite[Thm 4.1]{Neumann-combi}), 
the lattice $\Lambda$
has rank $N-r$ and the  matrix $C$ gives the linear relations between the 
edge equation coefficients $E_i \in  \Z^{2N}$. More precisely,
\be
\label{cusp_relations}
\sum_i c_{hi} E_i =0 \text{ for all } h=1, \ldots, r
\ee
and any other linear relation between the $E_i$ arises from
a real linear combination of the rows of $C$.

\begin{defn}\label{maximal tree with cycle}
A subset of the edges of a graph $\Gamma$ is a \emph{maximal tree with 1- or 3-cycle} in $\Gamma$ if (together with the vertices) it consists of any maximal tree $T$  
together with one additional edge that either (1) is a loop at one vertex, 
or (2) forms a 3-cycle together with two edges in $T$.
\end{defn}

\begin{thm} \label{pick_edges}
There exists an integer basis for $\Lambda$ consisting of $N-r$ 
of the edge equation coefficients $E_1, \ldots, E_n$.  In fact, we can choose
such a basis by omitting $r$ edge equations corresponding to a maximal tree with 1- or 3-cycle in $G$.
\end{thm}

\begin{rmk}
In other words, we can choose any maximal tree with 1- or 3- cycle for our set $\mathcal X$ of excluded edges, and hence choose the remaining edges as our set $\mathcal B$ of basic edges.
\end{rmk}

This result and its proof were inspired by Jeff Weeks' argument in
\cite[pp. 35--36]{Weeks:thesis}.

\begin{proof}  First we show that we can find a maximal tree with 1- or 3-cycle.  
If there exists a loop in $G$ we use this loop together with any maximal tree.
If not, any face of the triangulation has its ideal vertices on 3 distinct cusps.
Pick two edges of this face and  extend these to a maximal tree $T \subset G$. Adding
the third edge of the face gives the desired subgraph.

Now let $S$ be a maximal tree with 1- or 3-cycle. Next we show that the $N-r$ equations $E(e)$ corresponding to the edges $e \notin S$ give 
an integer basis for $\Lambda$.  We show that for each $s\in S,$ the equation $E(s)$ can be written as an integer linear combination of
the equations $E(e)$ with $e \notin S$. Given this, the $N-r$ equations $E(e)$ with $e \notin S$ form an integer spanning set for $\Lambda$. The work of
Neumann and Zagier (\cite{NZ} and \cite[Thm 4.1]{Neumann-combi}), implies that these equations are also linearly independent, hence form an integer basis for $\Lambda$, and we are done.

So, we have to show that every $E(s)$ can be written as an integer linear combination of the $E(e)$ for $e\notin S$. To organise the construction, we use the following sequence of \emph{decorated graphs}. At each step we have a graph $G_k$ whose edges are labelled by names of  edges of $G$. 
We decorate each end of each edge of $G_k$ with a sign. Each vertex $v$ of $G_k$ is then incident to a set of ends of edges with signs. We list the names of the edges, together with the sign associated to this end: $\{ (e_{i_v(1)}, \epsilon_v(1)), (e_{i_v(2)}, \epsilon_v(2)), \ldots, (e_{i_v(d)}, \epsilon_v(d))\}$. Here $d$ is the degree of the vertex $v$. To this vertex we associate the equation
$$R_k(v) = \sum_{l=1}^d \epsilon_v(l) E(e_{i_v(l)})=0.$$

For each $G_k$ we have a subset $S_k$ of the edges of $G_k$ which is a maximal tree with 1-or 3-cycle in $G_k$. We set $G_0=G$ and $S_0=S$, with all signs set to $+$. Note that the equations associated to the vertices of $G_0$ are then the same as those given by  \eqref{cusp_relations}. 

We obtain the graph and edge subset $(G_{k+1},S_{k+1})$ from $(G_{k},S_{k})$ as follows. We arbitrarily choose a vertex $v$ of $G_k$ that has only one end of one edge $s$ of $S_k$ incident.  If there are no such vertices then the sequence ends at $(G_k,S_k)$. Let $w$ be the other end of $s$, which by assumption is distinct from $v$. The graph $G_{k+1}$ is the result of collapsing the edge $s$ of $G_k$; the two ends of $s$, $v$ and $w$, are identified in $G_{k+1}$. We label the edges of $G_{k+1}$ with the same names as in $G_k$ and set $S_{k+1}=S_k \setminus \{s\}.$ All of the signs decorating $G_{k+1}$ are the same as in $G_k$, except that the ends of edges that were incident to $v$ have their signs flipped. See Figure \ref{collapse_edge}.

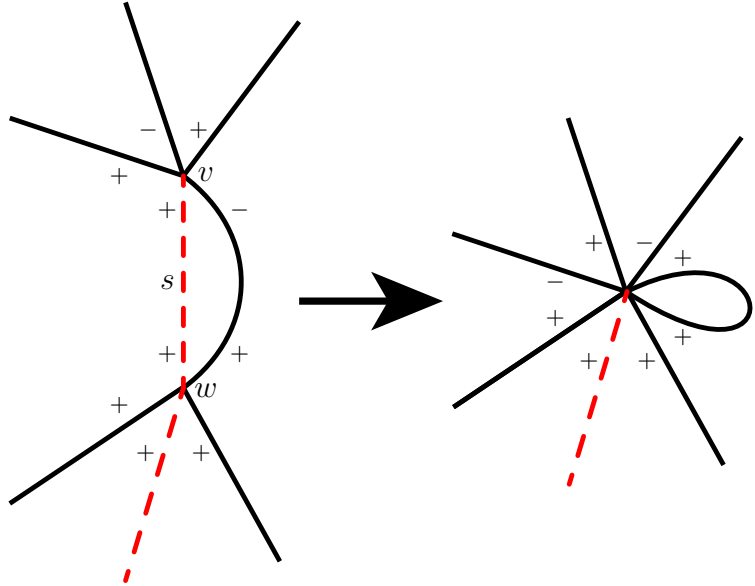
\begin{figure}[htbp]
\centering

\labellist
\pinlabel $v$ at 82 170
\pinlabel $w$ at 82 80
\pinlabel $s$ at 66 125
\scriptsize\hair 2pt
\pinlabel $+$ at 66 155
\pinlabel $+$ at 46 169
\pinlabel $-$ at 58 188
\pinlabel $+$ at 79 188
\pinlabel $-$ at 96 155

\pinlabel $+$ at 66 95
\pinlabel $+$ at 96 95
\pinlabel $+$ at 46 74
\pinlabel $+$ at 57 54
\pinlabel $+$ at 80 54

\pinlabel $+$ at 227 110
\pinlabel $+$ at 241 92
\pinlabel $+$ at 265 92
\pinlabel $+$ at 280 102
\pinlabel $+$ at 280 135
\pinlabel $-$ at 264 141
\pinlabel $+$ at 243 141
\pinlabel $-$ at 227 125

\endlabellist

\includegraphics[width=0.6\textwidth]{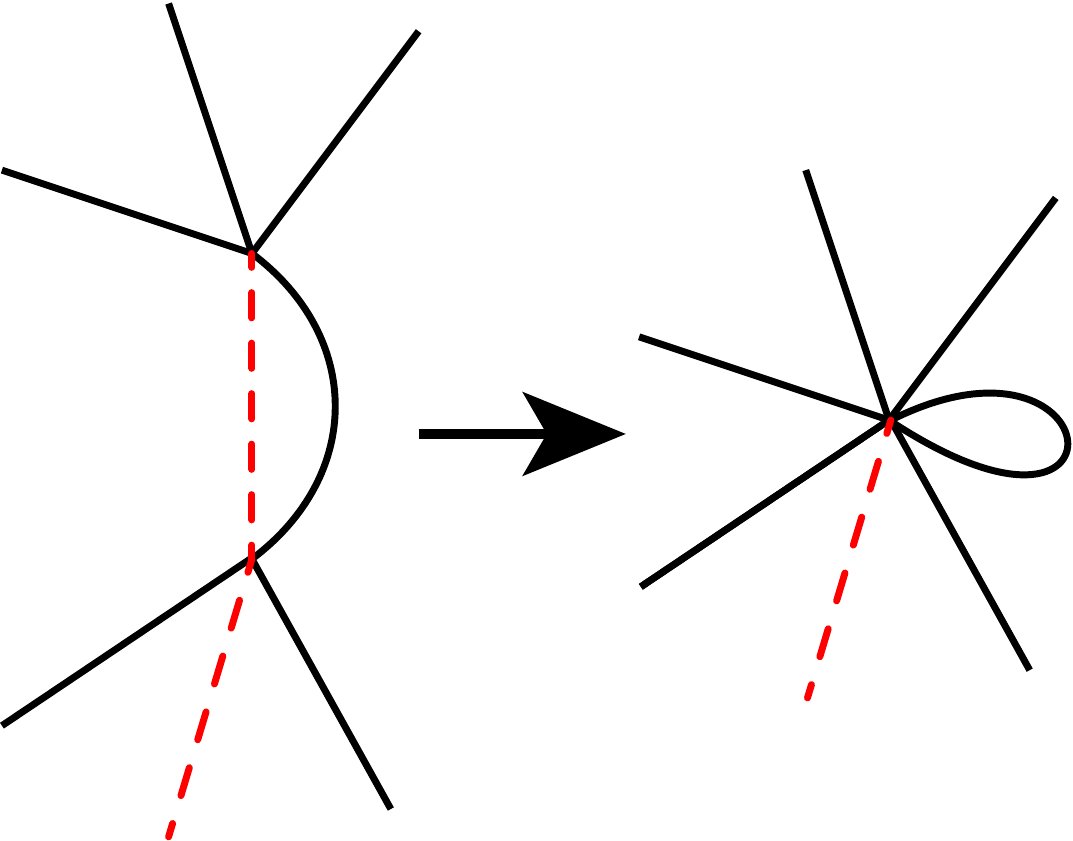}
\caption{Collapsing the edge $s$ flips the signs on the ends of the edges incident to $v$. Edges in $S$ are drawn dashed.}
\label{collapse_edge}
\end{figure}

Note that at each step of the sequence, both ends of each element of $S_k$ have a $+$ sign, since they do in $G_0$ and we never collapse an edge from a vertex with more than one incident edge end in $S$.
Consider the equations associated to the vertices of $G_k$ and $G_{k+1}$. We have

$$R_k(v) = +E(s) + \sum_{l, e_{i_v(l)} \neq s} \epsilon_v(l) E(e_{i_v(l)})=0, \quad R_k(w) = +E(s) + \sum_{m, e_{i_w(m)} \neq s} \epsilon_w(m) E(e_{i_w(m)})=0.$$

If we use $R_k(v)$ to solve for $E(s)$ we get 

$$E(s) = \sum_{l, e_{i_v(l)} \neq s} -\epsilon_v(l) E(e_{i_v(l)}).$$

Substituting this into $R_k(w)$ gives

$$ \sum_{l, e_{i_v(l)} \neq s} -\epsilon_v(l) E(e_{i_v(l)}) + \sum_{m, e_{i_w(m)} \neq s} \epsilon_w(m) E(e_{i_w(m)})=0.$$

This is the equation associated to the vertex of $G_{k+1}$ formed by the identification of $v$ with $w$. 
Thus, the sequence of graphs gives an expression for $E(s)$ for each edge $s\in S$ which is removed. 

This expression is an integer linear combination of the $E(e)$ for $e \notin S$. The sequence ends, at $G_K$ say. By construction $G_K$ has no vertices for which only one end of an edge of $S$ is incident. If we are in case (1) of Definition \ref{maximal tree with cycle} then $G_K$ has one vertex, $S_K$ has one edge and $G_K$ looks like Figure \ref{collapse_edge_endgame} (left). If we are in case (2) then $G_K$ has three vertices, $S_K$ has three edges, and $G_K$ looks like Figure \ref{collapse_edge_endgame} (right). 

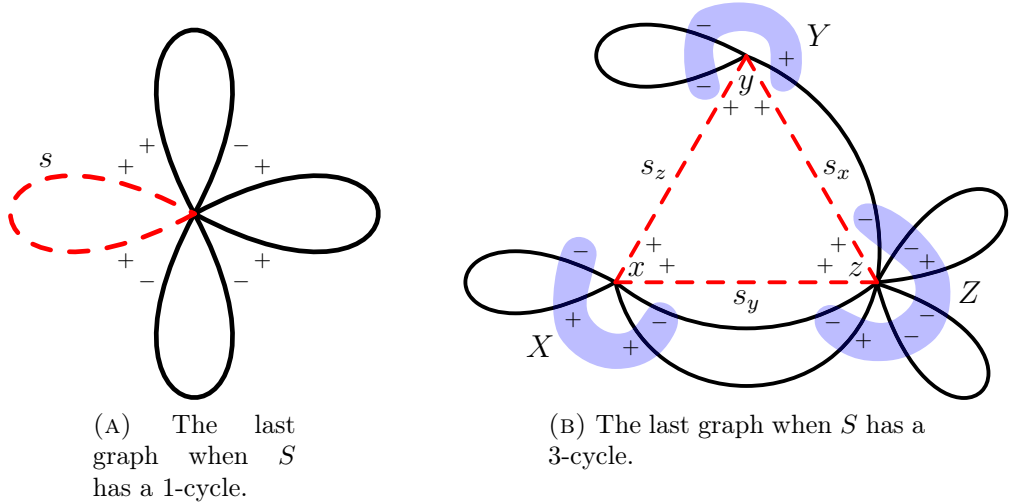
\begin{figure}[htb]
\centering
\subfloat[The last graph when $S$ has a 1-cycle.]{

\labellist
\pinlabel $s$ at 16 102
\tiny\hair 2pt
\pinlabel $+$ at 50 99
\pinlabel $+$ at 50 59
\pinlabel $+$ at 108 99
\pinlabel $+$ at 108 59

\pinlabel $-$ at 99 108
\pinlabel $+$ at 59 108
\pinlabel $-$ at 99 50
\pinlabel $-$ at 59 50
\endlabellist

\includegraphics[width=0.30\textwidth]{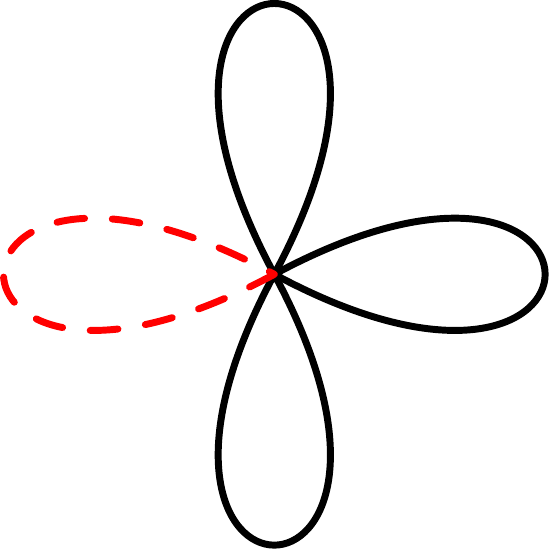}
\label{endgame_1}}
\qquad
\subfloat[The last graph when $S$ has a 3-cycle.]{

\labellist
\pinlabel $s_z$ at 99 120
\pinlabel $s_x$ at 194 120
\pinlabel $s_y$ at 147 51
\pinlabel $y$ at 147 164
\pinlabel $x$ at 90 68
\pinlabel $z$ at 204 68
\pinlabel $X$ at 40 28
\pinlabel $Z$ at 263 56
\pinlabel $Y$ at 185 190
\tiny\hair 2pt
\pinlabel $+$ at 106 70
\pinlabel $+$ at 99 81
\pinlabel $+$ at 188 70
\pinlabel $+$ at 195 81
\pinlabel $+$ at 139 152
\pinlabel $+$ at 155.3 152

\pinlabel $-$ at 61 77
\pinlabel $+$ at 57 42
\pinlabel $+$ at 87 27
\pinlabel $-$ at 102 40
\pinlabel $-$ at 125 163
\pinlabel $-$ at 125 195
\pinlabel $+$ at 168 178
\pinlabel $-$ at 209 94
\pinlabel $-$ at 233 79
\pinlabel $+$ at 241 72
\pinlabel $-$ at 241 44
\pinlabel $-$ at 233 33
\pinlabel $+$ at 208 31
\pinlabel $-$ at 192 40
\endlabellist

\includegraphics[width=0.44\textwidth]{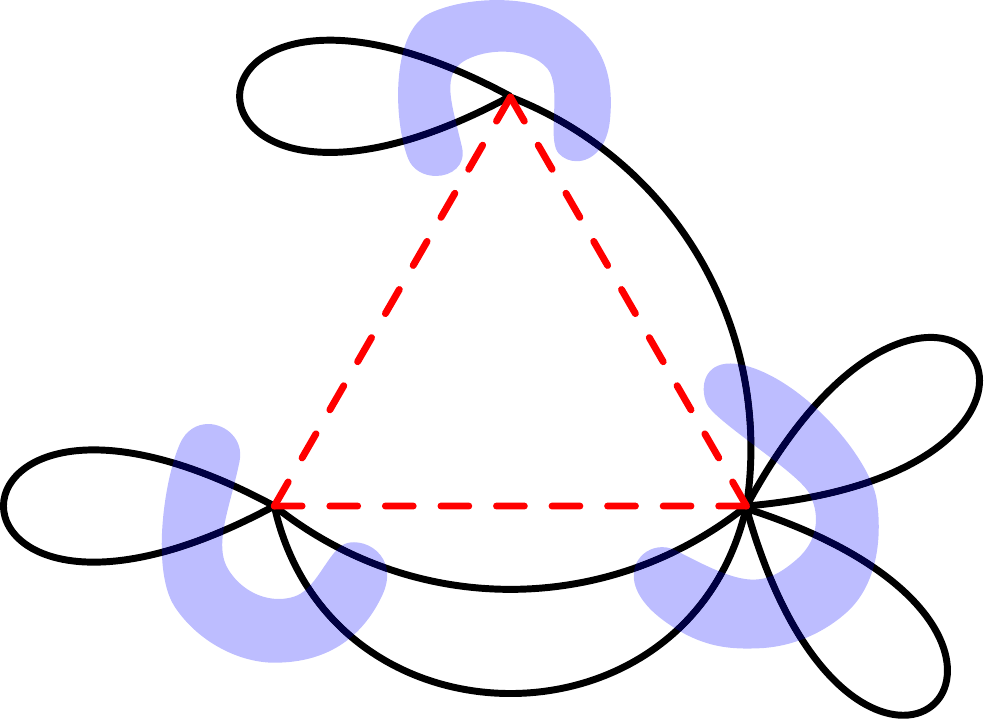}
\label{endgame_2}}
\caption{The last graph in the sequence has one of these two forms.}
\label{collapse_edge_endgame}
\end{figure}

In the first case, the equation from the last vertex is of the form $$2E(s) + \sum_{l, e_i(l) \neq s} \epsilon(l) E(e_{i(l)}) = 0.$$
Notice that since all edges are now loops, each $E(e_i)$ appears with total coefficient either $-2,0$ or $2$. So we can divide the entire equation by $2$, and get $E(s)$ as an integer linear combination of the $E(e_i)$. 

The second case is slightly more complicated. We have three vertices $x,y,z$, with three edges $s_x,s_y,s_z\in S$ connecting the vertices into a triangle. The three vertices give equations of the form

$$E(s_y) + E(s_z) + X = 0, \quad E(s_z) + E(s_x) + Y = 0, \quad E(s_x) + E(s_y) + Z= 0,$$
where $X,Y,Z$ are terms coming from the edges not in $S$. We can solve these equations for $E(s_z)$ as
$$2E(s_z) = -X-Y+Z,$$
and the other two expressions similarly. We just have to show that $-X-Y+Z$ has even coefficients, and then we will be done. As before, any loops contribute a coefficient in $\{-2,0,2\}$ to one of $X,Y$ or $Z$. For edges with the same endpoints as one of $s_x,s_y$ or $s_z$, their coefficients are $1$ or $-1$ at two of $X,Y$ and $Z$ and $0$ at the third. Thus their contribution to $-X-Y+Z$ is also in $\{-2,0,2\}$, and so $-X-Y+Z$ has even coefficients, as required.
\end{proof}

\begin{ex} The following example shows that the edges omitted must be chosen carefully in the multi-cusped case.
Let $M$ be the Whitehead link complement (with the triangulation given by m129 in SnapPy notation). 
Then the matrix of edge equations 
is
$$(\mb A \mid \mb B) = 
\begin{pmatrix}
2 & -1 & 1 & 1 & \vline & 1 & -2 &  0 & 0 \\
-1 & 0 & 0 & 0  & \vline & -1 & 1 & 1 & 1 \\
0 & 1 & -1 & -1 & \vline & 1 & 0 & -2 & -2 \\
-1 & 0 & 0 & 0 & \vline & -1 & 1 & 1 & 1
\end{pmatrix}
$$

The cusp incidence matrix is 
$$C = \begin{pmatrix} 1 & 2 & 1 & 0 \\ 1 & 0 & 1 & 2 \end{pmatrix}$$
and the corresponding graph $G$ is shown in Figure \ref{wh_link_edges_and_cusps}.

\begin{figure}[htbp]
\centering

\labellist
\pinlabel $e_1$ at -5 20
\pinlabel $e_3$ at 168 20
\pinlabel $e_0$ at 81.5 44
\pinlabel $e_2$ at 81.5 -5
\pinlabel $c_0$ at 49 30
\pinlabel $c_1$ at 114 30

\endlabellist

\includegraphics[width=0.36\textwidth]{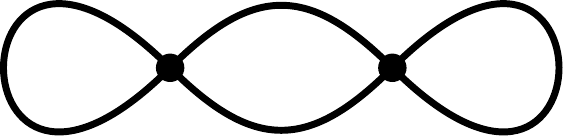}
\caption{The graph of edges and cusps for the Whitehead link.}
\label{wh_link_edges_and_cusps}
\end{figure}
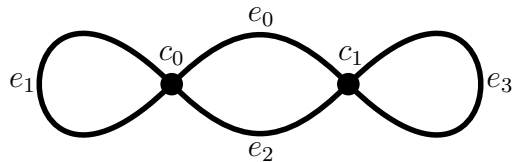

The rows $E_0, E_1, E_2, E_3$ of $(\mb A \mid \mb B)$ satisfy the relations: 
$$E_0 + 2 E_1 + E_2   = 0 \text{ (from cusp 0)}$$
 and 
 $$E_0 + E_2  + 2 E_3  = 0 \text{ (from cusp 1)},$$
which imply that $E_1=E_3$.\\

It follows that $E_0, E_1$ are linearly independent and form an integer basis for
the $\Z$-span of $\{ E_0,E_1,E_ 2,E_3 \}$.
(This basis corresponds to removing the edge $e_2$  
in a maximal tree and an additional loop $e_3$.)

On the other hand, $E_0$, $E_2$ are also linearly independent
and $2E_1$ is in the $\Z$-span of $\{ E_0,E_2 \}$ but $E_1$ is not in $\Z$-span
of $\{ E_0,E_ 2 \}$.
So  $\Z\text{-span} \{ E_0,E_2 \}$ is an index 2 subgroup in $\Z\text{-span} \{ E_0,E_1 \}$.
Thus, a summation using $E_0, E_2$ will most likely
give a different result for the index than using $E_0, E_1$.
\end{ex}

\subsection{A reformulation of the definition of the index}
It is sometimes convenient to work with a slight variation on the tetrahedral index function
\eqref{eq.ID}.
Whenever $a-b, b-c \in \Z$ we define 
\be
\lbl{eq.JDdef}
\JD(a,b,c)= (-q^{\frac{1}{2}})^{-b} \ID(b-c,a-b)
= (-q^{\frac{1}{2}})^{-c} \ID(c-a,b-c) = (-q^{\frac{1}{2}})^{-a} \ID(a-b,c-a).
\ee
Note that the above expressions are equal by the triality identity (\ref{eq.Z3}) for $\ID$, and by
using the duality identity (\ref{eq.Z2}), it follows that $\JD$ is invariant under 
{\em all  permutations} of its arguments.  
Further, we have
\be
\lbl{eq.JD_add_tet_sol}\JD(a+s,b+s,c+s) = (-q^{\frac{1}{2}})^{-s} \JD(a,b,c) \text{ for all } s \in \R.
\ee
We also note that the quadratic identity \eqref{eq.IDquadratic} can be rewritten in the form
\be
\lbl{eq.JD_quad_id}
\sum_{a\in\Z} \JD(a,b,c) \JD(a+x,b,c) \,q^a= \delta_{x,0}.
\ee
This follows since
\begin{align*} 
LHS &= \sum_{a\in\Z} (-q^{1/2})^{-b} \ID(b-c,a-b)\, (-q^{1/2})^{-b} \ID(b-c,a-b+x) \, q^a \\
&= \sum_{e\in \Z} \ID(m,e) \ID(m,e+x) \,q^e = \delta_{x,0}
\end{align*}
by (\ref{eq.IDquadratic}) with $m=b-c$ and  $e = a - b$.\\

Now suppose that $M$ is  a 3-manifold  whose boundary $\bd M$ consists of $r \ge 1$ tori.
Let $\overline{\mb A}$,  $\overline{\mb B}$ and  $\overline{\mb C}$ be the matrices of
angle structure equation coefficients as in Definition \ref{matrix equations}, and 
let $\bar{\boldsymbol{a}}_j, \bar{\boldsymbol{b}}_j, \bar{\boldsymbol{c}}_j$ for $j=1,\dots N$ denote the columns of $\overline{\mb A}$,  $\overline{\mb B}$ and  $\overline{\mb C}$ respectively. 
For each $\kk \in \Z^N$ and oriented multi-curve $\varpi$ in $\bd M$ representing a homology class
$[\varpi]  \in H_1(\bd M,\Z)$,
let 
$$\bar a_j(\kk,\varpi)=\kk \cdot \bar\aa_j + \bar a_{\varpi,j},~~
\bar b_j(\kk,\varpi)=\kk \cdot \bar\bb_j + \bar b_{\varpi,j},~~
\bar c_j(\kk,\varpi)=\kk \cdot \bar\cc_j + \bar c_{\varpi,j},
$$
and 
$$a_j(\kk,\varpi) =\bar a_j(\kk,\varpi)- \bar b_j(\kk,\varpi),~
b_j(\kk,\varpi)=\bar c_j(\kk,\varpi)-\bar b_j(\kk,\varpi) .$$

Then the definition \eqref{eq.indT} of the index of the triangulation $\mathcal{T}$ of a manifold $M$ 
can be written 
 \be
I_{\CT}(\varpi)(q)=\sum_{\kk \in \BZ^{N-r}\subset \Z^N}
(-q^{\frac{1}{2}})^{\kk \cdot \bnu + \nu_{\varpi}} 
\prod_{j=1}^N \ID(-b_j(\kk,\varpi), a_j(\kk,\varpi) ) \, ,
\ee
where the sum is over $\boldsymbol{k}$ in any coordinate plane $\Z^{N-r} \subset \Z^N$
corresponding to a set $\mathcal B$ of $N-r$ basic edges as given by Theorem \ref{pick_edges}.

Let $\kk=(k_1, \ldots, k_N)$, $\bnu=(\nu_1, \ldots, \nu_N)$ and let $\bar b_{ij}$ be the $(i,j)$ entry of $\overline{\mb B}$. Then we have 
$$\kk \cdot \bnu = \sum_i k_i \nu_i = \sum_i 2 k_i - \sum_{i,j} k_i \bar b_{ij}
= \sum_i 2 k_i - \sum_j \kk \cdot \bar \bb_j$$
and $\nu_\varpi = - \sum \bar b_{\varpi,j}$, so
$$\kk \cdot \bnu  + \nu_\varpi =  \sum_i 2 k_i - \sum_j \bar b_j(k,\varpi).$$
Hence, grouping together the contributions from tetrahedron $j$, we have
\begin{align}
I_{\CT}(\varpi)(q) &= \sum_{\kk \in \Z^{N-r}} 
 q^{\sum_i k_i}   \prod_j  (-q^{\frac{1}{2}})^{- \bar b_j (k,\varpi)} 
\ID(\bar b_j(\kk,\varpi)-\bar c_j(\kk,\varpi) , \bar a_j(\kk,\varpi) -\bar b_j(\kk,\varpi)) \notag\\
& = \sum_{\kk \in \Z^{N-r}} 
 q^{\sum_i k_i}   \prod_j 
\JD(\bar a_j(\kk,\varpi), \bar b_j(\kk,\varpi) , \bar c_j(\kk,\varpi)),
\label{JD_index_eq}
\end{align}
where $\Z^{N-r} \subset \Z^N$ corresponds to a set $\mathcal B$ of $N-r$ basic edges as given by Theorem \ref{pick_edges}.

In particular, this expression shows that the index does not depend on the quad-choice used in the original definition.

\begin{rmk}\label{basis_indept}
Next we show that the definition of index in \eqref{eq.indT} does not depend on the 
choice of integer basis for the integer row space $\Lambda \subset \R^{2N}$ of the Neumann-Zagier 
matrix $(\mb A \mid \mb B)$.

Each $x \in \Lambda$ can be written in the form 
\be \label{lambda_rep}
x = \sum_i k_i E_i
\ee
 where $E_i$ is the 
$i$th row of $(\mb A \mid \mb B)$ and $\kk =(k_1, \ldots, k_N) \in \Z^N$.
We claim that the expression
\be\label{term_in_sum}
J(x,\varpi) = q^{\sum_i k_i}   \prod_j  \JD(\bar a_j(\kk,\varpi), \bar b_j(\kk,\varpi) , \bar c_j(\kk,\varpi))
\ee
is well-defined,  depending only on $x \in \Lambda$ and not on the choice of $\kk$ in \eqref{lambda_rep}.

To see this, consider the linear map 
$\psi : \R^N \to \R^{2N}$ defined by  $\psi(k_1, \ldots,k_N) = \sum_i k_i E_i$
and let $\langle C \rangle \subset \R^N$ be the real subspace generated by the cusp relation vectors
$(c_{h1}, \ldots, c_{hN})$ where the cusp index $h$ varies over $\{1, \ldots, r \}$. 
Then $\psi(\langle C \rangle) = 0$ by the cusp relations  \eqref{cusp_relations}, 
and the work of Neumann and Zagier (\cite{NZ,Neumann-combi}) also shows that
$\dim \rm{Im}\, \psi =N-r$ and $\dim \langle C \rangle=r$
where $r$ is the number of cusps. Hence $\langle C \rangle = \ker \psi$. 

So if $x = \sum_i k_i E_i = \sum_i k'_i E_i$ where $\kk'=(k'_1, \ldots, k'_N) \in \R^N$ then
$\kk'=\kk + \boldsymbol{c}$ where  $\boldsymbol{c} \in \langle C \rangle$.
We claim that replacing $\kk$ by $\kk'$ does not change the  expression \eqref{term_in_sum}.
To see this, suppose
we replace $\kk$ by $\kk'$ where $k'_i= k_i + s  c_{hi} $ for $i=1, \ldots, N$ and $s \in \R$.
Then the term $q^{\sum_i k_i}$ in \eqref{term_in_sum} is multiplied 
by $q^{s n_h}$ where  $n_h$ is the number of vertices in the triangulation of cusp $h$, while 
$\bar a_j(\kk,\varpi), \bar b_j(\kk,\varpi) , \bar c_j(\kk,\varpi)$ are increased by $s$ for each
triangle of tetrahedron $j$ lying in the cusp $h$.
By (\ref{eq.JD_add_tet_sol}), this changes $\prod_j \JD(\bar a_j(\kk,\varpi), \bar b_j(\kk,\varpi) , \bar c_j(\kk,\varpi))$ by a factor $(-q^{1/2})^{-2s n_h}$ since there are $2n_h$ triangles on cusp $h$.
Hence the right hand side of \eqref{term_in_sum} does not change.

We conclude that the expression for index in \eqref{JD_index_eq} can be rewritten in the form
\be
I_\CT(\varpi) = \sum_{x \in \Lambda} J(x,\varpi)
\ee
and so does not depend on a choice of basis for $\Lambda$.
Further, we can evaluate $I_\CT(\varpi)$ by choosing an integer basis for $\Lambda$
corresponding to a set of basic edges as given by Theorem \ref{pick_edges}, and we recover
the definition of index in  \eqref{eq.indT}.

It also follows that we can write the index in the form 
\be
I_\CT(\varpi) = \sum_{\kk \in S} q^{\sum_i k_i}   \prod_j 
\JD(\bar a_j(\kk,\varpi), \bar b_j(\kk,\varpi) , \bar c_j(\kk,\varpi)),
\ee
where $S \subset \Z^N$ is any complete set of coset representatives for 
$(\Z^N + \langle C \rangle)/\langle C \rangle \subset \R^N /\langle C \rangle$. 
 \end{rmk}

\subsection{Invariance of index under isotopy of peripheral curve}

\begin{thm}
\label{index_isotopy_invariance}
Let $\varpi$ be an oriented simple closed curve $\varpi$ in $\partial M$ which is a normal curve relative to the triangulation $\mathcal{T}_\bd$ of  $\partial M$. 
Then the  index $I_\mathcal{T}(\varpi)$ is invariant under isotopy of the curve $\varpi$ in $\bd M$.
\end{thm}
\begin{proof}
Suppose we have two isotopic oriented normal curves $\varpi_1, \varpi_2$. Then we can convert one into the other via a sequence of moves (and their inverses) of the form shown in Figure \ref{isotopy moves}. That is, we choose a point $p$ on the curve $\varpi$ and an arc $\alpha$, disjoint  from $\varpi$ other than at $p$, and which joins $p$ to either a vertex or a point in the interior of an edge of $\mathcal{T}_\bd$. We then push the curve along and in a regular neighbourhood of $\alpha$ over the vertex or edge.

\begin{figure}[htb]
\centering
\includegraphics[width=0.4\textwidth]{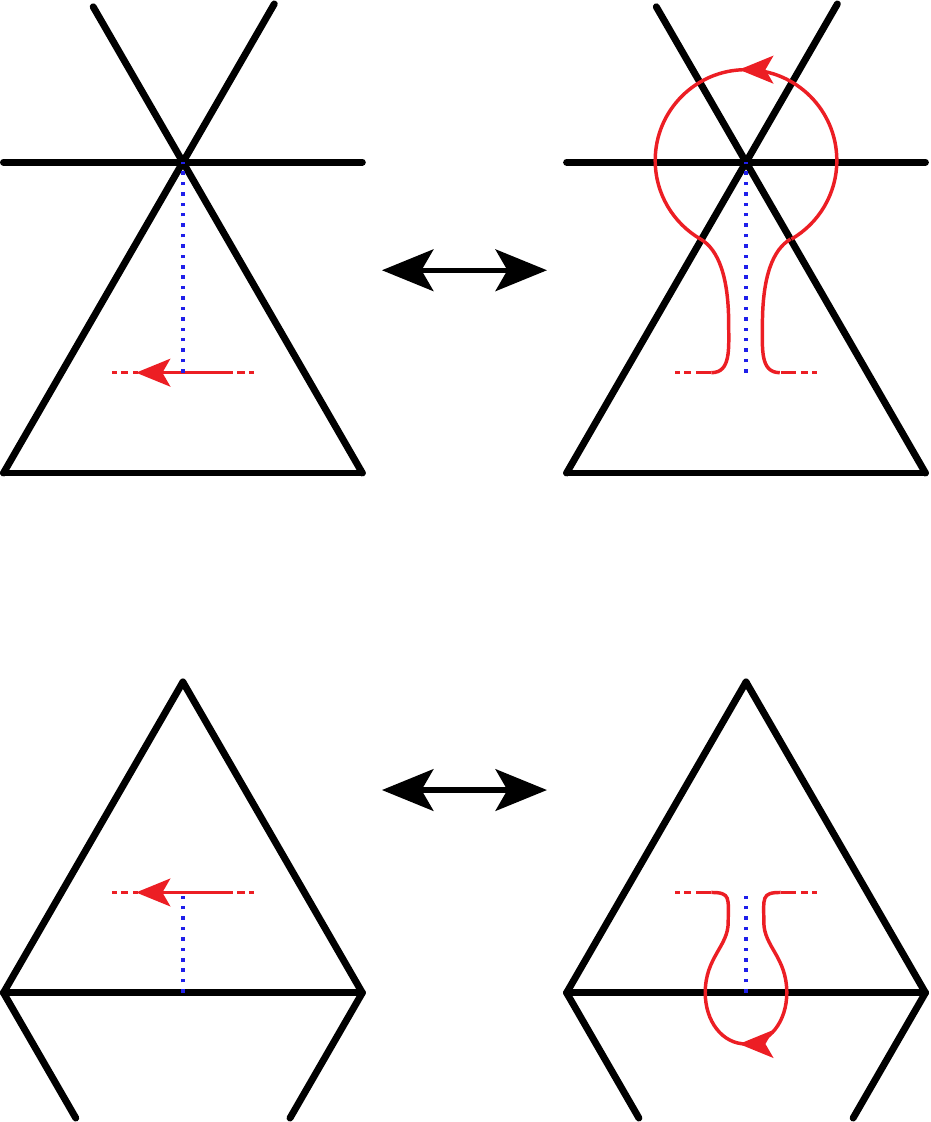}
\caption{The two kinds of isotopy move on a curve relative to the triangulation of $\partial M$.}
\label{isotopy moves}
\end{figure}

We will show that $I_\mathcal{T}(\varpi)$ is invariant under these moves. Note that the result of these isotopies will not in general be normal curves, so we need to extend the definition of the index to deal with these cases as well. The class of curves we work in consists of oriented simple closed curves, transverse to $\mathcal{T}_\bd^{(1)}$ and disjoint from  $\mathcal{T}_\bd^{(0)}$. For our purposes we will deal only with curves that are non trivial in $H_1(\partial M)$, and so none of our curves is disjoint from $\mathcal{T}_\bd^{(1)}$. Given such a curve, it enters a triangle somewhere on one edge, and can exit out either of the two other edges, or the same edge that it entered, either to the left or the right of its entry point. Thus there are four ways in which a component of a curve intersects a given triangle. These contribute to the index in the following way. See Figure \ref{4_types_of_curve_in_triangle2}.

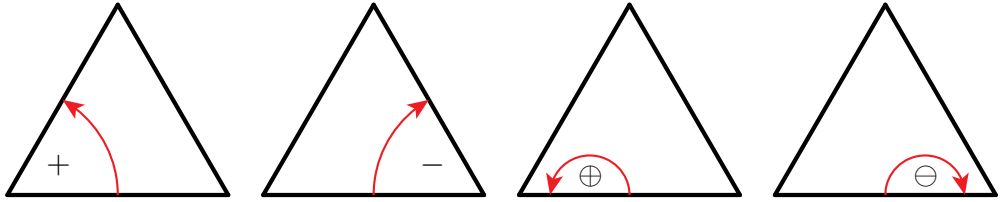
\begin{figure}[htb]
\centering
\labellist
\pinlabel $+$ at 25 15
\pinlabel $-$ at 200 15
\pinlabel $\oplus$ at 274 10
\pinlabel $\ominus$ at 431 10
\endlabellist
\includegraphics[width=0.8\textwidth]{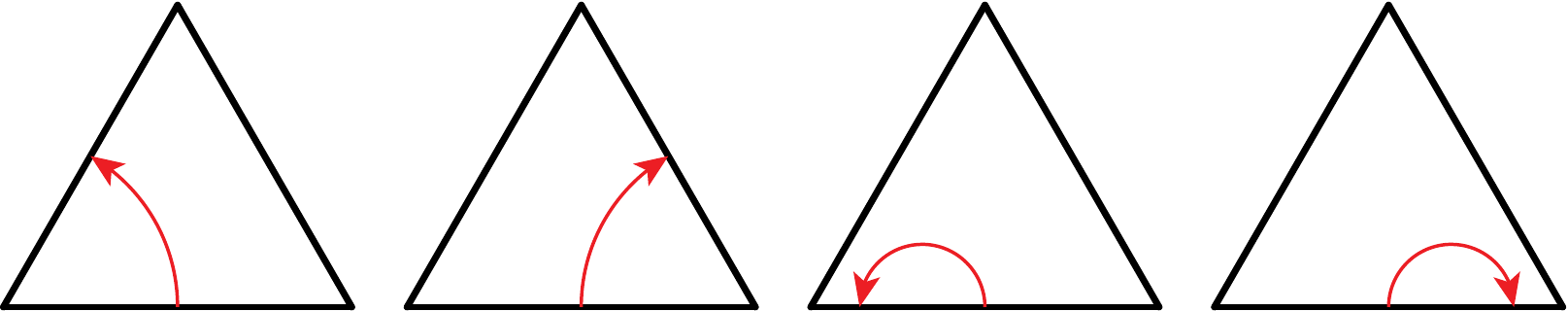}
\caption{The four ways in which a curve can travel through a triangle. We call the last two possibilities a \emph{positive backtrack} and \emph{negative backtrack} respectively.}
\label{4_types_of_curve_in_triangle2}
\end{figure}

If the curve turns either left or right around a corner of the triangle then it contributes to the index in exactly the same way as for a normal curve: we add $+1$ to the entry in the vector $(\bar{a}_\varpi | \bar{b}_\varpi | \bar{c}_\varpi)$ corresponding to the angle at the edge of the tetrahedron we are turning around if we are going anticlockwise around the corner, and add $-1$ if we are going clockwise. Compare with Figure \ref{truncated_tetra2}.

Here we define the effect of backtracks on the index calculation: (This is, of course, chosen in such a way as to be consistent with the index calculated with curves without backtracks.) We do not change the vector $(\bar{a}_\varpi | \bar{b}_\varpi | \bar{c}_\varpi)$. These backtracks only alter the power of $(-q^{1/2})$, either multiplying the expression by $(-q^{1/2})$ for a positive backtrack (turning to the left), or by $(-q^{1/2})^{-1}$ for a negative backtrack (turning to the right). We indicate these using the symbols $\oplus$ and $\ominus$. 

Note that by (\ref{eq.JD_add_tet_sol}), a positive backtrack has the same effect on the index as anticlockwise turns around each of the three corners of a triangle. Note also that by (\ref{JD_index_eq}), an anticlockwise loop around a vertex of the triangulation produces a power of $(-q^{1/2})^{2}$. This follows since adding an anticlockwise loop around an end of the $i$th edge has the effect of shifting the sum by one in the $k_i$ component. The terms are unchanged after shifting other than the term $q^{\sum_i k_i}$, and the effect is to multiply the index by $q$.

Thus an anticlockwise loop around a vertex is cancelled by two negative backtracks, and anticlockwise turns around each of the three corners of a triangle are cancelled by one negative backtrack.

\begin{figure}[htbp]
\centering
\labellist
\small\hair 2pt
\pinlabel $+$ at 37 396
\pinlabel $+$ at 37 376
\pinlabel $+$ at 68 396
\pinlabel $+$ at 68 376
\pinlabel $+$ at 52.5 404
\pinlabel $+$ at 52.5 368
\pinlabel $\ominus$ at 37 340
\pinlabel $\ominus$ at 68 340

\pinlabel $+$ at 165.1 396
\pinlabel $+$ at 165.1 376
\pinlabel $+$ at 196.1 396
\pinlabel $+$ at 196.1 376
\pinlabel $+$ at 180.6 404
\pinlabel $\ominus$ at 196.1 340
\pinlabel $-$ at 145.1 305
\pinlabel $-$ at 216.1 305

\pinlabel $+$ at 293.3 396
\pinlabel $+$ at 293.3 376
\pinlabel $+$ at 324.3 396
\pinlabel $+$ at 324.3 376
\pinlabel $+$ at 308.8 404
\pinlabel $+$ at 308.8 368
\pinlabel $\ominus$ at 324.3 340
\pinlabel $\ominus$ at 324.3 316.5

\pinlabel $+$ at 421.5 396
\pinlabel $+$ at 421.5 376
\pinlabel $+$ at 452.5 396
\pinlabel $+$ at 452.5 376
\pinlabel $+$ at 437 404
\pinlabel $\ominus$ at 437 310
\pinlabel $-$ at 401.7 305
\pinlabel $-$ at 472.7 305

\pinlabel $+$ at 116.75 252
\pinlabel $+$ at 81.75 192
\pinlabel $+$ at 151.75 192
\pinlabel $\ominus$ at 116.5 174.5

\pinlabel $\oplus$ at 229.6 199.5
\pinlabel $\ominus$ at 244.6 174.5

\pinlabel $\ominus$ at 386.8 199.5
\pinlabel $\oplus$ at 372.8 174.5

\pinlabel $\ominus$ at 116.5 75
\pinlabel $\oplus$ at 116.5 30

\pinlabel $\ominus$ at 244.6 44
\pinlabel $\ominus$ at 244.6 30
\pinlabel $\oplus$ at 230.6 55
\pinlabel $\oplus$ at 258.6 55

\pinlabel $\ominus$ at 386.8 55
\pinlabel $\oplus$ at 372.8 30

\endlabellist
\includegraphics[width=\textwidth]{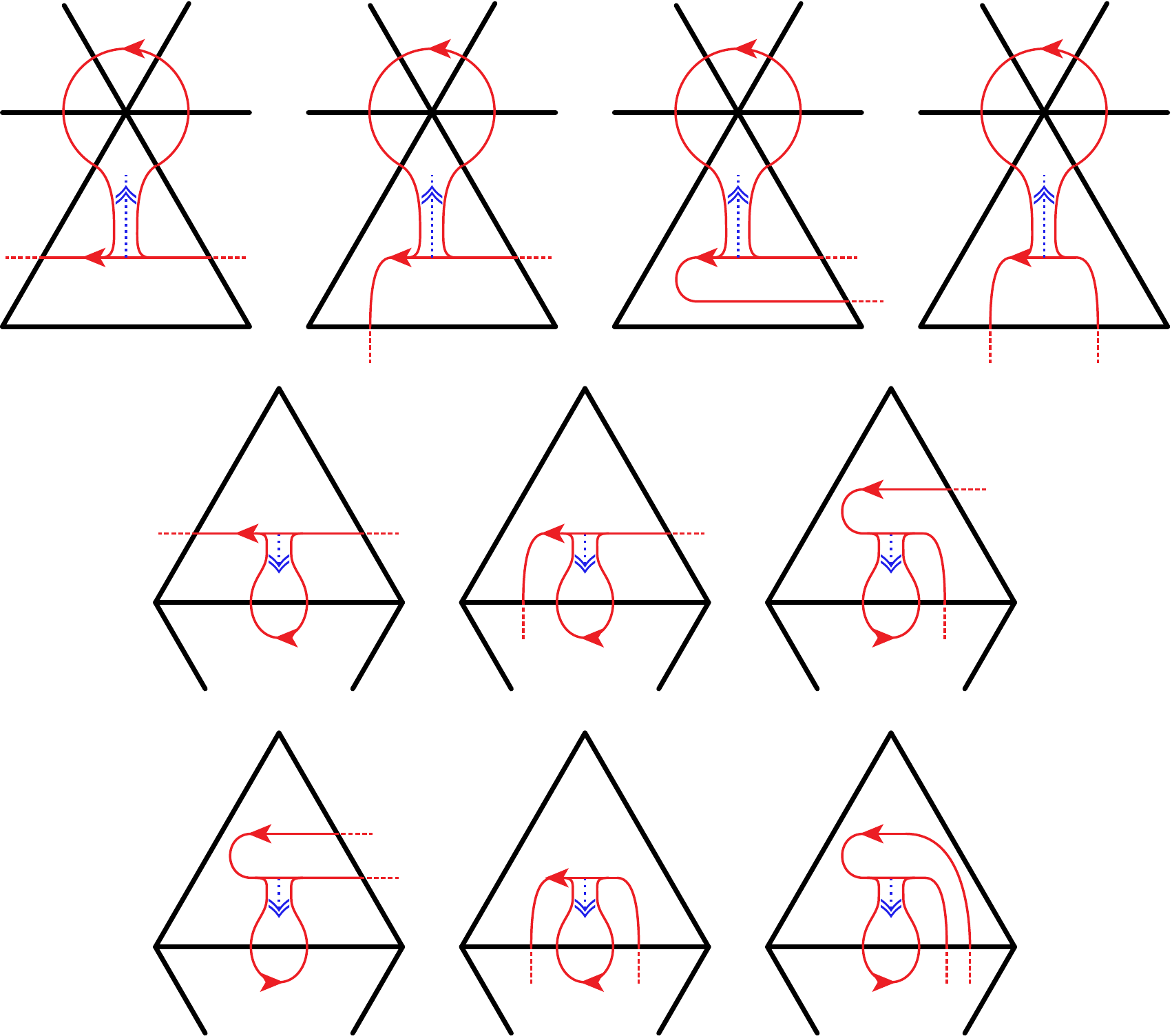}
\caption{Four cases of isotopy across a vertex and six cases of isotopy across an edge. All other cases are symmetries of these.}
\label{index_changes_under_isotopies}
\end{figure}
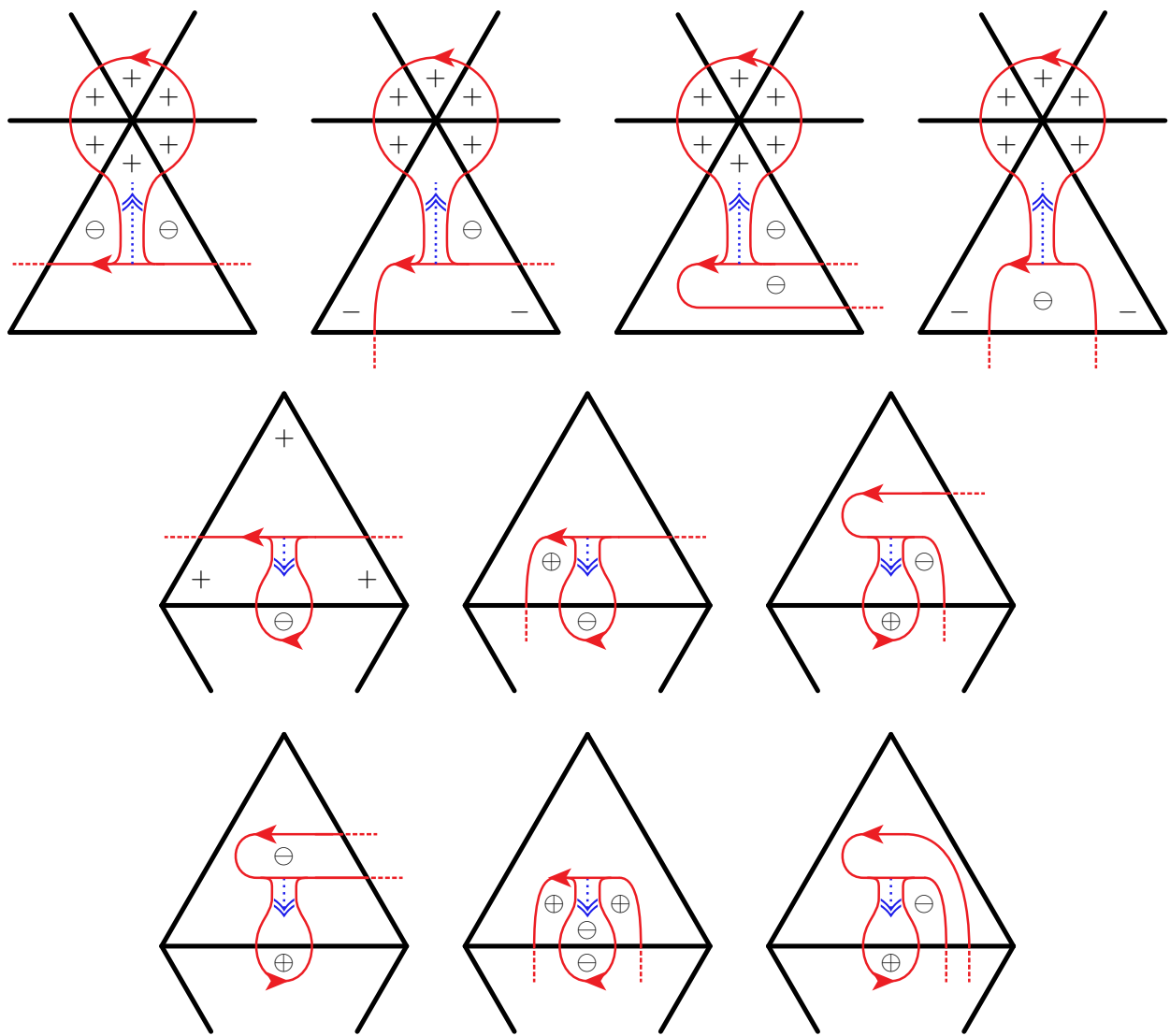

Now all we need to do is to show that each version of the moves from Figure \ref{isotopy moves} preserves the index, using the above rules. There are different versions of the isotopy moves depending on where the curve we are acting on enters or exits the triangle. We show the possibilities in Figure \ref{index_changes_under_isotopies}. Here the $+,-,\oplus$ and $\ominus$ signs show the \emph{difference} in the index calculation under the isotopy as we change from one curve to the other in the direction following the double head arrow. Note that reversing the arrow on the curve flips all of the signs, as does reflecting the picture. With combinations of these symmetries applied to the ten cases shown we obtain all possible ways in which the isotopy can be made relative to the position of the curve. Considering each case in turn, we can see that the signs cancel out and so the index is unchanged by these moves. 

For example, consider the second diagram in the top row of Figure \ref{index_changes_under_isotopies}. We start with a curve that enters the right side of the triangle and exits the bottom. We isotope this curve by pushing it over the top vertex of the triangle. This has the following effects:
\begin{itemize}
\item Remove an anticlockwise turn around the lower right corner of the triangle, this changes the coefficient at that angle by $-1$.
\item Add a negative backtrack to the right edge of the triangle.
\item Add an anticlockwise turn around each of the angles at the top vertex of the triangle other than the top corner of the triangle itself.
\item Add a clockwise turn around the lower left corner of the triangle. 
\end{itemize}
We view the top corner of the triangle as having both a $+$ and a $-$, so that the total change in the index calculation consists of one anticlockwise turn around a vertex, one negative backtrack and clockwise turns around each of the three corners of the triangle. By our above rules, these cancel out and so the index is unchanged.
\end{proof}

\begin{rmk}
These calculations are exactly analogous to those for calculating the holonomy of a peripheral curve given shapes of ideal hyperbolic tetrahedra satisfying Thurston's gluing equations. Thus this argument can easily be adapted to reprove the well-known fact that the holonomy is independent of the choice of simple closed curve representing an element of $H_1(\bdry M;\Z).$
\end{rmk}

\section{Invariance of index under the 0--2 move}
\label{sec.02move}

Let $M$ be a cusped 3-manifold and consider the 0--2 move on a pair $\CT$ and $\wt\CT$ of ideal
triangulations of $M$ 
with $N$ and $N+2$ tetrahedra, as shown in Figure 
\ref{0-2move_labelled}. 

\begin{figure}[htb]
\centering
\labellist
\pinlabel $e$ at 32 95
\pinlabel $e$ at 32 95
\pinlabel $e_2$ at 8 129
\pinlabel $e_4$ at 55 127
\pinlabel $e_1$ at 8 55
\pinlabel $e_3$ at 55 54

\pinlabel $e_2$ at 165 124
\pinlabel $e_4$ at 215 123
\pinlabel $e_1$ at 165.5 55
\pinlabel $e_3$ at 216 55
\pinlabel $e'$ at 175 79.5
\pinlabel $e''$ at 207 97.5
\pinlabel $\wt{e}$ at 186 110.5

\small\hair 2pt
\pinlabel $\bar a_2$ at 113 58
\pinlabel $\bar a_1$ at 138 58
\pinlabel $\bar c_2$ at 121 71
\pinlabel $\bar b_2$ at 121 45
\pinlabel $\bar b_1$ at 130 71
\pinlabel $\bar c_1$ at 130 45

\pinlabel $\bar a_1$ at 238 70
\pinlabel $\bar a_2$ at 263 70
\pinlabel $\bar c_1$ at 246 83
\pinlabel $\bar b_1$ at 246 57
\pinlabel $\bar b_2$ at 255 83
\pinlabel $\bar c_2$ at 255 57

\pinlabel $\bar a_1$ at 172 20
\pinlabel $\bar c_1$ at 159 20
\pinlabel $\bar b_1$ at 185 20
\pinlabel $\bar a_2$ at 172 9.4
\pinlabel $\bar b_2$ at 159 9.4
\pinlabel $\bar c_2$ at 185 9.4

\pinlabel $\bar a_2$ at 216 162
\pinlabel $\bar c_2$ at 203 162
\pinlabel $\bar b_2$ at 229 162
\pinlabel $\bar a_1$ at 216 151.4
\pinlabel $\bar b_1$ at 203 151.4
\pinlabel $\bar c_1$ at 229 151.4

\endlabellist
\includegraphics[width=\textwidth]{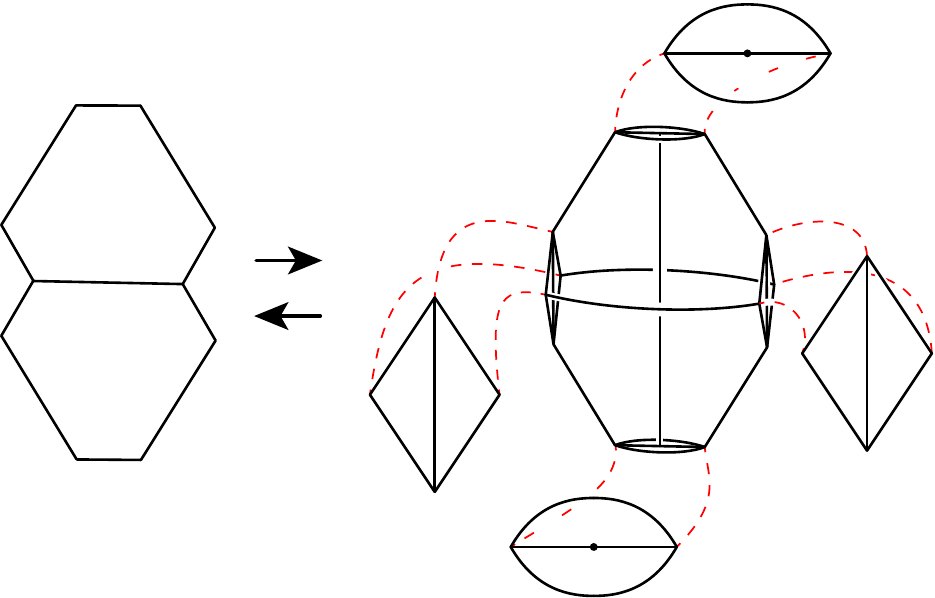}
\caption{The 0--2 move shown with truncated tetrahedra. The four triangulated ends of the new pair of tetrahedra are shown. All of the ``zoomed in'' pictures are as seen from viewpoints outside of the pair of tetrahedra. The labels at the corners of the triangles in the ``zoomed in'' pictures are explained in Section \ref{sub.peripheral}.
}

\label{0-2move_labelled}
\end{figure}
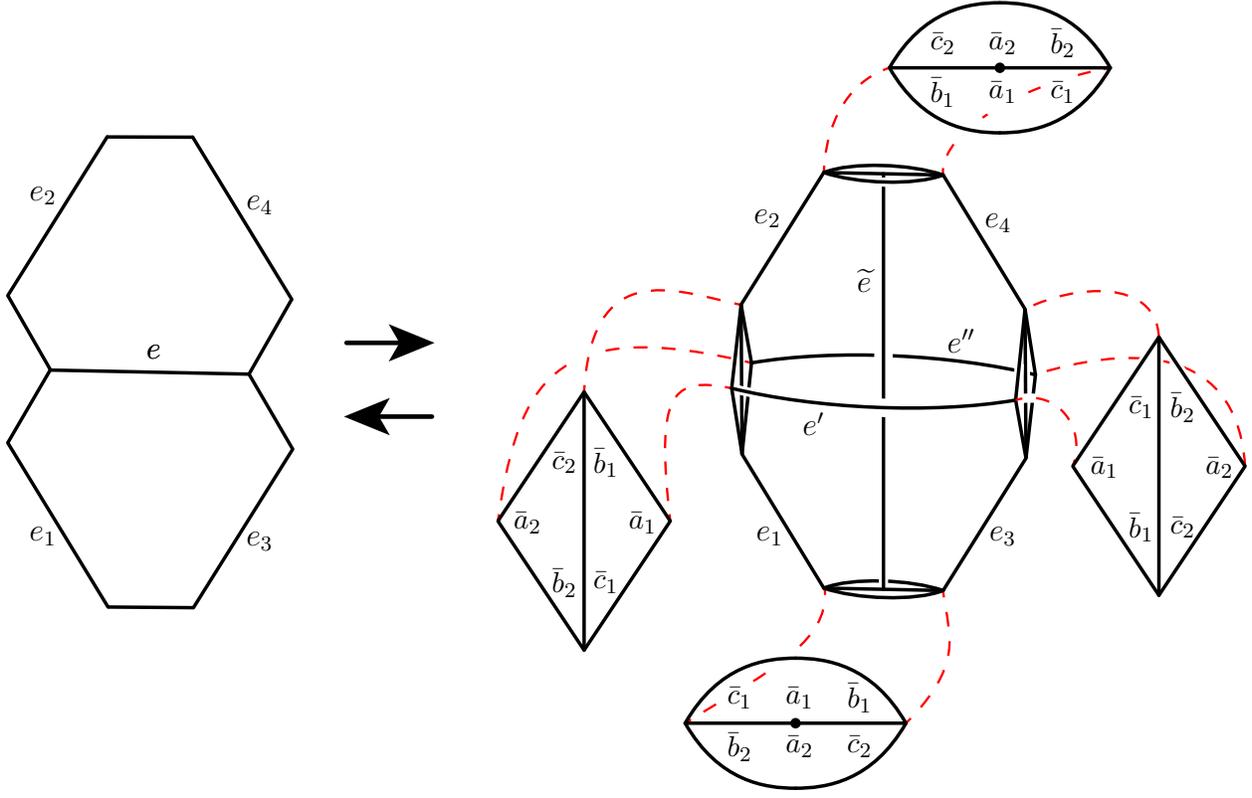

\begin{thm}
\lbl{thm.20move.index}
Suppose that $\CT$ and $\wt\CT$ are ideal triangulations related by a 
0--2 move and both admit an index structure. Then, for any $[\varpi] \in H_1(\bd M;\Z)$,
$I_{\CT}([\varpi])=I_{\wt\CT}([\varpi])$.
\end{thm}

\begin{proof}
Our assumptions imply that both $I_{\CT}$ and $I_{\wt\CT}$ exist.
We now compare these indices using the alternative definition \eqref{JD_index_eq} and
quadratic identity \eqref{eq.JD_quad_id}.  

We use the labelling of the two bigons and triangles on $\wt\CT$ shown
in Figure \ref{0-2move_labelled}. 
Let $T_i$ for $i=3,\dots,N+2$ denote the tetrahedra in $\CT$,
and let $T_1, T_2$ be the additional tetrahedra added in $\ti\CT$. 
Note that the edge $e$ in $\CT$ splits into two edges $e',e''$ in $\wt\CT$, and there is another new edge $\wt e$ in $\wt\CT$. We abuse notation by identifying the symbols for the corresponding remaining edges in $\CT$ and $\wt\CT$. We denote these as $e_1, \ldots, e_{N-1}$.

Let $\wt\kk \in \Z^{N+2}$ be a weight function on
the edges of $\wt\CT$ and write $\ti\kk= (k',k'',\wt k, k_1, \ldots, k_{N-1})$
where  $k',k'',\wt k, k_i$ are the values of $\wt\kk$ on $e',e'', \wt e$ and $e_i$ respectively. 
Similarly, let $\kk = (k, k_1, \ldots, k_{N-1})\in\Z^N$ be a corresponding weight function on $\CT$.
We choose label $\bar a_j$ on the edge $\ti e$ on tetrahedron $T_j$ for $j=1,2$;
then the location of labels $\bar b_j,\bar c_j$ are determined using the orientation on $M$. 

Let $\varpi$ be an oriented multi-curve which is normal with respect to the triangulation $\CT_\bd$
of $\bd M$ induced by $\CT$, and let $\ti\varpi$ an oriented multi-curve which is normal
with respect to $\wt\CT_\bd$ and represents the same homology class $[\varpi]\in H_1(\bd M;\Z)$.
Let  $J(T_j,\kk,\varpi) = \JD(\bar a_j(\kk,\varpi), \bar b_j(\kk,\varpi), \bar c_j(\kk,\varpi))$ denote the contribution of tetrahedron $T_j$
to the index with weight function $\kk$ on its edges and peripheral curve $\varpi$ on its truncated ends, and similarly let $J(T_j, \ti \kk,\ti\varpi)$ be contribution with weight function $\ti \kk$ and peripheral curve $\ti\varpi$.

To compute $I_{\CT}$ we use Theorem \ref{pick_edges} to choose an excluded set $\mathcal X$ of $r$ edges
in a maximal tree with 1- or 3-cycle in $\CT$ 
to be omitted from the summation in \eqref{JD_index_eq}.   

Case 1: If $e \notin \mathcal X$ we can order the edges of $\CT$ so
that $\mathcal X = \{ e_{N-r}, \ldots, e_{N-2}, e_{N-1} \}$. 
Then we can compute $I_{\wt\CT}$ by omitting the same edge set $\widetilde{\mathcal X} = \mathcal X$.

Case 2: If $e \in \mathcal X$  we can order the edges so $\mathcal X = \{ e, e_{N-r+1}, \ldots, e_{N-2}, e_{N-1} \}$.
Then we can compute $I_{\wt\CT}$ by omitting the edge set 
$\widetilde{\mathcal X} =  \{ e'', e_{N-r+1}, \ldots, e_{N-2}, e_{N-1} \}$.

Then
\be
\label{eq.JD_index_sum1}
I_{\wt\CT} (\ti\varpi)= \sum_{\wt\kk \in \wt S} \ 
q^{k' + k'' + \ti k + \sum_{i=1}^{N-1} k_i } \,
J(T_1,\ti\kk,\ti\varpi) J(T_2,\ti\kk,\ti\varpi) \prod_{j=3}^{N+2} J(T_j, \ti \kk,\ti\varpi),
\ee
where $$\wt S = \left\{ \wt\kk = (k',k'',\wt k, k_1, \ldots, k_{N-1}) \in \Z^{N+2} : k_{N-i}=0 \text { for } i=1, \ldots, r  \right\} \text{  in case 1}$$
and 
$$\wt S = \left\{ \wt\kk = (k',k'',\wt k, k_1, \ldots, k_{N-1})\in \Z^{N+2} : k''=0, k_{N-i}=0 \text { for } i=1, \ldots, r-1 \right\} \text{  in case 2}.$$
Note that in both cases, the new edge $\tilde e$ is {\em included} in the set of basic edges so $\tilde k$
varies over $\Z$ in the sum.

Now we look at the contribution to 
$I_{\wt\CT}$ coming from the tetrahedra $T_1,T_2$ and summed over the weight $\wt k$ on $\tilde e$,  namely
\be
\label{tilde_k_sum}
\sum_{\wt k \in \Z}
q^{\wt k}
J( T_1; \ti\kk,\ti\varpi)  J( T_2; \ti\kk,\ti\varpi)
\ee
where
$$J( T_1; \ti\kk,\ti\varpi)=\JD(k' + \wt k+\bar a_{\ti\varpi,1}, k_2+k_3+\bar b_{\ti\varpi,1},k_1+k_4+\bar c_{\ti\varpi,1})$$
and
$$J( T_2; \ti\kk,\ti\varpi)=\JD(k'' + \wt k +\bar a_{\ti\varpi,2},  k_1+k_4+\bar b_{\ti\varpi,2},k_2+k_3+\bar c_{\ti\varpi,2}).$$

Recall that  $\varpi$ be an oriented multi-curve which is normal with respect to the triangulation $\CT_\bd$
of $\bd M$ induced by $\CT$. 
Since the index only depends on the homology class of a peripheral curve,
we can calculate $I_{\wt\CT}([\varpi])$ by using for $\wt\varpi$ a corresponding curve on $\bd M$ which is
normal with respect to $\wt\CT_\bd$ and goes ``straight through'' each pair of added triangles on $\bd M$.
See Figure \ref{0-2_bdry_curve_changes}.  

\begin{figure}[htb]
\centering
\labellist
\pinlabel $\bar a_2$ at 93 29
\pinlabel $\bar a_1$ at 118 29
\pinlabel $\bar c_2$ at 101 42
\pinlabel $\bar b_2$ at 101 16
\pinlabel $\bar b_1$ at 110 42
\pinlabel $\bar c_1$ at 110 16

\pinlabel $\bar a_2$ at 105.5 88
\pinlabel $\bar c_2$ at  92.5 88
\pinlabel $\bar b_2$ at 118.5 88
\pinlabel $\bar a_1$ at 105.5 77.4
\pinlabel $\bar b_1$ at 92.5 77.4
\pinlabel $\bar c_1$ at 118.5 77.4
\endlabellist
\includegraphics[width=10cm]{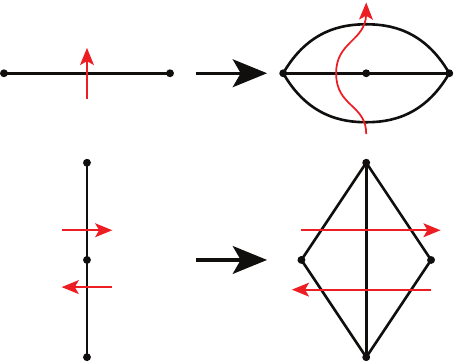}
\caption{Changes in peripheral curve from $\varpi$ in $\CT_\bd$  to $\ti\varpi$ in $\ti\CT_\bd$. Note that the curve in the top diagram
could go either way around the new vertex. }
\label{0-2_bdry_curve_changes}
\end{figure}
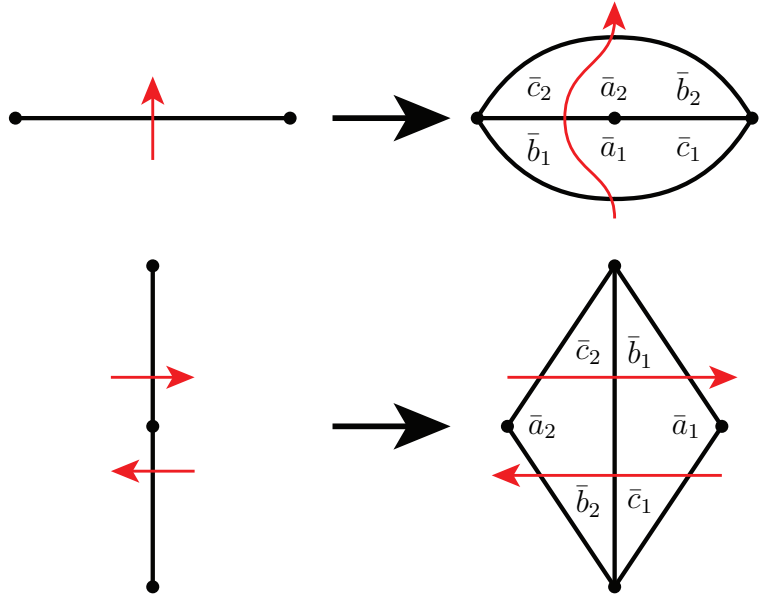

Then we have 
$$
\bar a_{\ti\varpi,1}=\bar a_{\ti\varpi,2}=0,\quad
\bar b_{\ti\varpi,1} = \bar c_{\ti\varpi,2} = x,\quad \bar b_{\ti\varpi,2} = \bar c_{\ti\varpi,1} = y,
$$
for some $x,y \in \Z$,
and 
$$\bar a_{\ti\varpi,j}=\bar a_{\varpi,j},~~ \bar b_{\ti\varpi,j}=\bar b_{\varpi,j},~~\bar c_{\ti\varpi,j}=\bar c_{\varpi,j} \text{ for } j = 3, \ldots, N+2.$$

Using the invariance of $\JD$ under all permutations of its arguments and the quadratic identity \eqref{eq.JD_quad_id},  
the sum \eqref{tilde_k_sum} becomes 
$$
\sum_{\wt k \in \Z}  q^{\wt k}
\JD(k' + \wt k, k_2+k_3+x,k_1+k_4+y)
\JD(k'' + \wt k ,  k_1+k_4+y,k_2+k_3+x)
= q^{-k'} \,\delta_{k',k''}.
$$
This means that in the sum \eqref{eq.JD_index_sum1} 
we can remove the summation over $\ti\kk$ and put $k' = k'' =k$. 
Hence $J(T_j, \ti \kk,\ti\varpi) = J(T_j, \kk,\varpi)$ for $j=3, \ldots, N+2$, and 
$$
I_{\wt\CT} (\ti\varpi)= \sum_{\kk \in S} \ 
q^{k + \sum_{i=1}^{N-1} k_i } \,
 \prod_{j=3}^{N+2} J(T_j, \kk,\varpi)= I_{\CT} (\varpi),$$
 where 
 $$S = \left\{ \kk=(k, k_1, \ldots, k_{N-1})  \in \Z^{N} : k_{N-i}=0 \text { for } i=1, \ldots, r \right\} \text{  in case 1}$$
 and 
 $$S = \left\{ \kk= (k, k_1, \ldots, k_{N-1}) \in \Z^{N} : k=0, k_{N-i}=0 \text { for } i=1, \ldots, r-1 \right\} \text{  in case 2}.$$
This completes the proof of invariance of the index under the 0--2 move.
\end{proof}

%%%%%%%%%%%%%%%%%%%%%%%%%%%%%%%%%%%%%%%%%%%%%%%%%%%%%%%%%%%%%%%%%%%%%%%%%%%%
%%%%%%%%%%%%%%%%%%%%%%%%%%%%%%%%%%%%%%%%%%%%%%%%%%%%%%%%%%%%%%%%%%%%%%%%%%%%

\section{The $\calX_M^{\EP}$ class of triangulations}
\label{sec.XEPM}

\subsection{Subdivisions of the Epstein--Penner decomposition}

For a once-cusped hyperbolic 3--manifold $M$, the Epstein--Penner 
decomposition (see \cite{EP}) divides $M$ into a finite number of ideal 
hyperbolic polyhedra. This subdivision is canonical, depending only on 
the topology of the manifold, if $M$ has a single cusp.
If $M$ has $r \geq 1$ cusps, then the Epstein-Penner cell decomposition
is canonical up to the choice of a scale vector $(t_1,\dots,t_r)$
with $t_1, t_2, \dots, t_r >0$
giving the relative size of the cusps. The scale vector is well-defined up
to multiplication by a positive real number. 
For the purposes of defining our canonical set, we can choose all $t_i$ to be 
the same. 

Very often, the cells of the decomposition are all ideal tetrahedra, but other polyhedra can occur. For many applications, including the use in this paper, we need a subdivision of $M$ into ideal tetrahedra only. It is well known that every cusped 3--manifold has a decomposition into ideal tetrahedra, but one often needs more than a purely topological structure on the tetrahedra. The Epstein--Penner decomposition, coming as it does with a geometric structure, provides all of the nice geometric properties one could want. So, in the cases when the cells of the decomposition are not themselves tetrahedra, we would like to further subdivide the polyhedra into tetrahedra. However, there is no canonical way to subdivide, and it is not even clear if one can subdivide the various polyhedra in a consistent way, so that the triangulations induced on the faces of the polyhedra match when the polyhedra are glued to each other. In particular, it is still unknown whether every cusped hyperbolic 3--manifold admits a \emph{geometric triangulation} (that is, a subdivision into positive volume ideal hyperbolic tetrahedra), either constructed by further subdividing the Epstein--Penner decomposition or otherwise.

However, one can use the Epstein--Penner decomposition to produce an ideal triangulation by subdividing the ideal polyhedra, if we also allow flat tetrahedra inserted between faces of the polyhedra to bridge between incompatible triangulations of those faces. Such an ideal triangulation has a \emph{natural semi-angle structure} (see Remark \ref{rem:natural semi-angle structure}), and so by Theorem \ref{semi-angle => 1-efficient} all of these triangulations are 1--efficient.

To describe our triangulations more precisely, we use the same notation as in \cite{HRS}.

\begin{defn}
In this paper, the term \emph{polyhedron} will mean a combinatorial object obtained by removing all of the vertices from a 3-cell with a given combinatorial cell decomposition of its boundary. We further require that this can be realised as a positive volume convex ideal polyhedron in hyperbolic 3-space $\H^3$.
\end{defn}
\begin{defn}
An (ideal) \emph{polygonal pillow} or  \emph{$n$-gonal pillow} is a combinatorial object obtained by removing all of the vertices from a 3-cell with a combinatorial cell decomposition of its boundary that has precisely two faces. The two faces are copies of an $n$-gon identified along corresponding edges.
\end{defn}
\begin{defn}
Suppose that $\mathcal{P}$ is a cellulation of a 3-manifold consisting of polyhedra 
and polygonal pillows with the property that polyhedra are glued to either polyhedra or polygonal pillows, but polygonal pillows are only glued to polyhedra. Then we call $\mathcal{P}$ a \emph{ polyhedron and polygonal pillow cellulation}, or for short, a \emph{ PPP-cellulation}. 
\end{defn}
\begin{defn}\label{diagonal flip}
Let $t$ be a triangulation of a polygon. A \emph{diagonal flip} move changes $t$ as follows. First we remove an internal edge of $t$, producing a four sided polygon, one of whose diagonals is the removed edge. Second, we add in the other diagonal, cutting the polygon into two triangles and giving a new triangulation of the polygon.
\end{defn}

\begin{defn}\label{layered_triang_of_polygonal_pillow}
Let $Q$ be a polygonal pillow, with triangulations $t_-$ and $t_+$ given on its two polygonal faces $Q_-$ and $Q_+$. By a \emph{layered triangulation of $Q$, bridging between $t_-$ and $t_+$}, we mean a triangulation produced as follows. We are given a sequence of diagonal flips which convert $t_-$ into $t_+$. This gives a sequence of triangulations $t_-=L_1, L_2, \ldots, L_k=t_+$, where consecutive triangulations are related by a diagonal flip. Starting from $Q_-$ with the triangulation $t_-=L_1$, we glue a tetrahedron onto the triangulation $L_1$ so that two of its faces cover the faces of $L_1$ involved in the first diagonal flip. The other two faces together with the rest of $L_1$ produce the triangulation $L_2$. We continue in this fashion, adding one tetrahedron for each diagonal flip until we reach $L_k=t_+$, which we identify with $Q_+$.
\end{defn}

Our class of triangulations $\calX_M^{\EP}$ of $M$ consists of triangulations that are subdivisions of PPP-cellulations. Our PPP-cellulation will have polyhedra being the polyhedra of the Epstein-Penner decomposition. It also has a polygonal pillow inserted between all pairs of identified faces that have at least 4 sides. We will form our triangulations by first subdividing the ideal hyperbolic polyhedra into positive volume ideal hyperbolic tetrahedra. Secondly, for each polygonal pillow, we insert any layered triangulation that bridges between the induced triangulations of the two boundary polygons of the polyhedra to each side.

\begin{rmk}\label{rem:natural semi-angle structure}
Any triangulation produced in the way described above has a \emph{natural semi--angle structure}. This comes from the shapes of the tetrahedra as ideal hyperbolic tetrahedra. The dihedral angles of the positive volume ideal hyperbolic tetrahedra, together with 0 and $\pi$ angles for the flat tetrahedra in the layered triangulations in the polygonal pillows satisfy all of the rules for a generalised angle structure, and all angles are in $[0,\pi]$. 
\end{rmk}

The natural semi--angle structure together with Theorem \ref{semi-angle => 1-efficient} show that each triangulation in our class is 1-efficient. However, we also need to show that our class is connected under 2--3, 3--2, 0--2 and 2--0 moves, and for this we will need some extra machinery. The main tool we will use is the theory of regular triangulations of point configurations, following \cite{DeLoera:book}.

\subsection{Regular triangulations}

The concept of a regular triangulation comes from the study of triangulations of convex polytopes in $\R^n$. Here we are not dealing with topological triangulations, where tetrahedra may have self-identifications or two vertices may have multiple edges connecting them. Rather, the vertices are concrete points in $\R^n$, the edges are straight line segments in $\R^n$ and so on. In this context, a triangulation of a convex polytope is a subdivision of the polytope into concrete Euclidean simplices. Roughly speaking, a triangulation of a polytope in $\R^n$ is \emph{regular} if it is isomorphic to the lower faces of a convex polytope in $\R^{n+1}$. The following series of definitions make this idea precise.

\begin{defn}
An \emph{affine combination} of a set of points $(\bp_j)_{j\in C}$ in $\R^n$ is a sum $\sum_{j\in C} \lambda_j \bp_j$ where $ \sum_{j\in C} \lambda_j =1.$ A set of points is \emph{affinely independent} if none of them is an affine combination of the others. A \emph{$k$--simplex} is the convex hull of an affinely independent set of $k+1$ points.
\end{defn}
\begin{defn}
A \emph{point configuration} is a finite set of labelled points in $\R^n$. Let $\bA = (\bp_j)_{j\in J}$ be a point configuration with label set $J$. For $C\subset J$, the \emph{affine span of $C$ in $\bA$} is the set of affine combinations of the set of points labelled by $C$. The \emph{dimension of $C$} is the dimension of the affine span of $C$. The \emph{convex hull of $C$ in $\bA$} is the convex hull in $\R^n$ of the set of points labelled by $C$. 
\[ 
\text{conv}_\bA(C):=\left\{ \sum_{j\in C} \lambda_j \bp_j \mid \lambda_j\geq0 \text{ for all } j \in C, \text{ and } \sum_{j\in C} \lambda_j =1\right\}
\]
The \emph{relative interior of $C$ in $\bA$} is the interior of the convex hull in its affine span. 
\[ 
\text{relint}_\bA(C):=\left\{ \sum_{j\in C} \lambda_j \bp_j \mid \lambda_j >0 \text{ for all } j \in C, \text{ and } \sum_{j\in C} \lambda_j =1\right\}
\]
\end{defn}
\begin{defn}
With the above notation, if $\psi\in(\R^n)^*$ is a linear functional, then the \emph{face of $C$ in direction $\psi$} is the subset of $C$ given by
\[ 
\text{face}_\bA(C, \psi):=\left\{ j\in C \mid \psi(\bp_j)=\max_{b\in C}(\psi(\bp_b)) \right\}
\]
If $F$ is a face of $C$, we write $F\leq_\bA C$ and if in addition $F\neq C$ we write $F<_\bA C$.
\end{defn}
\begin{defn}
With the above notation, a collection $\mathcal{S}$ of subsets of $J$  is a \emph{polyhedral subdivision of $\bA$} if it satisfies the following conditions. (The elements of $\mathcal{S}$ are called \emph{cells}.)
\begin{enumerate}
\item If $C\in\mathcal{S}$ and $F\leq C$ then $F\in \mathcal{S}$. 
\item $\cup_{C\in\mathcal{S}} \text{conv}_\bA(C) \supset \text{conv}_\bA(J)$.
\item If $C\neq C'$ are two cells in $\mathcal{S}$ then $\text{relint}_\bA(C)\cap\text{relint}_\bA(C')=\emptyset$.
\end{enumerate}  
\end{defn}
\begin{rmk}
The first condition says that if some cell is in our subdivision then all faces of it are also. The second condition says that it is a subdivision of the whole convex hull of the points in $\bA$. The third condition says that the cells can only overlap with each other on their faces, not their interiors.
\end{rmk}

\begin{defn}\label{point config triangulation}
With the above notation, a \emph{triangulation of $\bA$} is a polyhedral subdivision of $\bA$ such that every cell is a simplex.
\end{defn}

\begin{defn}
Let $\bA = (\bp_j)_{j\in J}$ be a point configuration with label set $J$. Suppose that $\omega\co \bA\to\R$ is any map. The \emph{lifted point configuration} in $\R^{n+1}$ is the point configuration 
$$\bA^\omega = (\bp^\omega_j)_{j\in J} := (\bp_j, \omega(\bp_j))_{j\in J} $$
(again with label set $J$) given by adjoining to each vector an $n+1$th coordinate given by the value of $\omega$ at that point. Consider the set of faces of $\text{conv}_{\bA^\omega}(J)$. A \emph{lower face} of this convex hull is a face that is ``visible from below''. That is, $\text{face}_{\bA^\omega}(J, \psi)$ is a lower face if $\psi$ is negative 
on the last coordinate. 
\end{defn}
\begin{defn}
The \emph{regular polyhedral subdivision of $\bA$ produced by $\omega$}, denoted $\mathcal{S}(\bA,\omega)$, is the set of lower faces of the point configuration $\bA^\omega$. 
\end{defn}
Note that a face in these definitions is a set of labels for the point configuration. So the set of faces making up the polyhedral subdivision of $\bA$ is defined in terms of $\bA^\omega$, but this works because the same set of labels is used for the two point configurations. Lemma 2.3.11 of 
\cite{DeLoera:book} shows that $\mathcal{S}(\bA,\omega)$ is indeed a polyhedral subdivision of $\bA$, for every $\omega$.

\begin{defn}
A \emph{regular triangulation of $\bA$} is a regular polyhedral subdivision of $\bA$ that is a triangulation of $\bA$. 
\end{defn}

Proposition 2.2.4 of \cite{DeLoera:book} shows that every point configuration has a regular triangulation. The connection between regular triangulations of point configurations and our situation can be made via the Klein model of $\H^3$. In this model, geodesics are represented as straight lines in Euclidean space, $\E^3$. A convex ideal hyperbolic polyhedron is represented as a convex Euclidean polyhedron whose vertices lie on a sphere. Such a Euclidean polyhedron can be seen as a point configuration, with the points consisting of the vertices of the polyhedron. Note that this is a more restrictive situation than the full generality discussed in \cite{DeLoera:book} -- since all the vertices lie on a sphere, there can be no internal points, and no three points can lie on a line. These observations imply that the set of subdivisions of a convex ideal hyperbolic polyhedron into (strictly positive) volume ideal hyperbolic tetrahedra is in one-to-one correspondence with the set of triangulations of the convex Euclidean polyhedron, in the sense of Definition \ref{point config triangulation}. The bijective map between the two sets preserves the combinatorial structure of the triangulations.

\begin{defn}
Given the above discussion, we define a \emph{regular ideal triangulation} of a convex ideal hyperbolic polyhedron to be an ideal triangulation of the polyhedron whose corresponding Euclidean triangulation of the corresponding convex Euclidean polyhedron is regular.\footnote{Note that this definition has no relation to the definition of a regular ideal hyperbolic tetrahedron, in the sense of a tetrahedron with all dihedral angles being $\pi/3$. } 
\end{defn}
We are now  in a position to be able to define our class of triangulations $\calX_M^{\EP}$. 

\begin{defn}
Let $M$ be a cusped hyperbolic 3--manifold. Let $\mathcal{P}$ be the PPP--cellulation of $M$ derived from the Epstein--Penner decomposition of $M$ (choosing equal volumes for the cusps if there is more than one) by inserting polygonal pillows between any two non-triangular faces of the decomposition. The class of triangulations $\calX_M^{\EP}$ consists of all triangulations constructed via the following method:
\begin{enumerate}
\item Insert a regular ideal triangulation of each polyhedron $P$ of $\mathcal{P}$ into $P$.
\item Each polygonal pillow $Q$ has two (not necessarily distinct) polyhedra $P_-$ and $P_+$ glued to it. The regular ideal triangulations of $P_-$ and $P_+$ induce triangulations $t_-$ and $t_+$ of the two polygonal faces of $Q$. Insert into $Q$ a layered triangulation of $Q$, bridging between $t_-$ and $t_+$.
\end{enumerate}
\end{defn}
Note that although there are only finitely many regular ideal triangulations of a given convex ideal polyhedron, there may be infinitely many triangulations in $\calX_M^{\EP}$, since the layered triangulations can be arbitrarily long.

\begin{rmk}\label{Xbar^EP 1-efficient}
From the geometric construction, every triangulation of $\calX_M^{\EP}$ has a natural semi-angle structure, as in Remark \ref{rem:natural semi-angle structure}, so they are all 1-efficient by Theorem \ref{semi-angle => 1-efficient}.
\end{rmk}

\begin{defn}
The \emph{corank} of a $d$--dimensional point configuration with $n$ points is the number $n-d-1$.
\end{defn}
A point configuration has corank \emph{zero} if and only if it is affinely independent. A point configuration has corank \emph{one} if and only if it has a unique affine dependence relation. This means that there is a unique solution to 
\[
\sum_{j\in J} \lambda_j \bp_j =0 \text{ with } \sum_{j\in J} \lambda_j =0, \text{ where at least one } \lambda_j\neq0.
\]
 Uniqueness is up to scaling all $\lambda$'s by the same factor. The affine dependence divides $J$ into three subsets:
\[
J_+ := \{j\in J \mid \lambda_j>0\}, \quad J_0 := \{j\in J \mid \lambda_j=0\}, \quad J_- := \{j\in J \mid \lambda_j<0\}.
\] 
(Which is which of $J_+$ and $J_-$ is not well defined since we can multiply all of the coefficients by $-1$ to swap them.) Then $\text{relint}_\bA(J_+)\cap\text{relint}_\bA(J_-)$ is a single point, given by 
\[\sum_{j\in J_+} \lambda_j \bp_j = \sum_{j\in J_-} -\lambda_j \bp_j,
\]
where we have normalised the $\lambda$'s so that  
\[
\sum_{j\in J_+} \lambda_j = \sum_{j\in J_-} -\lambda_j  = 1.
\]

\begin{defn}
Let $\bA = (\bp_j)_{j\in J}$ be a point configuration with label set $J$. A subset of $J$ is called a \emph{circuit}, $Z$, if it is a minimal affinely dependent set (i.e. it is dependent, but every proper subset is independent).
\end{defn}
In the above discussion, $Z=J_+\cup J_-$, and $Z$ is partitioned into the two sets, $Z_+=J_+$ and $Z_-=J_-$ since if $\lambda_j=0$ in the affine dependence then it could be removed from $Z$, contradicting minimality. \\ 

In $\R^3$, five points in general position are a circuit, but there may be circuits with fewer points. Four points are a circuit if they lie in a plane, three if they lie on a line, and two if they are coincident. However, for our purposes the points are the vertices of a convex Euclidean polyhedron, so we may assume that there are no repeated points. Moreover, the points lie on a sphere, so no three lie on a line. Therefore, the only two possibilities are five points in general position, or four points that lie on a plane, as shown in Figure \ref{circuits}.

\begin{figure}[htb]
\centering
\labellist
\pinlabel $-$ at 0 106
\pinlabel $+$ at 8 19
\pinlabel $-$ at 136 38
\pinlabel $+$ at 109 119

\pinlabel $+$ at 195 77
\pinlabel $+$ at 319 94
\pinlabel $+$ at 309 22
\pinlabel $-$ at 260 4
\pinlabel $-$ at 260 137

\endlabellist
\includegraphics[width=0.6\textwidth]{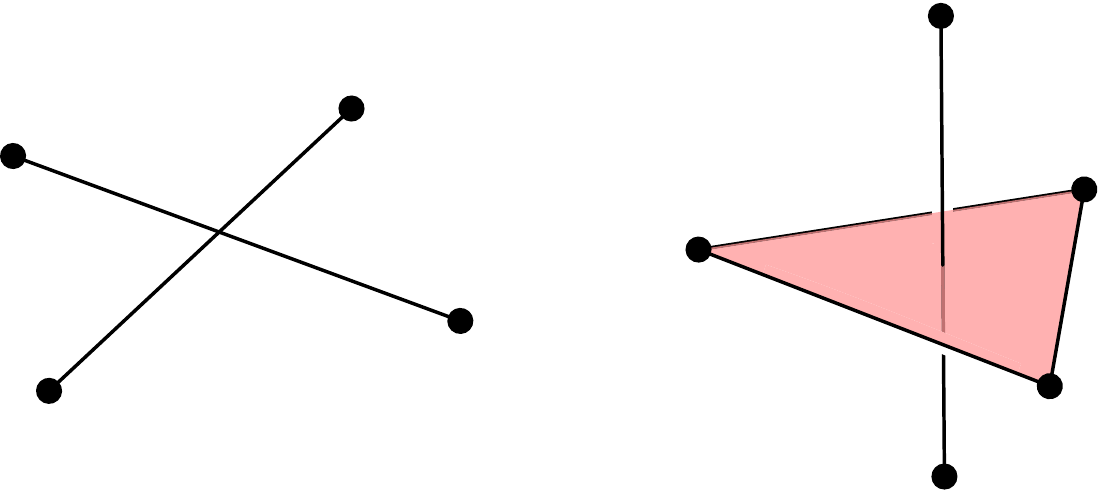}
\caption{Circuits with 4 and 5 elements.}
\label{circuits}
\end{figure}
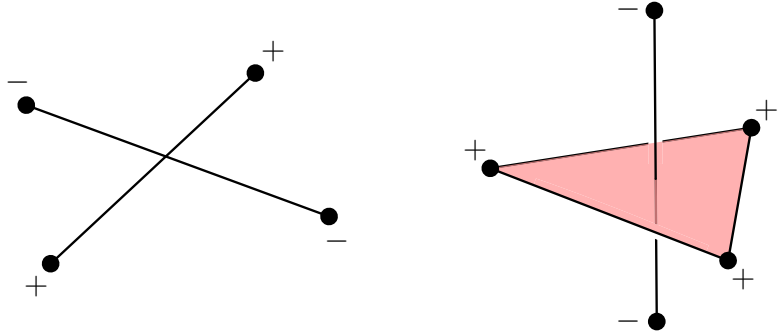

\begin{rmk}\label{our corank one configs}
For us then, the only possible corank one configurations are 
\begin{enumerate}
\item five points in general position, or \label{corank 1 5}
\item four points in a plane, or  \label{corank 1 4}
\item four points in a plane plus one point not in that plane. \label{corank 1 4+1}
\end{enumerate}
\end{rmk}

\begin{defn}\label{def:almost-triangulation}
Let $\mathcal{S}$ be a polyhedral subdivision that is not a triangulation. Then $\mathcal{S}$ is an \emph{almost--triangulation} if
\begin{enumerate}
\item all of the cells of $\mathcal{S}$ have corank at most one, and
\item all of the cells of $\mathcal{S}$ of corank one contain the same circuit.
\end{enumerate}
\end{defn}

\begin{lemma}\label{lem:almost-triangulations}
In our case, the 3--cells of an almost--triangulation are all simplices apart from one or two 3--cells. These 3--cells can have the following forms:
\begin{enumerate}
\item The convex hull of five points on a sphere, in general position,  as in the upper diagram of Figure \ref{flip_2-3}. \label{a-t_flip_2-3}
\item A 4--sided pyramid, with the base of the pyramid on a boundary face of the polyhedron, as in the upper diagram of Figure \ref{flip_2-2_external}. \label{a-t_flip_2-2_external}
\item Two 4--sided pyramids whose bases are coincident, as in the upper diagram of  Figure \ref{flip_2-2_internal}. \label{a-t_flip_2-2_internal}  
\end{enumerate}
\end{lemma}
\begin{proof}
This follows immediately from Definition \ref{def:almost-triangulation} and Remark \ref{our corank one configs}. 
\end{proof}

\begin{figure}[htb]
\centering

\labellist
\small
\pinlabel refinement at 60 156
\pinlabel refinement at 225 156
\pinlabel flip at 139 86
\endlabellist

\includegraphics[width=0.44\textwidth]{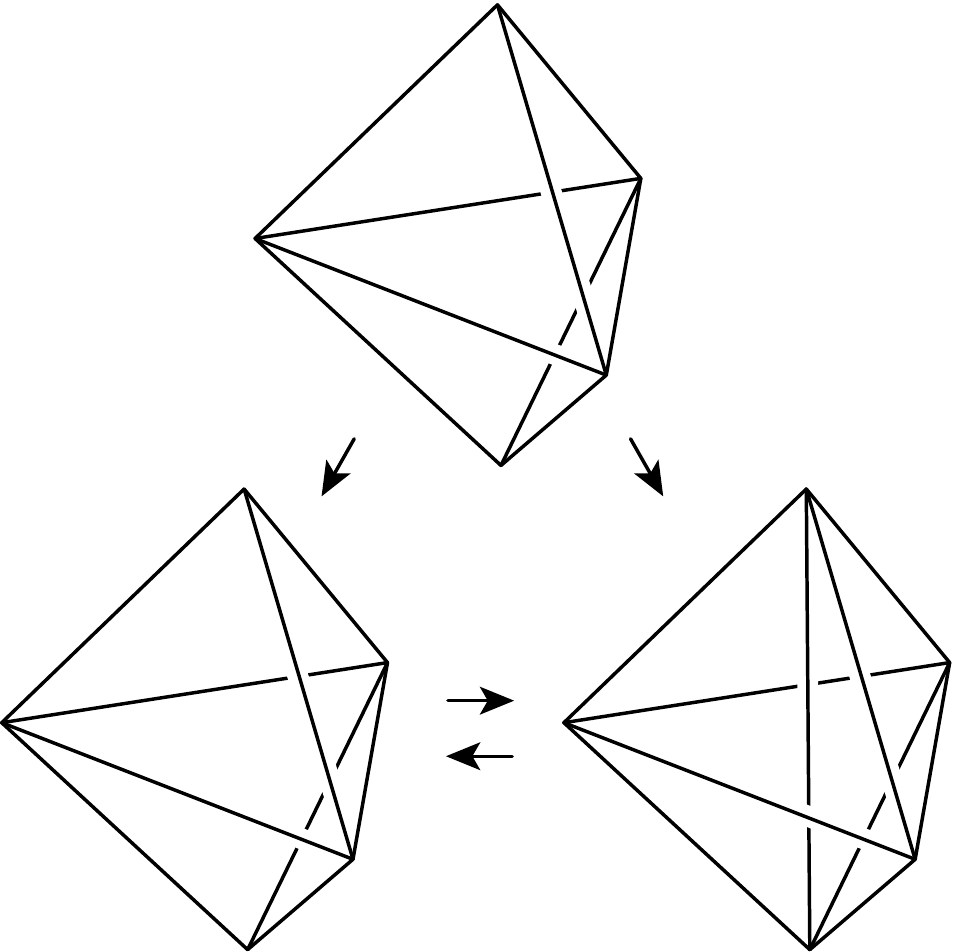}
\caption{The almost--triangulation in case (\ref{a-t_flip_2-3}) of Lemma \ref{lem:almost-triangulations}, and its two refinements to triangulations. Note that although the top and lower left pictures are identical, the top is to be interpreted as a single 3--cell, while the lower left shows two tetrahedra meeting in a triangle. The associated flip between the two triangulations is a 2--3 flip.}
\label{flip_2-3}
\end{figure}

\begin{figure}[htb]
\centering

\subfloat[The almost--triangulation in case (\ref{a-t_flip_2-2_external}) of Lemma \ref{lem:almost-triangulations}, and its two refinements to triangulations.]{

\labellist
\small
\pinlabel refinement at 48 114
\pinlabel refinement at 207 114
\pinlabel flip at 123 54
\pinlabel $v_1$ at 72 187
\pinlabel $v_2$ at 80 129
\pinlabel $v_3$ at 191 145
\pinlabel $v_4$ at 167 196

\endlabellist

\includegraphics[width=0.46\textwidth]{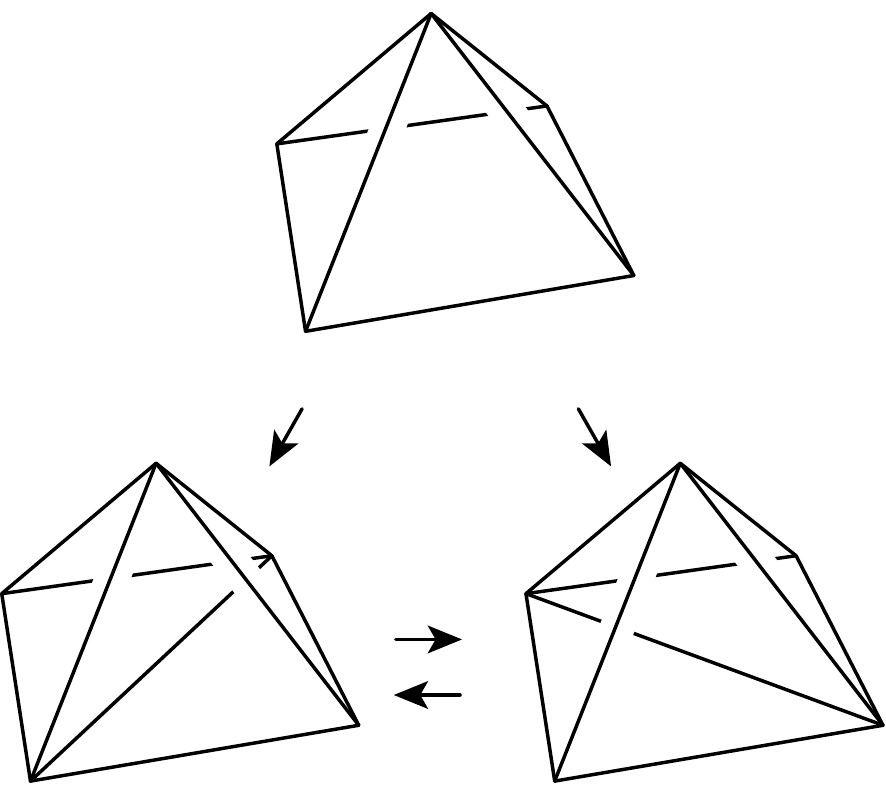}
\label{flip_2-2_external}}
\quad
\subfloat[The almost--triangulation in case (\ref{a-t_flip_2-2_internal}) of Lemma \ref{lem:almost-triangulations}, and its two refinements to triangulations.]{

\labellist
\small
\pinlabel refinement at 48 121
\pinlabel refinement at 207 121
\pinlabel flip at 123 61
\pinlabel $v_1$ at 72 194
\pinlabel $v_2$ at 80 136
\pinlabel $v_3$ at 191 152
\pinlabel $v_4$ at 167 201
\endlabellist

\includegraphics[width=0.46\textwidth]{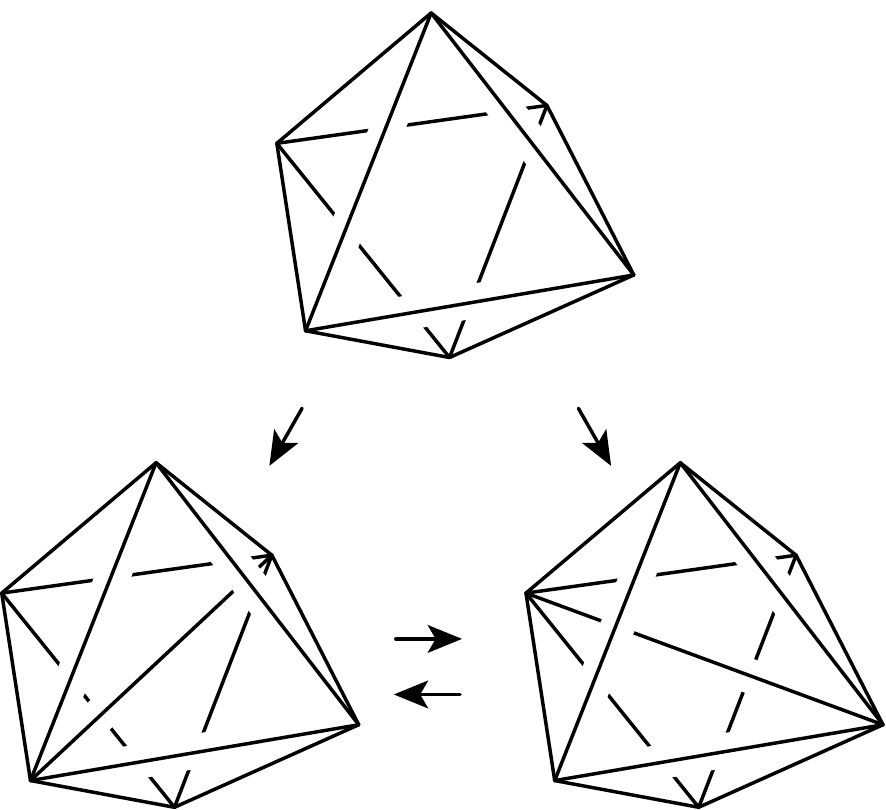}
\label{flip_2-2_internal}}

\caption{Almost--triangulations in cases (\ref{a-t_flip_2-2_external}) and (\ref{a-t_flip_2-2_internal}) of Lemma \ref{lem:almost-triangulations}, and the associated 2--2 flips. In all diagrams, vertices $v_1$ through $v_4$ are coplanar.}
\label{2-2_flips}
\end{figure}
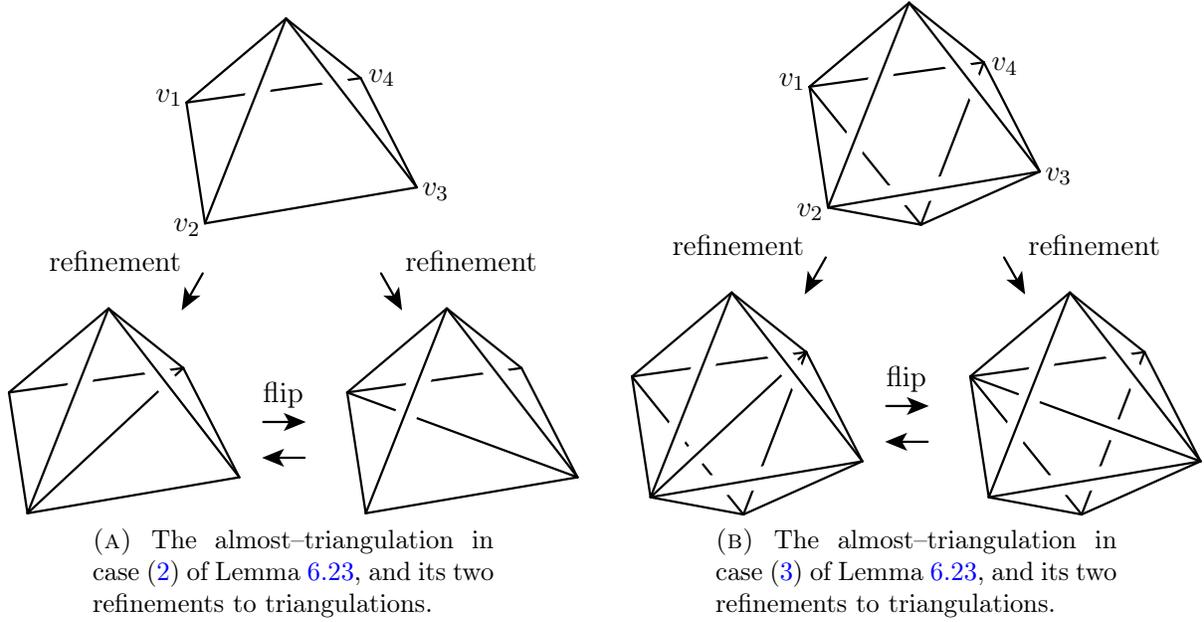

\begin{defn}
Let $\mathcal{S}$ and $\mathcal{S}'$ be two polyhedral subdivisions of a point configuration $\bA$. Then \emph{$\mathcal{S}$  is a refinement of $\mathcal{S}'$} if for each $C\in \mathcal{S}$, there is a $C'\in\mathcal{S}'$ with $C\subset C'$.
\end{defn}

\begin{lemma}[Corollary 2.4.6 of \cite{DeLoera:book}]
Every almost--triangulation has exactly two proper refinements, which are both triangulations.
\end{lemma}

\begin{defn}
Two triangulations of the same point configuration are \emph{connected by a flip supported on the almost--triangulation $\mathcal{S}$} if they are the only two triangulations refining $\mathcal{S}$.
\end{defn}

\begin{defn}
We call the flips associated to the three possible almost--triangulation types listed in Lemma \ref{lem:almost-triangulations} the \emph{2-3 flip}, \emph{external 2--2 flip}, and \emph{internal 2--2 flip} respectively. See Figures \ref{flip_2-3} and \ref{2-2_flips}.
\end{defn}

\begin{defn}
The \emph{flip graph} of the point configuration $\bA$ is the graph whose vertices are the triangulations of $\bA$ and whose edges are triangulations connected by flips.
\end{defn}

We use the following result, due to Gelfand, Kapranov and Zelevinsky, and given as Corollary 5.3.14 in \cite{DeLoera:book}.

\begin{thm}[Gelfand, Kapranov and Zelevinsky \cite{GKZ}]\label{GKZ}
Let $\bA$ be a point configuration. The subgraph of the flip graph induced by all regular triangulations of $\bA$ that use the same vertices is connected.
\end{thm}

In our case, since the vertices lie on a sphere, all vertices are used in every triangulation, so this says that we can get from any regular triangulation  of the polyhedron to any other by performing flips.\\

\begin{rmk}\label{retriangulate strategy}
Our strategy for connecting two triangulations $\tri_1, \tri_2\in\calX_M^{\EP}$ is as follows. Both triangulations consist of regular triangulations of the polyhedra of the PPP--cellulation, together with layered triangulations in the polygonal pillows between them.
\begin{enumerate}
\item Use Theorem \ref{GKZ} on each polyhedron, to change the triangulation of each polyhedron in $\tri_1$ into the corresponding triangulation of  the polyhedron in $\tri_2$. This step may alter the triangulations of the polygonal pillows as well. \label{retriangulate step 1}
\item Change the triangulation in each polygonal pillow of the resulting triangulation so that it matches with the corresponding triangulation in $\tri_2$. This step does not alter the triangulations of the polyhedra. \label{retriangulate step 2}
\end{enumerate}
\end{rmk}

\subsection{Interpreting flip moves using 2--3 moves}

In order to carry out step (\ref{retriangulate step 1}) of our plan in Remark \ref{retriangulate strategy}, we will interpret the flip moves in terms of the 2--3 and 3--2 moves allowed in Theorem \ref{thm.00}. 
\begin{enumerate}
\item First, consider a 2--3 flip in one polyhedron. This is simply a 2--3 move. Since the triangulations on either side of the move are in $\calX_M^{\EP}$, they are both 1--efficient by Remark \ref{Xbar^EP 1-efficient}. Therefore the triangulations have the same index by Theorem \ref{thm.00}.

\item Second, consider an external 2--2 flip. See Figure \ref{2-2move_external}. The base of the pyramid is on a face of the polyhedron that is glued to a polygonal pillow. The two triangles on the base of the pyramid are glued to either a single tetrahedron in the pillow, or two tetrahedra in either the polygonal pillow or the polyhedron on the other side of the polygonal pillow.
\begin{itemize}
\item If the two triangles are glued to a single tetrahedron then we are in the situation shown in the top diagram of Figure \ref{2-2move_external}, and we can perform a 3--2 move, which performs the flip to the polyhedron, and removes the single tetrahedron from the layered triangulation of the polygonal pillow. 
\item Otherwise, we perform a 2--3 move, which performs the flip to the polyhedron, and adds a flat tetrahedron to the layered triangulation of the polygonal pillow. (Note that we could use only this move, even in the previous case; the difference between the two options is a 0--2 move.)
\end{itemize}
Once again, the triangulations on either side of the move are in $\calX_M^{\EP}$, so they are both 1--efficient and the triangulations have the same index.

\item Lastly, consider an internal 2--2 flip. See Figure \ref{2-2move_internal}. We can perform a 2--3 move followed by a 3--2 move, which together perform the flip. Since the four vertices in the circuit are coplanar, at the intermediate step we introduce a flat tetrahedron. The intermediate triangulation is not in $\calX_M^{\EP}$, since it includes a flat tetrahedron in a polyhedron, and so the polyhedron does not have a regular triangulation. However, the intermediate triangulation still has a natural semi--angle structure in the obvious way, and so it is 1--efficient, and again the three triangulations involved all have the same index.
\end{enumerate}

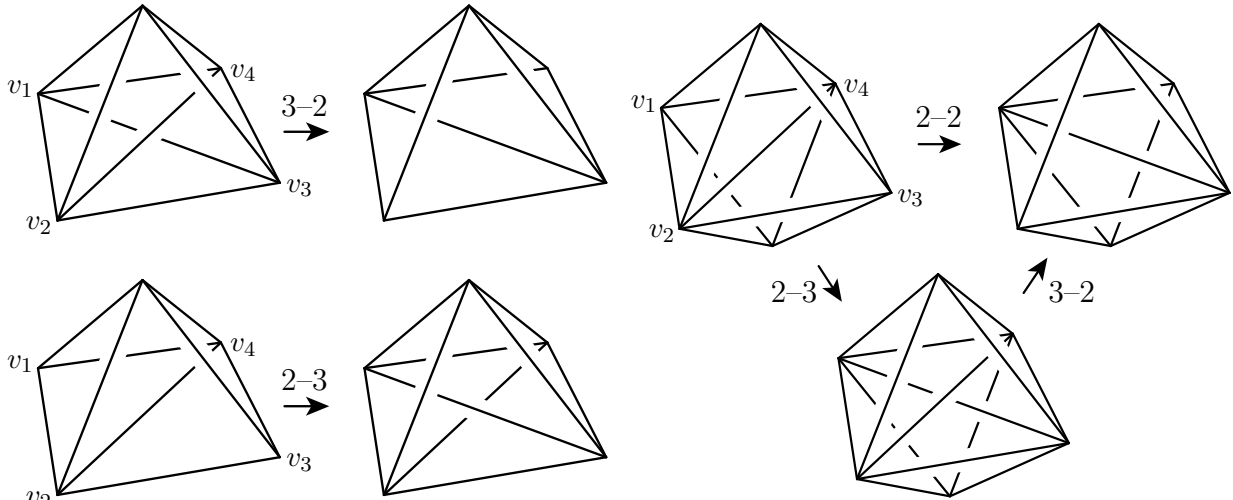
\begin{figure}[htb]
\centering
\subfloat[Two possible ways to perform an external 2--2 move, depending on whether or not there is a suitable flat tetrahedron in the polygonal pillow that the base of the pyramid is glued to.]{
\labellist
\pinlabel 3--2 at 114 167
\pinlabel 2--3 at 114 50
\small
\pinlabel $v_1$ at -7 174
\pinlabel $v_2$ at 1 116
\pinlabel $v_3$ at 112 132
\pinlabel $v_4$ at 88 181

\pinlabel $v_1$ at -7 57
\pinlabel $v_2$ at 1 -1
\pinlabel $v_3$ at 112 15
\pinlabel $v_4$ at 88 64
\endlabellist
\includegraphics[width=0.46\textwidth]{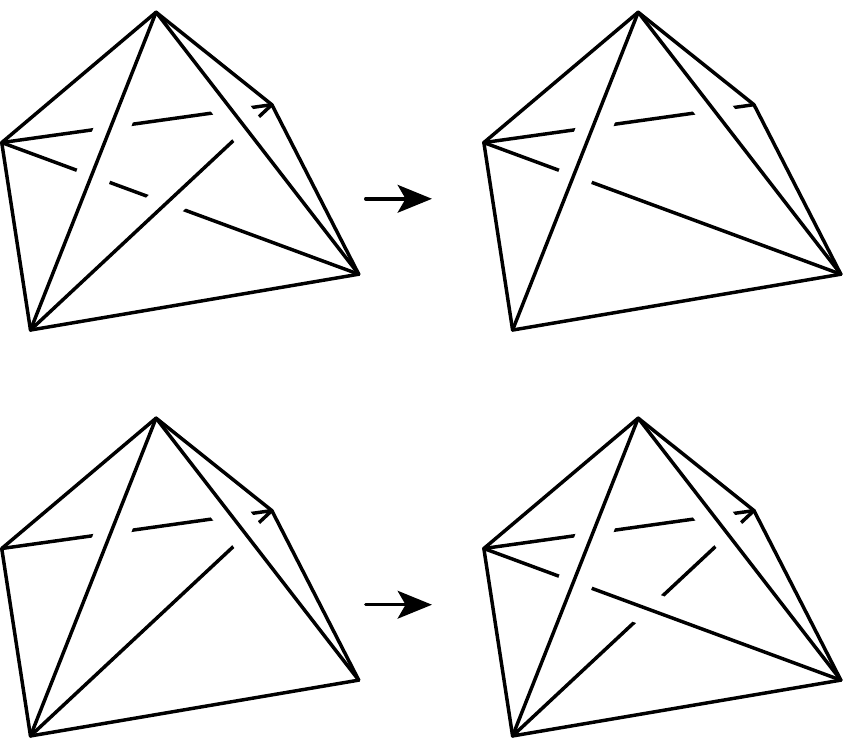}
\label{2-2move_external}}
\quad
\subfloat[An internal 2--2 flip, obtained by performing a 2--3 move followed by a 3--2 move. At the intermediate step we get a flat tetrahedron with vertices $v_1$ through $v_4$.]{
\labellist
\pinlabel 2--2 at 124.5 170
\pinlabel 2--3 at 60 92
\pinlabel 3--2 at 184 92
\small
\pinlabel $v_1$ at -7 176
\pinlabel $v_2$ at 1 118
\pinlabel $v_3$ at 112 134
\pinlabel $v_4$ at 88 183

\endlabellist
\includegraphics[width=0.46\textwidth]{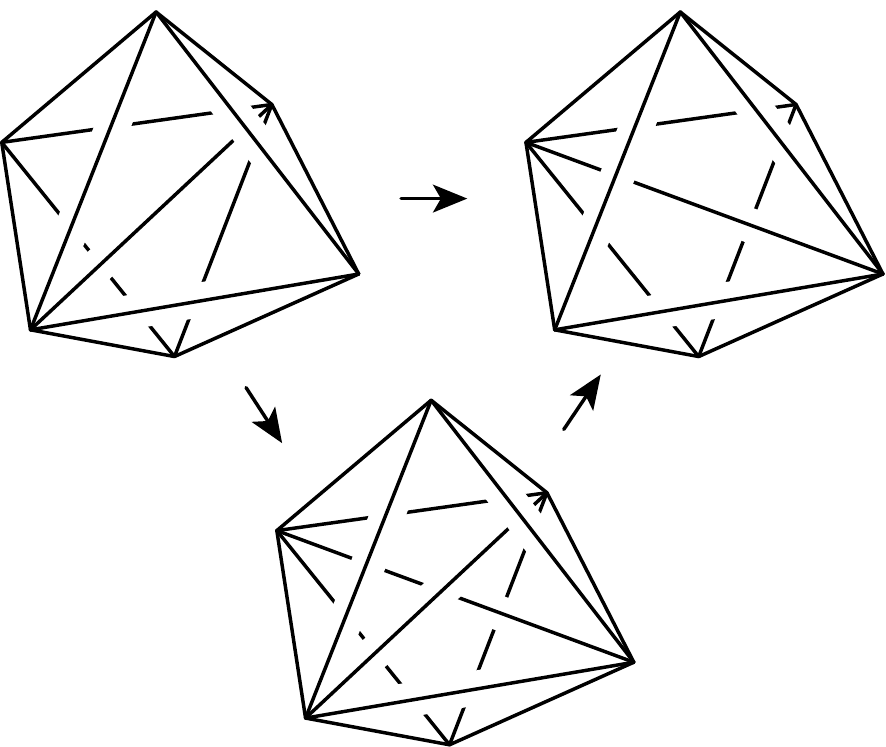}
\label{2-2move_internal}}
\caption{The possible 2--2 moves on a triangulation of a point configuration, realised using 2--3 and 3--2 moves. Again in each case, vertices $v_1$ through $v_4$ are coplanar.}
\label{2-2 moves}
\end{figure}

The arguments so far show that we can achieve step (\ref{retriangulate step 1}) of Remark \ref{retriangulate strategy}. Now we need to deal with step (\ref{retriangulate step 2}).

\subsection{Moving paths in the 1--skeleton of the associahedron using 2--3 and 0--2 moves}
\label{sub.moving.paths}

In order to carry out step (\ref{retriangulate step 2}) of our plan in Remark \ref{retriangulate strategy}, we need to modify the layered triangulations in the polygonal pillows. We will do this using the 2--3, 3--2, 0--2 and 2--0 moves allowed in Theorems \ref{thm.00} and \ref{thm.20move.index}.\\

Having completed step (\ref{retriangulate step 1}), we have two triangulations, $\tri_1$ and $\tri_2$, which agree on the polyhedra but may differ in the polygonal pillows. Let $Q$ be a polygonal pillow. Since $\tri_1$ and $\tri_2$ agree on the polyhedra glued to either side of $Q$, they agree on the triangulations of the polygonal faces $Q_+$ and $Q_-$. Call these triangulations $t_+$ and $t_-$.\\

The next well-known theorem concerns the identification of the
set of triangulations of an $n$-gon with the vertices of the 
{\em associahedron} $K_{n-2}$, and the set of geometric bistellar flips
with the edges of $K_{n-2}$. The associahedron $K_{n-2}$ was introduced 
by Stasheff \cite{Stasheff:1963}. An identification of $K_{n-2}$ with a convex
polytope in Euclidean space was given in the appendix to \cite{Stasheff:Luminy}.
The cellular decomposition of the polytope $K_{n-2}$ (and in particular,
its 2-skeleton) is discussed at length in the above references and also in 
\cite{Loday}. The fact that the 2-dimensional faces of the associahedron 
are squares and pentagons also follows from MacLane's coherence theorem
\cite{MacLane}. A vast generalization of regular triangulations of point
configurations was studied by Gelfand-Kapranov-Zelevinsky, and in 
\cite[Sec.7.3]{GKZ} it is explained how to identify the secondary polytope
of 2-dimensional configurations with the associahedron.

\begin{thm}\label{associahedron properties}
The set of triangulations of an $n$--gon and the set of diagonal flips connecting them correspond to the vertices and edges of a convex polytope called the \emph{associahedron}. The set of 2--cells of the associahedron consists of squares and pentagons. Each square corresponds to two commuting diagonal flips (as in Definition \ref{diagonal flip}) on two 4-gons whose interiors are disjoint. Each pentagon corresponds to a \emph{pentagon relation} between the five triangulations of a pentagon. See Figure \ref{associahedron 2-cells}.
\end{thm}

\begin{figure}[htb]
\centering
\subfloat[A square 2--cell of the associahedron.]{
\includegraphics[width=0.43\textwidth]{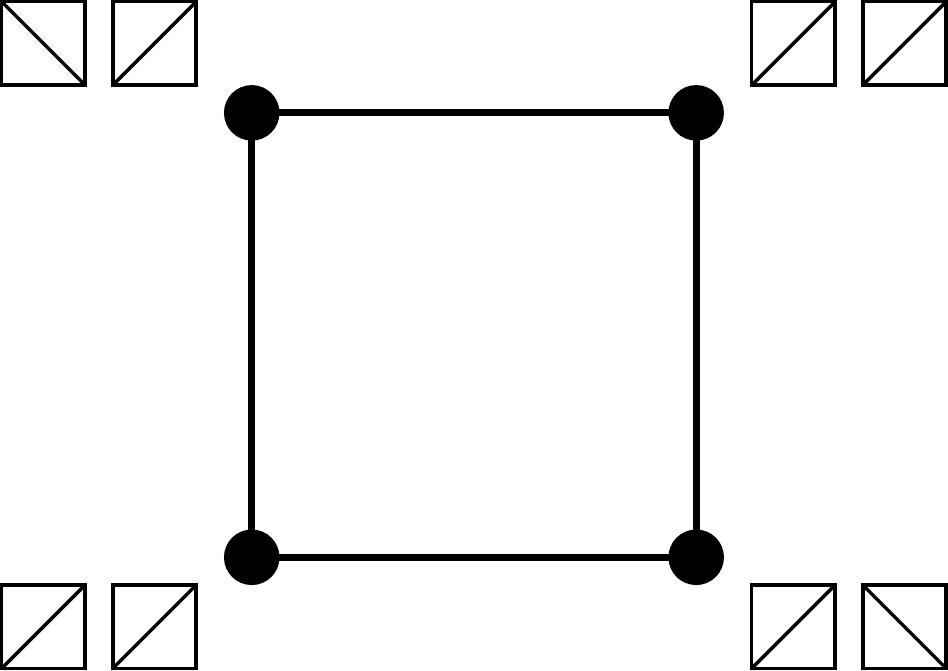}
\label{associahedron_square}}
\quad
\subfloat[A pentagon 2--cell of the associahedron.]{
\includegraphics[width=0.49\textwidth]{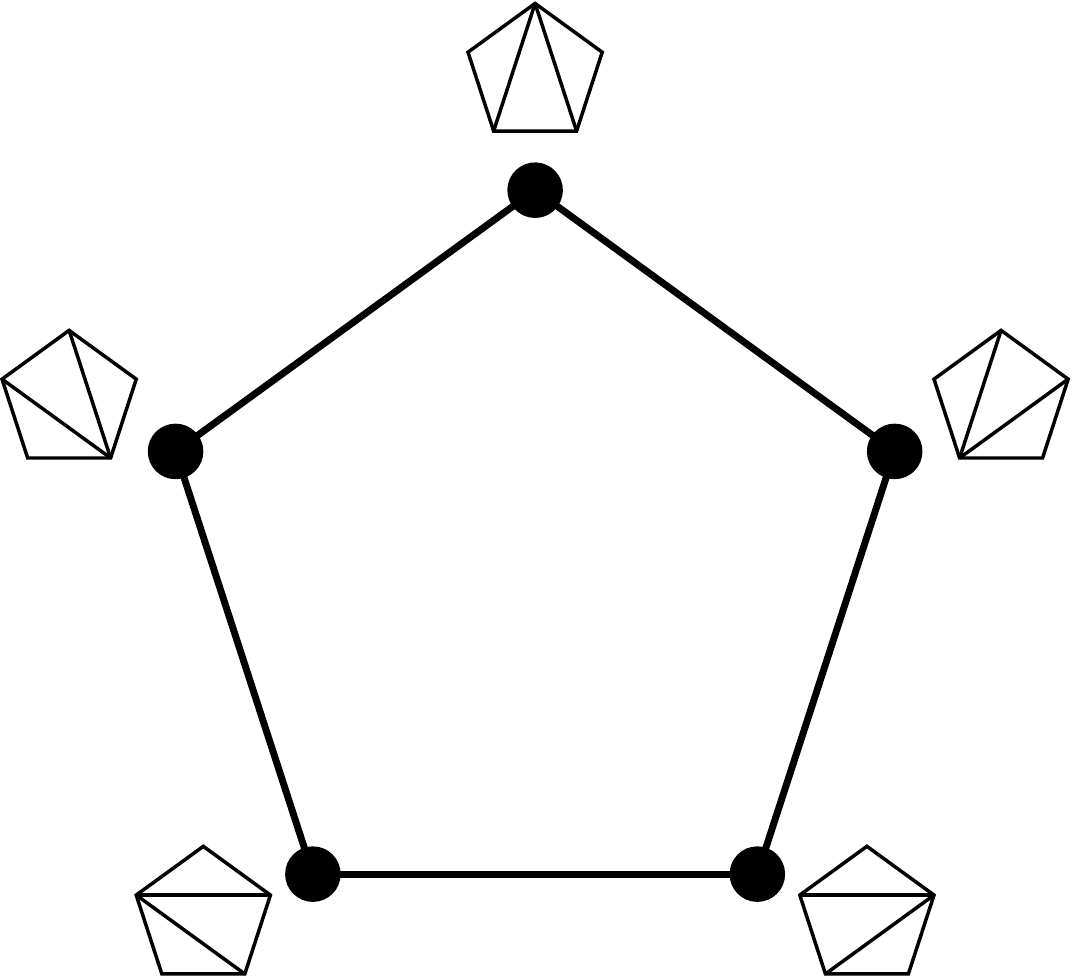}
\label{associahedron_pentagon}}
\caption{The 2--cells of the associahedron.}
\label{associahedron 2-cells}
\end{figure}

Our two layered triangulations of the polygonal pillow $Q$ can be represented as two paths $p_1$ and $p_2$ in the 1--skeleton of the associahedron that both start at the vertex corresponding to $t_-$ and end at the vertex corresponding to $t_+$. Since the fundamental group of the 2--skeleton $X^{(2)}$ of a CW complex
$X$ satisfies $\pi_1(X^{(2)})=\pi_1(X)$, it follows that the 2--skeleton
of a convex polytope is simply connected, and every loop along the 1-skeleton
can be trivialized by moving pieces of it across the 2-cells.
It follows that the 
2--skeleton of the  associahedron is simply connected, and so we can homotope 
$p_1$ to $p_2$, fixing endpoints,  by moving the path across some finite sequence of the 2--cells. To be precise, we can homotope $p_1$ to $p_2$ with a combination of the following moves and their inverses.

\begin{enumerate}
\item Remove a \emph{backtracking}, that is, replace a segment of a path of the form $\ldots,t_1, t_2, t_1,\ldots$ with $\ldots, t_1, \ldots$, where $t_1$ and $t_2$ are related by a diagonal flip. \label{associahedron isotope 2-0}
\item Move a path that follows two consecutive sides of a square 2--cell to follow the other two sides. \label{associahedron isotope 2-2}
\item Move a path that follows three consecutive sides of a square 2--cell to follow the other side. \label{associahedron isotope 3-1}
\item Delete a part of a path that goes around all four sides of a square 2--cell. \label{associahedron isotope 4-0}
\item Move a path that follows three consecutive sides of a pentagon 2--cell to follow the other two sides. \label{associahedron isotope 3-2}
\item Move a path that follows four consecutive sides of a pentagon 2--cell to follow the other side. \label{associahedron isotope 4-1}
\item Delete a part of a path that goes around all five sides of a pentagonal 2--cell. \label{associahedron isotope 5-0}
\end{enumerate}

These moves can be achieved as follows.
\begin{itemize}
\item Move (\ref{associahedron isotope 2-0}) corresponds to performing a 2--0 move on the triangulation. Since the two triangulations are both in $\calX_M^{\EP}$, they are both 1--efficient by Remark \ref{Xbar^EP 1-efficient}. Therefore the two triangulations have the same index by Theorem \ref{thm.20move.index}.
\item Move (\ref{associahedron isotope 2-2}) does not change the layered triangulation at all, it only swaps the order in which we add two non-overlapping flat tetrahedra to the layering.
\item Move (\ref{associahedron isotope 3-1}) can be made by applying move (\ref{associahedron isotope 2-2}) followed by move (\ref{associahedron isotope 2-0}).
\item Move (\ref{associahedron isotope 4-0}) can be made by applying move (\ref{associahedron isotope 2-2}) followed by move (\ref{associahedron isotope 2-0}) twice.
\item Move (\ref{associahedron isotope 3-2}) corresponds to performing a 3--2 move on the triangulation. As before, since the two triangulations are both in $\calX_M^{\EP}$, they are both 1--efficient by Remark \ref{Xbar^EP 1-efficient}. Therefore the two triangulations have the same index by Theorem \ref{thm.00}(b).
\item Move (\ref{associahedron isotope 4-1}) can be made by applying move (\ref{associahedron isotope 3-2}) followed by move (\ref{associahedron isotope 2-0}).
\item Move (\ref{associahedron isotope 5-0}) can be made by applying move (\ref{associahedron isotope 3-2}) followed by move (\ref{associahedron isotope 2-0}) twice.
\end{itemize}

\begin{proof}[Proof of Theorem \ref{thm.2}]
We have shown in this section that given any two triangulations in $\calX_M^{\EP}$, the steps in Remark \ref{retriangulate strategy} can be made using 2--3, 3--2, 0--2 and 2--0 moves which preserve the index. Thus the entire class of triangulations has the same index. Since $\calX_M^{\EP}$ depends only on the topology of $M$, we can take the index of any of these triangulations for the value of the index for the manifold, and this depends only on the topology of $M$.
\end{proof}

\begin{proof}[Proof of Theorem \ref{thm.3}]
Fix a cusped hyperbolic manifold $M$ and $\CT \in \calX_M^{\EP}$. $\CT$
consists of two types of tetrahedra: the ones that subdivide the ideal
hyperbolic cells of the Epstein-Penner cell decomposition of $M$, and the
ones that are parts of the pillows. The former have geometric shapes (i.e.,
shapes that are in the upper half plane), and the latter have real 
non-degenerate shapes. By construction, the shapes $Z_{\CT}$ 
satisfy the gluing equations of $\CT$, proving part (a). Their arguments
also satisfy the gluing equations, proving part (c).

Also by construction, the shapes $Z_{\CT}$ and $Z_{\CT'}$ are related by
2--3, 3--2, 0--2 and 2--0 moves if the corresponding triangulations 
$\CT$ and $\CT'$ are related by the same moves. 
\end{proof}

%%%%%%%%%%%%%%%%%%%%%%%%%%%%%%%%%%%%%%%%%%%%%%%%%%%%%%%%%%%%%%%%%%%%%%%%%%%%%
%%%%%%%%%%%%%%%%%%%%%%%%%%%%%%%%%%%%%%%%%%%%%%%%%%%%%%%%%%%%%%%%%%%%%%%%%%%%%

\section{Computations}
\lbl{sec.computation}

\subsection{How to compute the coefficients of a $q$-series}
\lbl{sub.overview.compute}

In this section we will explain a general method to compute the
coefficients of a $q$-series which is given by a multi-dimensional sum
of some basic $q$-series. The idea is simple, but the effective aspects
of it are tricky and were explained to the first author by D. Zagier. 
The method is applied in forthcoming work \cite{GV} which computes the 
coefficients of the stabilization of the coloured Jones polynomial of 
an alternating knot \cite{GL2}.

Fix a 1-efficient ideal triangulation $\CT$ with $n$ tetrahedra of a 1-cusped 
hyperbolic manifold $M$ and an oriented  multi-curve $\varpi$ on $\bd M$.
The index 
$I_{\CT}(\varpi)(q) \in \BZ((q^{1/2}))$ is given by a convergent 
$(n-1)$-dimensional sum over the integers of a summand that
depends on the angle structure equation matrices of $\CT$; see Equation 
\eqref{eq.indT}. The summand of $I_{\CT}(\varpi)$ is a product of tetrahedron 
indices (one per tetrahedron of $\CT$) of linear forms in 
$\kk =(k_1, \ldots, k_{n-1}) \in \BZ^{n-1}$ and the turning number vectors $(a_{\varpi}|b_{\varpi})$ of $\varpi$.

The building block of the summand is the tetrahedron
index $I_\Delta$, whose degree (i.e., minimum degree with respect to $q$)
is a piecewise quadratic function on $\R^2$; see Section \ref{sub.degI}.
By a {\em piecewise quadratic function} on $\BR^m$, we mean that there exists a partition $\calF$
of $\BR^{m}$ into a finite number of chambers whose boundaries are 
rational polyhedral cones such that the restriction of the function to 
each chamber is given by a quadratic polynomial. 
It follows that for fixed $\varpi$, the degree $\d(\varpi,\kk)$ of the summand 
in \eqref{eq.indT} is a piecewise quadratic function defined on a partition 
$\calF_{\CT}$ of $\BR^{n-1}$. In a later publication, we will explain how 
to compute $\calF_{\CT}$ directly from $\CT$. 
A priori, $\calF_{\CT}$ need not be a fan in $\BR^{n-1}$ as the cones in 
the partition may not be convex.

Since $\CT$ is 1-efficient, it follows that its index $I_{\CT}(\varpi)(q)$
is a convergent series. In other words, the degree of the summand is
a proper function on $\BZ^{n-1}$. Thus, for fixed $\varpi$ and 
every half-integer $N$, the set $\{\kk \in \BZ^{n-1} \,: \d(\varpi,\kk) \leq N \}$
is finite. To compute the index of $\CT$, we need 
to compute bounds on this finite set. Since $\d$ is convex and 
piecewise quadratic, it follows that to bound $\d(\varpi,\kk)$, it suffices
to bound the restriction of $\d(\varpi,\kk)$ to an arbitrary ray $\rho = \{ k' \rho_0: k' \in \BN\}$ of $\calF_T$.
Now $\d(\varpi,k' \rho_0)$ is a quadratic function of $k' \in \BN$, and 
we obtain sharp bounds for 
$k_i$ typically of the form $k_i=O(\sqrt{N})$
and {\em exceptionally} of the form $k_i=O(N)$. The latter happens when
$\d$ has linear growth on some ray of $\calF_T$. These directions of 
linear growth (also observed in \cite{GL2} in the context of stabilization
of the coloured Jones function) are computationally costly. For an example,
see the case of the knot $6_1$ discussed below. 

The bounds for $k_i$ discussed above are rigorous and sharp, and 
work well for $n=2$ and $n=3$ tetrahedra. However, they
quickly become inefficient when $n$ increases (e.g. $n=9$).
The better way to proceed for larger $n$, as was explained to us 
by D. Zagier and is applied successfully in \cite{GV}, is to use 
{\em iterated summation}. In the examples shown below for $n=2$ and $n=3$
iterated summation is not needed. For simplicity, we focus on 1-cusped
manifolds and their index for $[\varpi]=0 \in H_1(\bd M;\Z)$.
The data presented below are available from \cite{Ga:data}.

\begin{rmk}
\lbl{rem.same.index}
If two 1-cusped hyperbolic 3-manifolds $M$ and $M'$ have 1-efficient 
ideal triangulations $\CT$ and $\CT'$ with equal angle structure matrices, 
then $I_{\CT}(0)(q)=I_{\CT'}(0)(q)$.
For example, $M$ and $-M$ have
such triangulations, where $-M$ denotes the orientation reversed cusped
hyperbolic manifold. Also, $M_{\phi}$ and $M_{-\phi}$ have such
triangulations for a pseudo-Anosov
homeomorphism $\phi$ of a once punctured torus, where 
$M_{\phi}$ denotes the mapping torus of $\phi$.
\end{rmk}

\subsection{The index of the $4_1$ knot complement}
\lbl{sub.index41}

The default {\tt SnapPy} triangulation $\CT$ of the $4_1$ knot complement uses 2 regular
ideal tetrahedra and coincides with the Epstein-Penner decomposition,
thus $\calX_{4_1}^{\EP}=\{\CT\}$. {\tt SnapPy} gives the angle structure 
matrices of $\CT$:
$$
\overline{\mb A} =
\begin{pmatrix}
2 & 2 \\
0 & 0 
\end{pmatrix} \qquad
\overline{\mb B} =
\begin{pmatrix}
1 & 1 \\ 
1 & 1 
\end{pmatrix} \qquad
\overline{\mb C} =
\begin{pmatrix}
 0 & 0 \\
 2 & 2 
\end{pmatrix}.
$$
{\tt SnapPy} also gives an additional two rows which correspond to the meridian and longitude equations; we will
not use these in our examples so we omit them here. We eliminate $\overline{\mb B}$ to obtain $\rm A$ and $\rm B$. We choose our basic edge set $\mathcal B = \{e_1\}$ and excluded edge set $\mathcal X = \{e_2\}$. We can therefore ignore the rows of our matrices corresponding to $e_2$, and write 

$$
\mb A'= 
\left(
\begin{array}{cc}
 1 & 1 \\
\end{array}
\right)
\qquad 
\mb B'= 
\left(
\begin{array}{cc}
 -1 & -1 \\
\end{array}
\right)
\qquad 
\nu'=
\left(
\begin{array}{c}
 0 \\
\end{array}
\right).
$$

We obtain the angle structure equations  $$
\mb A' \a + \mb B' \gamma = \pi \nu'. 
$$
Equation \eqref{eq.indT} gives that
$$
I_{4_1}(0)(q)=\sum_{k \in \BZ} \ID(k,k)(q)^2 \in \BZ[[q]] \, .
$$
Equation \eqref{eq.degID} implies that
the degree $\d(k)$ of the summand is given by the piecewise quadratic polynomial
$$
\d(k)=2 \d(k,k)= 2 k^2 + |k| \, .
$$
Its corresponding fan
$\calF_{\CT}$ in $\BR$ consists of two rays, corresponding to the columns of the matrix
$$
\left(
\begin{array}{cc}
 1 & -1
\end{array}
\right)\, .
$$
Here, a column vector $v$ of a matrix, spans the ray $\BR_+ v$.
It follows that for every natural number $N$ we have
$$
I_{4_1}(0)(q) + O(q)^{N+1}=
\sum_{k=\lfloor 1/4 (-1 - \sqrt{1 + 8 N}) \rfloor}^{\lceil 
1/4 (-1 + \sqrt{1 + 8 N}) \rceil }
\ID(k,k)(q)^2 + O(q)^{N+1} \, .
$$
To compute $\ID(k,k)(q) +  O(q)^{N+1}$, we use its definition \eqref{eq.ID}
and truncate the $n$-summation as follows
\be
\lbl{eq.IDtruncated}
\ID(k,k)(q) +  O(q)^{N+1}=
\sum_{n=\lfloor 1/2 (-1 + 2 k - \sqrt{1 - 4 k + 8 k^2 + 8 N}) \rfloor }^{
\lceil 1/4 (-1 + \sqrt{1 + 8 N}) \rceil } 
(-1)^n \frac{q^{\frac{1}{2}n(n+1)
-\left(n+\frac{1}{2}k\right)k}}{(q)_n(q)_{n+k}} + O(q)^{N+1} \, .
\ee
Putting everything together, the first 100 coefficients of $I_{4_1}(0)(q)$
are given by

\vspace{0.3cm}
{\tiny
\begin{math}
1-2 q-3 q^2+2 q^3+8 q^4+18 q^5+18 q^6+14 q^7-12 q^8-52 q^9-106 q^{10}-164 q^{11}-209 q^{12}-212 q^{13}-141 q^{14}+14 q^{15}+309 q^{16}+714 q^{17}+1249 q^{18}+1824 q^{19}+2401 q^{20}+2794 q^{21}+2898 q^{22}+2434 q^{23}+1256 q^{24}-918 q^{25}-4186 q^{26}-8712 q^{27}-14394 q^{28}-21046 q^{29}-28184 q^{30}-35094 q^{31}-40740 q^{32}-43732 q^{33}-42508 q^{34}-35068 q^{35}-19524 q^{36}+6288 q^{37}+43942 q^{38}+95026 q^{39}+159698 q^{40}+237774 q^{41}+326680 q^{42}+422880 q^{43}+519595 q^{44}+608636 q^{45}+677761 q^{46}+713352 q^{47}+697625 q^{48}+611956 q^{49}+434572 q^{50}+144616 q^{51}-279773 q^{52}-856288 q^{53}-1599627 q^{54}-2515906 q^{55}-3602521 q^{56}-4842516 q^{57}-6203552 q^{58}-7632646 q^{59}-9054429 q^{60}-10367858 q^{61}-11443874 q^{62}-12125534 q^{63}-12226286 q^{64}-11535062 q^{65}-9815935 q^{66}-6820480 q^{67}-2289703 q^{68}+4024698 q^{69}+12355340 q^{70}+22887604 q^{71}+35751602 q^{72}+50979996 q^{73}+68497913 q^{74}+88071340 q^{75}+109297633 q^{76}+131547294 q^{77}+153959928 q^{78}+175385202 q^{79}+194390216 q^{80}+209208210 q^{81}+217767013 q^{82}+217655122 q^{83}+206182023 q^{84}+180375446 q^{85}+137083864 q^{86}+73018494 q^{87}-15089960 q^{88}-130393760 q^{89}-275708923 q^{90}-453351590 q^{91}-664856517 q^{92}-910744842 q^{93}-1190185170 q^{94}-1500703210 q^{95}-1837805659 q^{96}-2194650672 q^{97}-2561673782 q^{98}-2926258326 q^{99}-3272416148 q^{100}
\end{math}
}

%%%% see Mathematica notebook: dimofte/3DIndex/3DIndex.41.41Sister.nb
The first 1000 coefficients are available from \cite{Ga:data}.

\subsection{The index of the sister of the $4_1$ knot complement}
\lbl{sub.index41s}

The $4_1$ knot complement and its sister are the census manifolds {\tt m004} and 
{\tt m003} respectively, and are punctured torus bundles over a circle 
with monodromy $+RL$ and $-RL$ respectively \cite{SnapPy}. The Epstein-Penner
decompositions for {\tt m004} and {\tt m003} consist of two regular ideal
tetrahedra. The edge gluing equations of $+RL$ and $-RL$ coincide. 
Thus, $I_{\text{{\tt m003}}}(0)=I_{\text{{\tt m004}}}(0)$.

\subsection{The index of the $5_2$ knot complement}
\lbl{sub.index52}

The default {\tt SnapPy} triangulation $\CT$ of the $5_2$ knot  complement uses 3 ideal 
tetrahedra. The Epstein-Penner decomposition $\CT^{\EP}$ uses 
4 ideal tetrahedra. Both triangulations carry geometric shape
structures and thus canonical strict angle structures. One can show that
$\CT$ and $\CT^{\EP}$ are related by geometric 2--3 moves hence they have
equal indices. For reasons of efficiency, we will work with the triangulation
$\CT$. Its angle structure matrices are given by:

$$
\overline{\mb A} =
\left(
\begin{array}{ccc}
 1 & 1 & 1 \\
 0 & 0 & 0 \\
 1 & 1 & 1 
\end{array}
\right)
\qquad
\overline{\mb B} =
\left(
\begin{array}{ccc}
 0 & 2 & 0 \\
 1 & 0 & 1 \\
 1 & 0 & 1 
\end{array}
\right)
\qquad
\overline{\mb C} =
\left(
\begin{array}{ccc}
 1 & 0 & 1 \\
 1 & 2 & 1 \\
 0 & 0 & 0 
\end{array}
\right) \, .
$$

Eliminating $\overline{\mb B}$, and removing the third row (which corresponds 
to the third edge equation), we obtain the angle
structure equations
$$
\mb A' \a + \mb B' \gamma = \pi \nu' 
$$
where
$$
\mb A'=
\left(
\begin{array}{ccc}
 1 & -1 & 1 \\
 -1 & 0 & -1 \\
\end{array}
\right)
\qquad 
\mb B'=
\left(
\begin{array}{ccc}
 1 & -2 & 1 \\
 0 & 2 & 0 \\
\end{array}
\right)
\qquad 
\nu'=
\left(
\begin{array}{c}
 0 \\
 0 \\
\end{array}
\right) \, .
$$
The index of $\CT$ is given by
\begin{equation}
\lbl{eq.ind.52}
I_{5_2}(0)(q)=\sum_{(k_1,k_2) \in \BZ^2} \ID(-k_1,k_1-k_2)^2 \ID(2 k_1-2 k_2,-k_1)
\in \BZ[[q]] \, .
\end{equation}
Equation \eqref{eq.degID} implies that the degree $\d(k_1,k_2)$ of the 
summand is a piecewise quadratic polynomial with fan $\calF_{\CT}$
given by six rays, $\rho_1, \ldots, \rho_6$, corresponding to the columns of the matrix
$$
\left(
\begin{array}{cccccc}
 2 & 1 & 0 & -1 & -1 & 0 \\
 1 & 1 & 1 & 0 & -1 & -1 \\
\end{array}
\right) \, .
$$

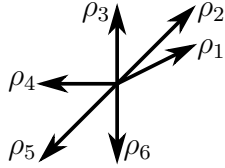
\begin{figure}[htpb]
\centering
\labellist
\pinlabel $\rho_5$ at -3 3
\pinlabel $\rho_4$ at -3 18
\pinlabel $\rho_3$ at 13 33
\pinlabel $\rho_2$ at 38 33
\pinlabel $\rho_1$ at 38 25
\pinlabel $\rho_6$ at 22 3

\endlabellist
\includegraphics[width=0.13\textwidth]{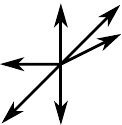}
\caption{The fan of the summand of $I_{5_2}(0)$.}
\label{6ays}
\end{figure}
Let $C_{ij}$ denote the two-dimensional cone with rays $\rho_i$ and $\rho_j$.
Then we have:
$$
\d(k_1,k_2)=
\begin{cases}
\frac{k_1}{2}+\frac{k_1^2}{2}
& \text{if} \,\, (k_1,k_2) \in C_{12} \\
-\frac{k_1}{2}-\frac{k_1^2}{2}+k_2+k_2^2
& \text{if} \,\, (k_1,k_2) \in C_{23} \\
-k_1+k_1^2+k_2-2 k_1 k_2+k_2^2
& \text{if} \,\, (k_1,k_2) \in C_{34} \\
-k_1+k_1^2
& \text{if} \,\, (k_1,k_2) \in C_{45} \\
k_1^2-k_2-2 k_1 k_2+2 k_2^2
& \text{if} \,\, (k_1,k_2) \in C_{56} \\
k_1+2 k_1^2-k_2-4 k_1 k_2+2 k_2^2
& \text{if} \,\, (k_1,k_2) \in C_{61}  \, .
\end{cases}
$$
A plot of $\d(k_1,k_2)$ for $k_1, k_2 \in [-1,1]$ is given in Figure \ref{plot.summand.52}.

\begin{figure}[htpb]
\centering
\includegraphics[width=0.5\textwidth]{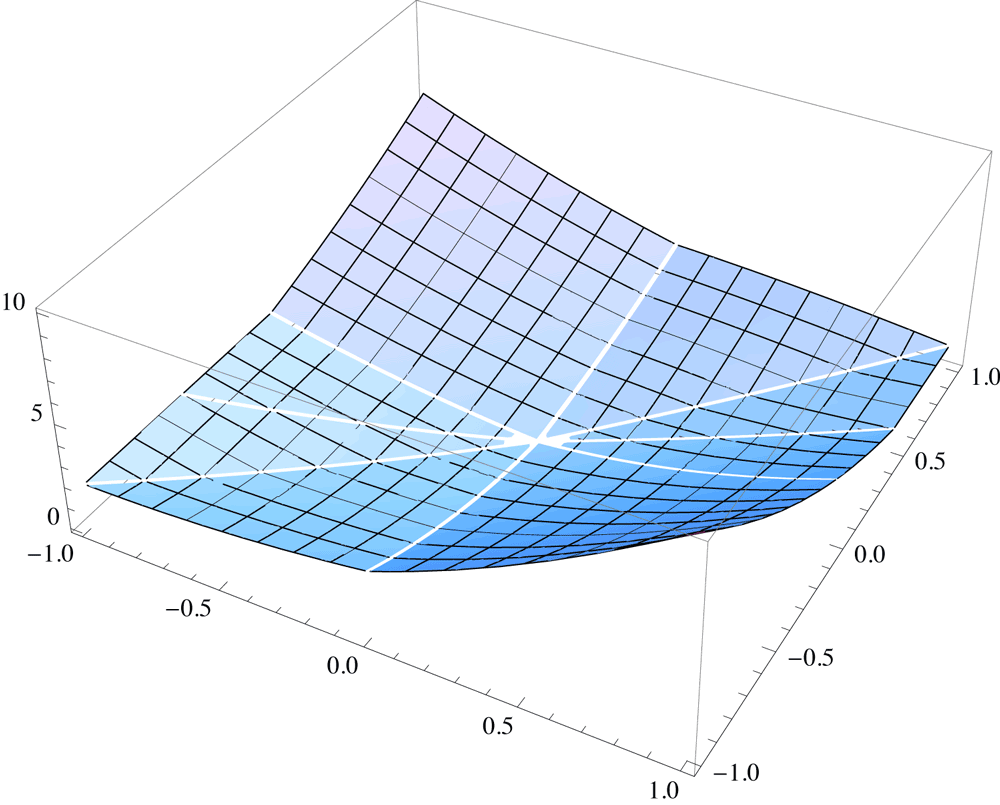}
\caption{The plot of the degree of the summand of $5_2$.}
\label{plot.summand.52}
\end{figure}

If $\d(k_1,k_2) \leq N$, then to bound $k_1$ 
from above and below we use the restriction to the rays $\rho_1$ and 
$\rho_5$ respectively. Likewise, to bound $k_2$ 
from above and below we use the restriction to the rays $\rho_2$ and 
$\rho_5$ respectively. The equations
$$
\d(k_1,k_1)=
\begin{cases} 
(-1 + k_1) k_1 & \text{if} \,\, k_1 \leq 0 \\
\frac{1}{2} k_1 (1 + k_1) & \text{if} \,\, k_1 \geq 0
\end{cases}
\qquad
\d(2 k_2,k_2)=
\begin{cases} 
2 k_2 (-1 + 2 k_2) k_2 & \text{if} \,\, k_2 \leq 0 \\
k_2 (1 + 2 k_2) & \text{if} \,\, k_2 \geq 0
\end{cases}
$$
and the inequality $\d(k_1,k_2) \leq N$ give the bounds
\begin{subequations}
\be
\lbl{eq.52boundsa}
\frac{1}{2} (1 - \sqrt{1 + 4 N}) \leq k_1  \leq 
\frac{1}{2} (1 + \sqrt{1 + 8 N}) 
\ee
\be
\lbl{eq.52boundsb}
\frac{1}{2} (1 - \sqrt{1 + 4 N}) \leq k_2  \leq 
\frac{1}{2} (1 + \sqrt{1 + 8 N}) \, .
\ee
\end{subequations}
With these bounds, we can compute the coefficients
of $I_{5_2}(0)(q)$. The first 100 of them are given by: 

\vspace{0.3cm}
{\tiny
\begin{math}
1-4 q-q^2+16 q^3+26 q^4+23 q^5-34 q^6-122 q^7-239 q^8-312 q^9-221 q^{10}+102 q^{11}+778 q^{12}+1757 q^{13}+2930 q^{14}+3825 q^{15}+4003 q^{16}+2560 q^{17}-1183 q^{18}-8033 q^{19}-18087 q^{20}-30864 q^{21}-44625 q^{22}-56225 q^{23}-60913 q^{24}-52342 q^{25}-23373 q^{26}+33675 q^{27}+124356 q^{28}+251997 q^{29}+412837 q^{30}+596153 q^{31}+778487 q^{32}+925195 q^{33}+984860 q^{34}+895092 q^{35}+579789 q^{36}-39418 q^{37}-1039055 q^{38}-2474979 q^{39}-4370844 q^{40}-6691737 q^{41}-9326308 q^{42}-12059462 q^{43}-14553043 q^{44}-16329323 q^{45}-16762776 q^{46}-15091256 q^{47}-10436174 q^{48}-1863234 q^{49}+11551967 q^{50}+30594044 q^{51}+55785006 q^{52}+87178516 q^{53}+124185931 q^{54}+165312079 q^{55}+207965719 q^{56}+248191454 q^{57}+280543349 q^{58}+297911617 q^{59}+291580973 q^{60}+251299184 q^{61}+165675668 q^{62}+22662435 q^{63}-189540512 q^{64}-481437576 q^{65}-860708203 q^{66}-1330509280 q^{67}-1887333621 q^{68}-2518748267 q^{69}-3200856694 q^{70}-3895826567 q^{71}-4549488766 q^{72}-5089337005 q^{73}-5423179738 q^{74}-5438686309 q^{75}-5004368747 q^{76}-3972096155 q^{77}-2181963773 q^{78}+530655834 q^{79}+4324768774 q^{80}+9340768927 q^{81}+15683465230 q^{82}+23402429365 q^{83}+32468254286 q^{84}+42746925974 q^{85}+53970937518 q^{86}+65710530396 q^{87}+77343886238 q^{88}+88031000047 q^{89}+96690414035 q^{90}+101984965333 q^{91}+102316085357 q^{92}+95834146866 q^{93}+80464612248 q^{94}+53958404128 q^{95}+13966121550 q^{96}-41855327561 q^{97}-115704867879 q^{98}-209460627592 q^{99}-324467541887 q^{100}
\end{math}
}

\vspace{0.3cm}
The first 300 coefficients are available from \cite{Ga:data}. 

\subsection{The index of the $(-2,3,7)$ pretzel knot complement}
\lbl{sub.index237}

The $(-2,3,7)$ pretzel knot is the 12 crossing knot {\tt 12n242}  in 
the census. The default {\tt SnapPy} triangulation $\CT$ of the {\tt 12n242} complement
uses 3 ideal tetrahedra. The Epstein-Penner decomposition $\CT^{\EP}$ uses 
3 ideal tetrahedra. However, the two triangulations $\CT$ and $\CT^{\EP}$
are combinatorially different; for instance have edge-valencies $5,6,7$ and $5,5,8$
%{\tt \{5: 1, 6: 1, 7: 1\}} and {\tt \{8: 1, 5: 2\}} 
respectively. 
%%% sage: from snappy import *
%%% sage: M=Manifold("12n242")
%%% sage: N=M.copy()
%%% sage: N.canonize()
%%% sage: M.edge_valences()
%%% {5: 1, 6: 1, 7: 1}
%%% sage: N.edge_valences()
%%% {8: 1, 5: 2}
Nevertheless, both triangulations are geometric with canonical strict
angle structure and are related by geometric 2--3 moves, which preserve the index. 
The angle structure equations of $\CT$ are given by:

$$
\overline{\mb A} =
\left(
\begin{array}{ccc}
 1 & 0 & 0 \\
 0 & 1 & 1 \\
 1 & 1 & 1
\end{array}
\right)
\qquad
\overline{\mb B} =
\left(
\begin{array}{ccc}
 0 & 1 & 0 \\
 1 & 1 & 0 \\
 1 & 0 & 2 
\end{array}
\right)
\qquad
\overline{\mb C} =
\left(
\begin{array}{ccc}
 1 & 1 & 2 \\
 1 & 0 & 0 \\
 0 & 1 & 0
\end{array}
\right) \, .
$$

Apply the following operations on $(\overline{\mb A}|\overline{\mb B}|\overline{\mb C})$:
\begin{itemize}
\item
Permute the rows according to $(123) \mapsto (312)$.
\item
Permute the second and third columns of $\overline{\mb A}$ and simultaneously, of
 $\overline{\mb B}$ and  $\overline{\mb C}$.
\item
If $\bar{\boldsymbol a}_1, \bar{\boldsymbol b}_1, \bar{\boldsymbol c}_1$ are the first columns of $\overline{\mb A}$,  $\overline{\mb B}$ and  $\overline{\mb C}$,
then permute $(\bar{\boldsymbol a}_1| \bar{\boldsymbol b}_1 | \bar{\boldsymbol c}_1)\mapsto 
(\bar{\boldsymbol b}_1| \bar{\boldsymbol c}_1 | \bar{\boldsymbol a}_1)$.
\end{itemize}
After the above permutations, the matrix $(\overline{\mb A}|\overline{\mb B}|\overline{\mb C})$ of the $(-2,3,7)$ 
pretzel knot becomes the corresponding matrix of the $5_2$ knot. Since
the above permutations do not change the index, it follows that
$I_{(-2,3,7)}(0)=I_{5_2}(0)$.

\begin{exercise}
\lbl{exerc.237} 
Using the matrix $(\overline{\mb A}|\overline{\mb B}|\overline{\mb C})$ above, it follows that
the index of $\CT$ is given by
\begin{equation}
\lbl{eq.ind.237}
I_{(-2,3,7)}(0)(q)=\sum_{(k_1,k_2) \in \BZ^2} 
(-1)^{k_2} q^{\frac{1}{2} (k_1 - 2 k_2)}
\ID(k_1, -k_2) \ID(-k_1 + k_2, k_1) \ID(-2 k_1 + 2 k_2, k_1 - 2 k_2).
\end{equation}
On the other hand, the index of $5_2$ is given by Equation \eqref{eq.ind.52}.
Using the identities of the tetrahedron index from Section \ref{sec.review},
show that
\begin{align*}
\sum_{(k_1,k_2) \in \BZ^2} \ID(-k_1,k_1-k_2)^2 \ID(2 k_1-2 k_2,-k_1)
&= \\
\sum_{(k_1,k_2) \in \BZ^2} 
(-1)^{k_2} q^{\frac{1}{2} (k_1 - 2 k_2)}
  \ID(k_1, -k_2) \ID(-k_1 + k_2, k_1) \ID(-2 k_1 + 2 k_2, k_1 - 2 k_2). &
\end{align*}
\end{exercise}

\subsection{The index of the $6_1$ knot complement}
\lbl{sub.index61}

In this section we discuss the index $I_{6_1}(0)$ as an example of the
phenomenon of linear growth. 
The default {\tt SnapPy} triangulation $\CT$ of the $6_1$ knot  complement has
4 ideal tetrahedra, and the Epstein-Penner decomposition uses 6 ideal
tetrahedra. The two triangulations are related by geometric 2--3 moves
and so have equal indices.
The angle structure equations of $\CT$ are given by:

$$
\overline{\mb A} =
\left(
\begin{array}{cccc}
 1 & 1 & 0 & 0 \\
 0 & 0 & 0 & 1 \\
 0 & 1 & 1 & 0 \\
 1 & 0 & 1 & 1 
\end{array}
\right)
\qquad
\overline{\mb B} =
\left(
\begin{array}{cccc}
 0 & 2 & 0 & 1 \\
 1 & 0 & 1 & 0 \\
 1 & 0 & 0 & 0 \\
 0 & 0 & 1 & 1
\end{array}
\right)
\qquad
\overline{\mb C} =
\left(
\begin{array}{cccc}
 1 & 0 & 1 & 1 \\
 1 & 2 & 1 & 0 \\
 0 & 0 & 0 & 1 \\
 0 & 0 & 0 & 0
\end{array}
\right) \, .
$$

Eliminating $\overline{\mb B}$, and removing the fourth row (which corresponds 
to the fourth edge equation), we obtain the angle
structure equations
$$
\mb A' \a + \mb B' \gamma = \pi \nu' 
$$
where
$$
\mb A'=
\left(
\begin{array}{cccc}
 1 & -1 & 0 & -1 \\
 -1 & 0 & -1 & 1 \\
 -1 & 1 & 1 & 0 \\
\end{array}
\right)
\qquad 
\mb B'=
\left(
\begin{array}{cccc}
 1 & -2 & 1 & 0 \\
 0 & 2 & 0 & 0 \\
 -1 & 0 & 0 & 1 \\
\end{array}
\right)
\qquad 
\nu'=
\left(
\begin{array}{c}
 -1 \\
 0 \\
 1 \\
\end{array}
\right) \, .
$$
The index of $\CT$ is given by
\begin{align*}
I_{6_1}(0)(q) &= \sum_{(k_1,k_2,k_3) \in \BZ^3} 
q^{\frac{1}{2} (-k_1 + k_3)}
  \ID(-k_1, -k_2 + k_3) \ID(
  2 k_1 - 2 k_2, -k_1 + k_3) \\
&  \ID(-k_3, -k_1 + k_2) \ID(-k_1 + k_3, 
  k_1 - k_2 - k_3) \in \BZ[[q]] \, .
\end{align*}
$\calF_{\CT}$ in $\BR^3$ has 18 rays, corresponding to the columns of the matrix
$$
\left(
\begin{array}{cccccccccccccccccc}
 -1 & -1 & -1 & -1 & -1 & -1 & 0 & 0 & 0 & 0 & 0 & 0 & 1 & 1 & 1 & 1 & 2 & 3 \\
 -2 & -2 & -1 & 0 & 0 & 1 & -1 & -1 & -1 & 0 & 1 & 1 & 0 & 0 & 1 & 1 & 1 & 2 \\
 -3 & -1 & 0 & -1 & 0 & 0 & -2 & -1 & 0 & 1 & -1 & 0 & -1 & 1 & 0 & 1 & 1 & 1 \\
\end{array}
\right) \, .
$$
If $\d(k_1,k_2,k_3)$ denotes the degree of the summand and 
$N \in \BN$, then $\d(k_1,k_2,k_3) \leq N$ implies that $(k_1,k_2,k_3)$
satisfy the bounds
\begin{equation*}
\frac{1}{2} (1 - \sqrt{1 + 4 N}) \leq k_1  \leq 
\frac{1}{2} (-1 + \sqrt{1 + 12 N}) 
\end{equation*}
\begin{equation*}
\frac{2}{3} (1 - \sqrt{1 + 3 N}) \leq k_2  \leq 
\frac{1}{2} (-1 + \sqrt{1 + 8 N})
\end{equation*}
\begin{equation*}
\frac{3}{7} (1 - \sqrt{1 + 7 N}) \leq k_3  \leq 
N  \, .
\end{equation*}
Observe that the upper bound for $k_3$ is linear in $N$.
For instance, when $N=3$, the following 17 terms 
(each a product of 4 truncated tetrahedron indices, 
as in Equation \eqref{eq.IDtruncated} for $N=3$)
contribute to $I_{6_1}(0)+O(q)^4= 1-4 q+q^2+18 q^3+O(q)^4$: 

\vspace{0.3cm}
{\tiny
\begin{math}
\ID(-2,0)^2 \ID(0,-2) \ID(0,0)
-q^{-\frac{1}{2}} \ID(-1,-1) \ID(-1,0) \ID(0,-1) \ID(0,0)
+\ID(-1,0)^2 \ID(0,-1) \ID(0,0)+\ID(0,0)^4
\\ +\ID(-2,0) \ID(0,-1)^2 \ID(0,1)
-q^{\frac{1}{2}} \ID(-2,0) \ID(-1,1) \ID(0,1) \ID(1,-2)
-q^{\frac{1}{2}} \ID(-1,0) \ID(0,1)^2 \ID(1,-1) \\
-q^{-\frac{1}{2}} \ID(-1,1) \ID(0,-1)^2 \ID(1,0)
-q^{\frac{1}{2}} \ID(-2,1) \ID(0,1) \ID(1,-1) \ID(1,0)
+\ID(0,0) \ID(0,1) \ID(1,0)^2 \\
+\ID(-2,0) \ID(0,0) \ID(1,-1) \ID(1,1)
-q^{\frac{1}{2}} \ID(0,0) \ID(0,1) \ID(1,0) \ID(1,1)
+q \ID(-2,0) \ID(0,2)^2 \ID(2,-2) \\
-q^{-\frac{1}{2}}\ID(-2,0) \ID(-1,-1) \ID(-1,0) \ID(2,-1)
-q^{-\frac{1}{2}}\ID(-1,0) \ID(-1,1) \ID(0,-1) \ID(2,-1) \\
-q^{-\frac{1}{2}} \ID(-1,2) \ID(0,0) \ID(1,-1) \ID(2,-1))
-q^{\frac{3}{2}} \ID(-3,0) \ID(0,3)^2 \ID(3,-3)
\end{math}
}
\vspace{0.3cm}

The first 50 coefficients of $I_{6_1}(0)$ are given by:

\vspace{0.3cm}
{\tiny
\begin{math}
1-4 q+q^2+18 q^3+22 q^4+q^5-78 q^6-178 q^7-254 q^8-188 q^9+167 q^{10}+855 q^{11}+1864 q^{12}+2892 q^{13}+3426 q^{14}+2583 q^{15}-488 q^{16}-6698 q^{17}-16273 q^{18}-28550 q^{19}-41189 q^{20}-49943 q^{21}-48554 q^{22}-28899 q^{23}+17621 q^{24}+98726 q^{25}+217819 q^{26}+371551 q^{27}+544496 q^{28}+707360 q^{29}+811832 q^{30}+792301 q^{31}+565550 q^{32}+40436 q^{33}-872995 q^{34}-2241496 q^{35}-4087180 q^{36}-6354321 q^{37}-8877834 q^{38}-11348143 q^{39}-13283739 q^{40}-14014789 q^{41}-12685231 q^{42}-8288627 q^{43}+266720 q^{44}+14032731 q^{45}+33862808 q^{46}+60173861 q^{47}+92687285 q^{48}+130092845 q^{49}+169735693 q^{50}
\end{math}
}

\subsection{The index of the $7_2$ knot  complement}
\lbl{sub.index72}

The index $I_{7_2}(0)$ is another example of linear growth. 
The default {\tt SnapPy} triangulation $\CT$ of the $7_2$ knot  complement has
4 ideal tetrahedra, and the Epstein-Penner decomposition has 8 ideal
tetrahedra. The two triangulations are related by geometric 2--3 moves
and have equal indices. The angle structure equations of $\CT$ are given by:

$$
\overline{\mb A} =
\left(
\begin{array}{cccc}
 1 & 2 & 1 & 0 \\
 0 & 0 & 1 & 0 \\
 1 & 0 & 0 & 1 \\
 0 & 0 & 0 & 1 
\end{array}
\right)
\qquad
\overline{\mb B} =
\left(
\begin{array}{cccc}
 1 & 0 & 0 & 0 \\
 1 & 0 & 0 & 0 \\
 0 & 2 & 1 & 0 \\
 0 & 0 & 1 & 2 
\end{array}
\right)
\qquad
\overline{\mb C} =
\left(
\begin{array}{cccc}
 0 & 0 & 1 & 0 \\
 0 & 1 & 0 & 2 \\
 1 & 1 & 1 & 0 \\
 1 & 0 & 0 & 0
\end{array}
\right) \, .
$$
Eliminating $\overline{\mb B}$, and removing the fourth row (which corresponds 
to the fourth edge equation), we obtain the angle
structure equations
$$
\mb A' \a + \mb B' \gamma = \pi \nu' 
$$
where
$$
\mb A'=
\left(
\begin{array}{cccc}
 0 & 2 & 1 & 0 \\
 -1 & 0 & 1 & 0 \\
 1 & -2 & -1 & 1 \\
\end{array}
\right)
\qquad 
\mb B'=
\left(
\begin{array}{cccc}
 -1 & 0 & 1 & 0 \\
 -1 & 1 & 0 & 2 \\
 1 & -1 & 0 & 0 \\
\end{array}
\right)
\qquad 
\nu'=
\left(
\begin{array}{c}
1 \\
1 \\
-1 \\
\end{array}
\right) \, .
$$
The index of $\CT$ is given by
\begin{align*}
I_{7_2}(0)(q) &= \sum_{(k_1,k_2,k_3) \in \BZ^3} 
(-1)^{k_1+k_2-k_3} q^{\frac{1}{2} (k_1+k_2-k_3)} 
\ID(-k_1,k_1+k_2-k_3) \ID(-2 k_2,k_3) \\
 & \ID(k_1+k_2-k_3,-k_2+k_3) \ID(-k_2+k_3,2 k_1-2 k_3)
\in \BZ[[q]] \, .
\end{align*}
$\calF_{\CT}$ in $\BR^3$ has 14 rays, spanned by the columns of the matrix
$$
\left(
\begin{array}{cccccccccccccc}
 -2 & -1 & -1 & 0 & 0 & 0 & 0 & 0 & 0 & 1 & 1 & 1 & 1 & 2 \\
 -1 & -1 & 0 & -1 & -1 & -1 & 0 & 1 & 1 & 0 & 0 & 0 & 1 & 1 \\
 -2 & -1 & 0 & -2 & -1 & 0 & 1 & -1 & 0 & 0 & 1 & 2 & 1 & 3 \\
\end{array}
\right) \, .
$$
If $\d(k_1,k_2,k_3)$ denotes the degree of the summand and 
$N \in \BN$, then $\d(k_1,k_2,k_3) \leq N$ implies that $(k_1,k_2,k_3)$
satisfy the bounds
\begin{equation*}
1 - \sqrt{1 + 2 N} \leq k_1  \leq 
\frac{1}{2} (-1 + \sqrt{1 + 8 N}) 
\end{equation*}
\begin{equation*}
\frac{1}{2} (1 - \sqrt{1 + 4 N}) \leq k_2  \leq 
N 
\end{equation*}
\begin{equation*}
1 - \sqrt{1 + 2 N} \leq k_3  \leq 
\frac{1}{2} (-1 + \sqrt{1 + 12 N}) \, .
\end{equation*}
Observe that the upper bound for $k_3$ is linear in $N$.
The first 50 coefficients of $I_{7_2}(0)$ are given by:

\vspace{0.3cm}
{\tiny
\begin{math}
1-4 q+q^2+16 q^3+20 q^4+q^5-72 q^6-156 q^7-206 q^8-98 q^9+275 q^{10}+924 q^{11}+1740 q^{12}+2370 q^{13}+2227 q^{14}+495 q^{15}-3485 q^{16}-10168 q^{17}-19045 q^{18}-28467 q^{19}-34899 q^{20}-33157 q^{21}-16460 q^{22}+22305 q^{23}+89035 q^{24}+185478 q^{25}+306146 q^{26}+434575 q^{27}+539981 q^{28}+575717 q^{29}+479148 q^{30}+176483 q^{31}-408854 q^{32}-1340316 q^{33}-2648389 q^{34}-4301970 q^{35}-6179555 q^{36}-8036073 q^{37}-9477453 q^{38}-9942897 q^{39}-8710346 q^{40}-4925980 q^{41}+2323715 q^{42}+13897628 q^{43}+30430263 q^{44}+52111135 q^{45}+78414600 q^{46}+107796294 q^{47}+137380650 q^{48}+162674912 q^{49}+177363801 q^{50}
\end{math}
}

%%%%%%%%%%%%%%%%%%%%%%%%%%%%%%%%%%%%%%%%%%%%%%%%%%%%%%%%%%%%%%%%%%%%%%%%%%%%
%%%%%%%%%%%%%%%%%%%%%%%%%%%%%%%%%%%%%%%%%%%%%%%%%%%%%%%%%%%%%%%%%%%%%%%%%%%%

\subsection{Acknowledgments}
S.G. wishes to thank Tudor Dimofte, Josephine Yu and especially Don Zagier
for enlightening conversations. A preliminary report was delivered by 
S.G. during the Oberwolfach 8/2012 workshop on 
{\em Low-Dimensional Topology and Number Theory}; we thank the 
organizers, Paul E. Gunnells, Walter Neumann, Adam S. Sikora and Don Zagier, 
for their hospitality.

\appendix

\section{The 2--3 move}
\lbl{app.23}

For completeness, in this appendix we give a detailed proof of the invariance of the index under
2--3 moves, following \cite[Sec.6]{Ga:index} and also \cite{DGG2}.
Let $M$ be a cusped $3$-manifold and let 
$\varpi$ be an oriented 
multi-curve on $\bd M$.
Consider two ideal triangulations $\CT$ and $\wt{\CT}$ of $M$ with $N$ and $N+1$
tetrahedra, respectively, related by a 2--3 move as shown in Figure \ref{fig.23},
matching the conventions of \cite[Sec.3.6]{DG}. 

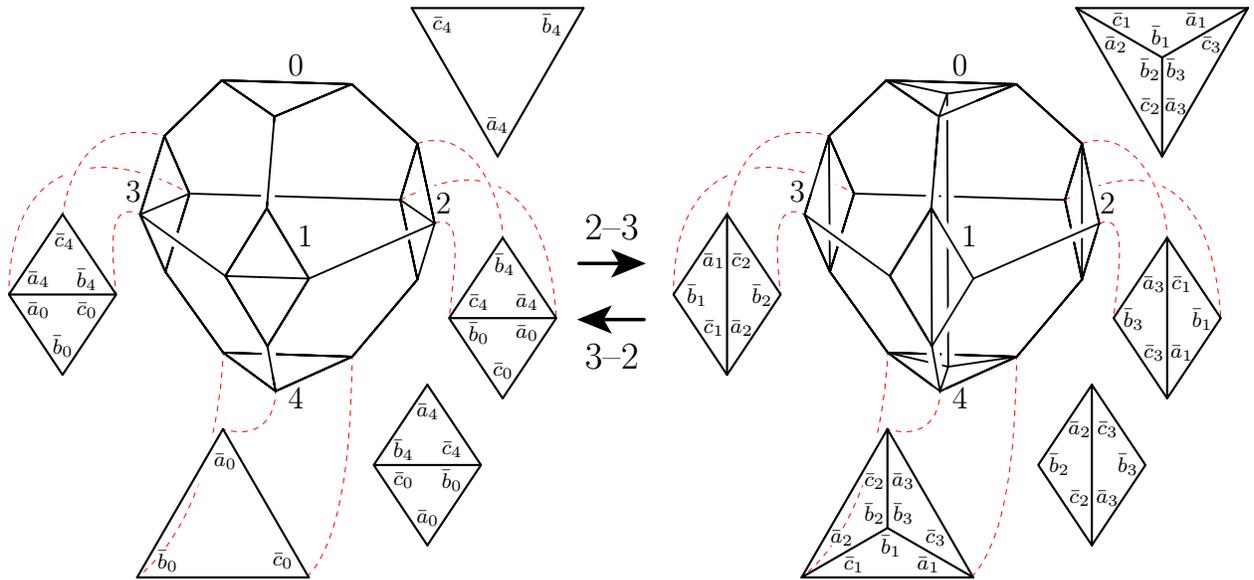
\begin{figure}[htpb]
\centering

\labellist
\large
\pinlabel 2--3 at 650 380
\pinlabel 3--2 at 650 240
\normalsize
\pinlabel 4 at 310 195
\pinlabel 0 at 310 555
\pinlabel 1 at 320 370
\pinlabel 3 at 135 418
\pinlabel 2 at 469 405

\pinlabel 4 at 1025 195
\pinlabel 0 at 1025 555
\pinlabel 1 at 1035 370
\pinlabel 3 at 850 418
\pinlabel 2 at 1184 405

\tiny
\pinlabel $\bar b_1$ at 951 35
\pinlabel $\bar c_1$ at 911 15
\pinlabel $\bar a_1$ at 991 15
\pinlabel $\bar b_2$ at 933 70
\pinlabel $\bar b_3$ at 964 70
\pinlabel $\bar a_2$ at 898 44
\pinlabel $\bar c_3$ at 999 44
\pinlabel $\bar c_2$ at 933 105
\pinlabel $\bar a_3$ at 964 105

\pinlabel $\bar b_1$ at 1242 585
\pinlabel $\bar c_1$ at 1200 603
\pinlabel $\bar a_1$ at 1280 603
\pinlabel $\bar b_2$ at 1230 550
\pinlabel $\bar b_3$ at 1258 550
\pinlabel $\bar a_2$ at 1193 577
\pinlabel $\bar c_3$ at 1296 577
\pinlabel $\bar c_2$ at 1230 510
\pinlabel $\bar a_3$ at 1258 510

\pinlabel $\bar a_3$ at 1234 318
\pinlabel $\bar c_1$ at 1264 318
\pinlabel $\bar c_3$ at 1234 244
\pinlabel $\bar a_1$ at 1264 244
\pinlabel $\bar b_3$ at 1212 280
\pinlabel $\bar b_1$ at 1286 280

\pinlabel $\bar a_2$ at 1154 161
\pinlabel $\bar c_3$ at 1184 161
\pinlabel $\bar c_2$ at 1154 87
\pinlabel $\bar a_3$ at 1184 87
\pinlabel $\bar b_2$ at 1132 123
\pinlabel $\bar b_3$ at 1206 123

\pinlabel $\bar a_1$ at 760 343
\pinlabel $\bar c_2$ at 790 343
\pinlabel $\bar c_1$ at 760 270
\pinlabel $\bar a_2$ at 790 270
\pinlabel $\bar b_1$ at 742 306
\pinlabel $\bar b_2$ at 812 306

\pinlabel $\bar a_0$ at 234 126
\pinlabel $\bar b_0$ at 172 19
\pinlabel $\bar c_0$ at 296 19

\pinlabel $\bar a_4$ at 527 490
\pinlabel $\bar b_4$ at 586 598
\pinlabel $\bar c_4$ at 468 598

\pinlabel $\bar a_0$ at 452 65
\pinlabel $\bar b_0$ at 478 106
\pinlabel $\bar c_0$ at 426 106
\pinlabel $\bar a_4$ at 452 180
\pinlabel $\bar c_4$ at 478 139
\pinlabel $\bar b_4$ at 426 139

\pinlabel $\bar c_0$ at 533 224
\pinlabel $\bar a_0$ at 559 265
\pinlabel $\bar b_0$ at 507 265
\pinlabel $\bar b_4$ at 535 339
\pinlabel $\bar a_4$ at 559 298
\pinlabel $\bar c_4$ at 507 298

\pinlabel $\bar b_0$ at 59 251
\pinlabel $\bar c_0$ at 85 290
\pinlabel $\bar a_0$ at 33 290
\pinlabel $\bar c_4$ at 61 364
\pinlabel $\bar b_4$ at 85 323
\pinlabel $\bar a_4$ at 33 323

\endlabellist

\includegraphics[width=\textwidth]{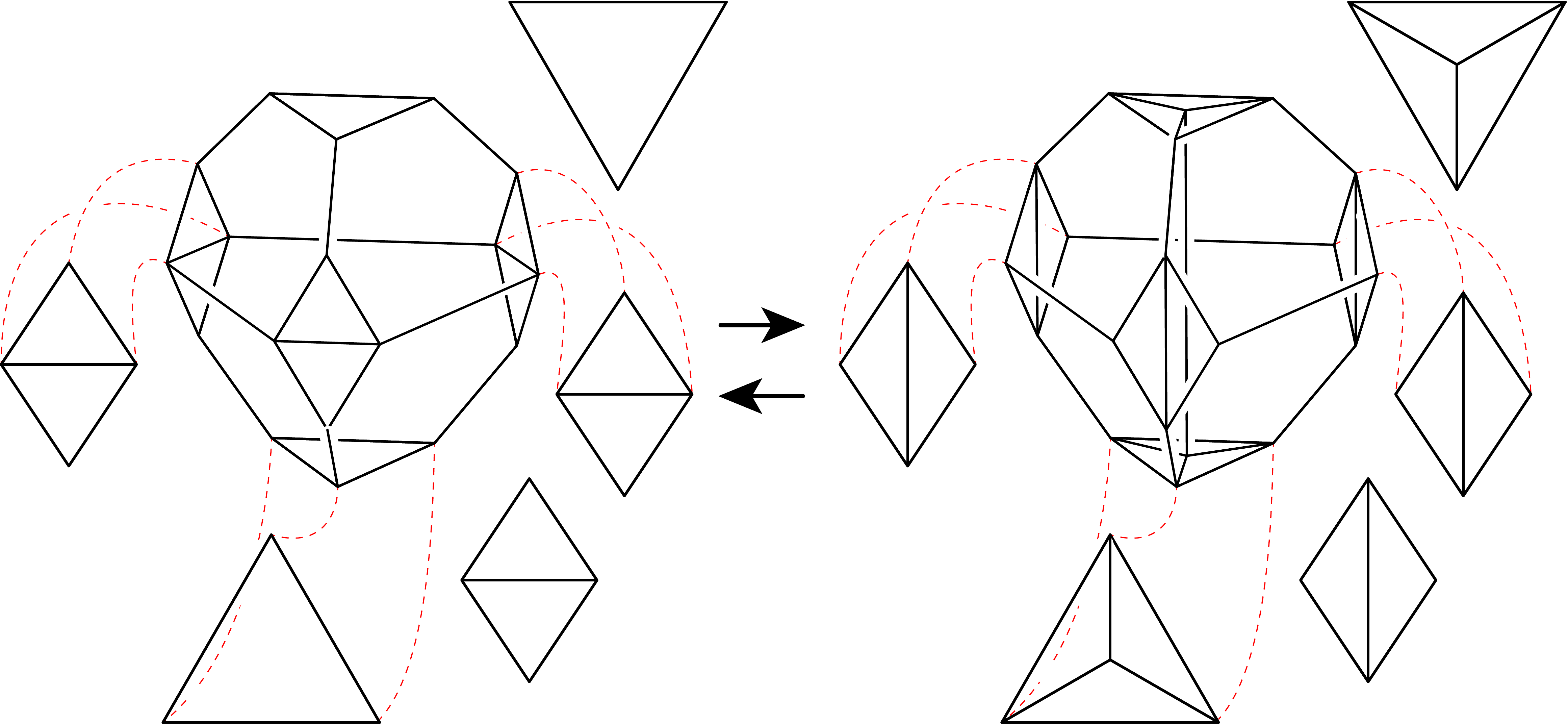}
\caption{The 2--3 move shown with truncated tetrahedra and labelled vertices. The five triangulated ends of the tetrahedra before and after are shown in detail in the ``zoomed in'' pictures. The labels on the corners of the triangles are explained in Section \ref{sub.peripheral}. All truncated triangulation pictures are viewed from outside the tetrahedra.}
\label{fig.23}
\end{figure}

 The main result of this section is the following.
 
 \begin{thm} 
 \label{2-3-invariance}

Suppose that $\CT$ and $\wt\CT$ are ideal triangulations related by a 
2--3 move and both admit an index structure. Then, for any $[\varpi] \in H_1(\bd M;\Z)$,
$I_{\CT}([\varpi])=I_{\wt\CT}([\varpi])$. 
 
 \end{thm}

 Recall that the index is not changed by isotopies of $\varpi$ and, from (\ref{JD_index_eq}),
 can be written in the form 
$$
I_{\CT}(\varpi)(q)  = \sum_{\kk \in \Z^{N-r}} 
 q^{\sum_i k_i}   \prod_j 
J(T_j;\kk,\varpi)
$$
where the contribution from tetrahedron $T_j$ is
$$J(T_j;\kk,\varpi)=\JD(\bar a_j(\kk,\varpi), \bar b_j(\kk,\varpi) , \bar c_j(\kk,\varpi)).$$
Here we can think of $\kk$ as a weight on the $N$ edges of $\tri$ and the summation
 is over 
 $\Z^{N-r} \subset \Z^{N}$ corresponding to a set $\mathcal B$ of $N-r$ basic edges
 as in Theorem \ref{pick_edges}.

\begin{defn} 
Given $\varpi$ as above and $\kk \in \Z^N$, let $\varpi_{\kk}$ denote a collection of disjoint oriented normal curves in $\bd M$ obtained from $\varpi$ by adding $k_i$ small linking circles around the vertex at one end of the $i$th edge in $\CT$.
(The circles are oriented anticlockwise if $k_i>0$ and clockwise if $k_i<0$.)
\end{defn}

By Remark \ref{edge_eqn_as_holonomy}, 
we can think of the coefficient $\bar a_j(\kk,\varpi)$ 
as the turning number $\bar a_{\varpi_{\kk},j}$, as defined in Section 4.4,  of  $\varpi_{\kk}$.
Thus $\bar a_j(\kk,\varpi)$ is a sum of 4 turning numbers of the multi-curve $\varpi_{\kk}$ 
as it turns around the $\bar a_j$ corners of the triangles of the $j$th truncated tetrahedron, 
and similarly for $\bar b_j(\kk,\varpi)$ and $\bar c_j(\kk,\varpi)$. This point of view allows us to treat the contributions from edge weights in the index calculation in the same way as we treat the contributions from peripheral curves.\\

The triangulations $\CT$ and $\wt{\CT}$ only differ inside a bipyramid, and we label its
 vertices $0,1,2,3,4$ as in Figure \ref{fig.23}, where $0$ and $4$ are the north and south poles,
and $1,2,3$ are on the equator. 
This determines 5 tetrahedra $T_0, \ldots ,T_4$ where $T_i$ is the tetrahedron ``opposite'' vertex $i$
with vertex set obtained by omitting vertex $i$ from the set $\{0,1,2,3,4\}$.   
The bipyramid can be decomposed 
into the two tetrahedra $T_0,T_4$ which contain the triangle $123$, 
or into the three tetrahedra $T_1,T_2,T_3$ which contain the edge $04$.
 Then it suffices to prove the following lemma.

\begin{lemma}[Pentagon equality for $J$] \label{lem_bipyramid} 
Let $\varpi$ be an oriented multi-curve on $\bd M$ as above. 
Let $\kk=(k_1, \ldots, k_N)$ be a weight function on the edges of $\tri$,
and let $\wt \kk =(k_0,\kk)$ be an extension of $\kk$ to a weight function on the edges of $\wt \tri$
where $k_0$ is the weight on the new edge ${04}$  introduced in the 2--3 move.
Then, with the notation above,
\be
\sum_{k_{0} \in \Z}   q^{k_{0}}   J(T_1;(k_0,\kk),\varpi) J(T_2;(k_0,\kk),\varpi) J(T_3;(k_0,\kk),\varpi)
= J(T_0;\kk,\varpi) J(T_4;\kk,\varpi).
\label{bipyramid-23-identity}
\ee
\end{lemma}

Theorem \ref{2-3-invariance}
now follows immediately: In the sum \eqref{JD_index_eq} for
$I_\CT$ we choose a set $\mathcal X$ of excluded edges from a maximal tree with 1- or 3-cycle for $\CT$ 
as in Theorem \ref{pick_edges}. Then the {\em same} set of edges can be excluded from the
summation for $I_{\wt{\CT}}$, so the summation for $I_{\wt{\CT}}$ is over the
original set of basic edges  $\mathcal B$ for $\CT$ together with the new edge introduced in the 2--3 move.

\begin{proof}[Proof of Lemma] 
The plan is as follows. First, we isotope $\varpi$ if necessary, so that it intersects the truncated ends of the bipyramid involved in the 2--3 move in a standardised way. For a given $\kk$, we consider $\varpi_{\kk}$, and the $\varpi_{\wt\kk}$ where $\wt \kk =(k_0,\kk)$ and $k_0\in\BZ$. We calculate the left hand side of \eqref{bipyramid-23-identity} using the $\varpi_{\wt\kk}$, and use the original version of the pentagon equality, \eqref{eq.pentagon}, to show that it is equal to the right hand side of \eqref{bipyramid-23-identity}, calculated using $\varpi_{\kk}.$

So, first we arrange $\varpi$ appropriately. The induced triangulations of $\bd M$, $\CT_\bd$ and $\wt \CT_\bd$, determined by $\CT$ and $\wt{\CT}$ respectively are related by 1--3 and 2--2 moves
as seen in Figure \ref{fig.23}.  We isotope $\varpi$ if necessary, so that it is normal relative to both of the two induced triangulations of $\bd M$.

At each of the polar vertices of the bipyramid 
there is a 1--3 move on the boundary, and 
we can assume, after a further isotopy if necessary, that $\varpi$ is represented by collections of oriented normal arcs as shown in Figure \ref{1-3move}, with all turning numbers 0 at the new central vertex. 

 \begin{figure}[htpb]
 \begin{center} 
 \labellist
\pinlabel $\bar a$ at 103 146
\pinlabel $\bar b$ at 30 20
\pinlabel $\bar c$ at 176 20

\pinlabel $\bar a$ at 320 146
\pinlabel $\bar a$ at 336 146
\pinlabel $\bar b$ at 254 29
\pinlabel $\bar b$ at 264 11
\pinlabel $\bar c$ at 402 29
\pinlabel $\bar c$ at 390 11

\endlabellist
 \includegraphics[width=12cm]{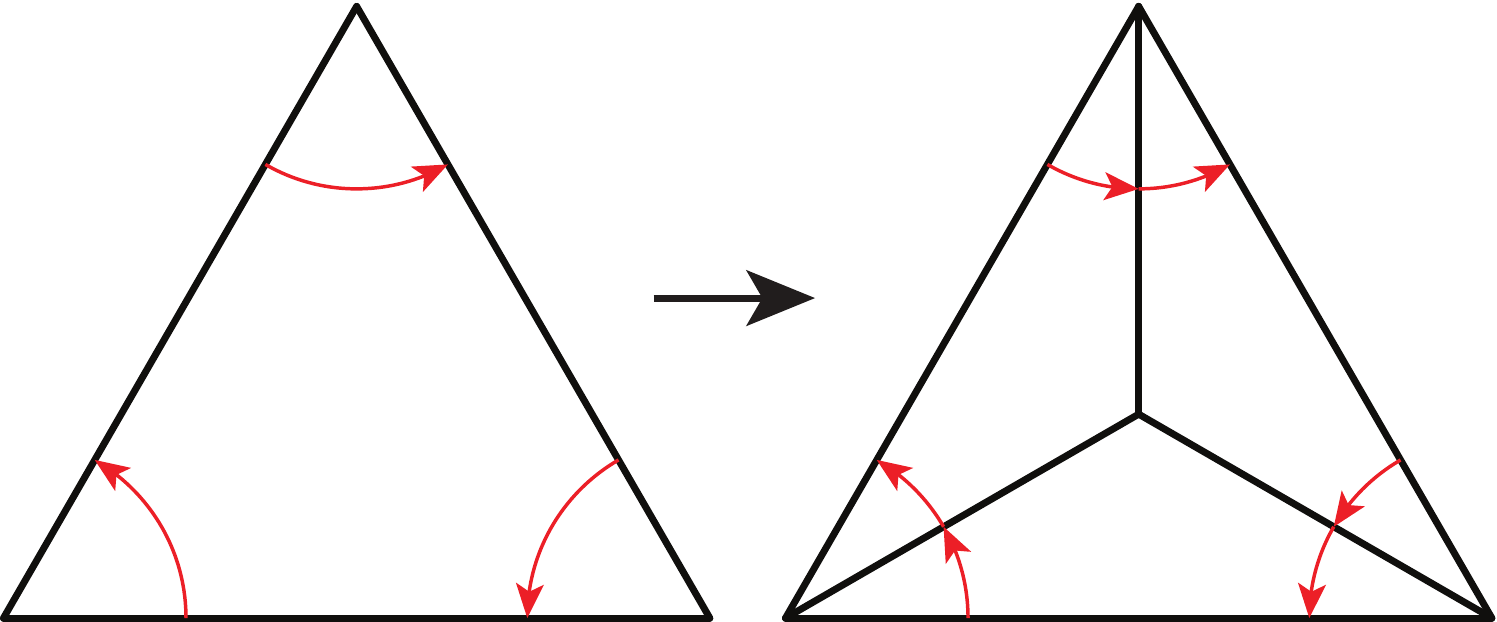}
 \caption{Change in turning numbers under a 1--3 move}
  \label{1-3move}
   \end{center}
 \end{figure}
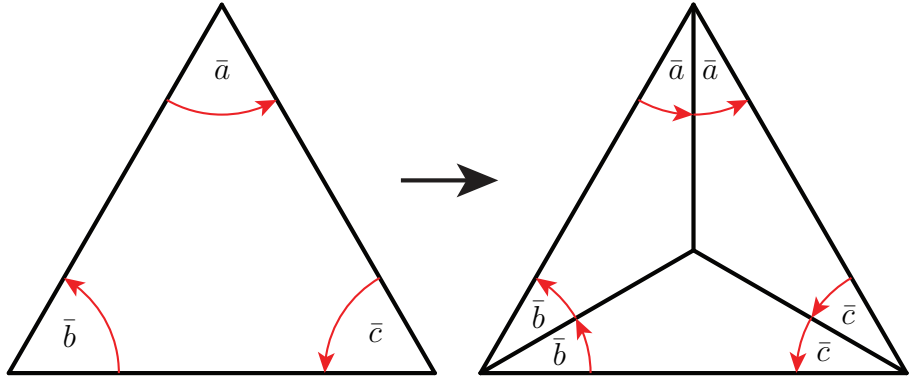
 
 At each of the equatorial vertices 
there is a 2--2 move on the boundary, and we can represent $\varpi$ 
by a collection of oriented normal arcs in a quadrilateral of the six types shown in the right hand side of Figure \ref{quad_paths}.
(Note that either $\bar s=0$ or $\bar t=0$ for an embedded normal curve $\varpi$.)

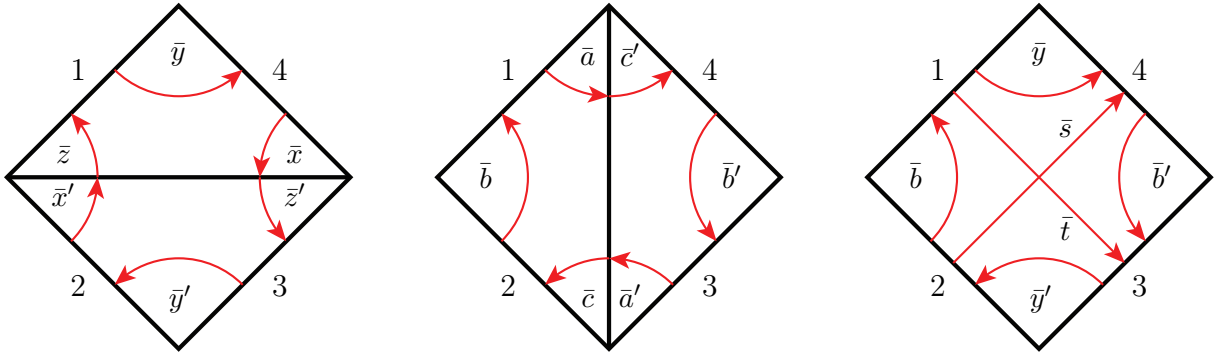
\begin{figure}[htpb]
\begin{center}
 \labellist

\pinlabel 1 at 35 133
\pinlabel 2 at 35 31
\pinlabel 4 at 130 133
\pinlabel 3 at 130 31

\pinlabel 1 at 237.5 133
\pinlabel 2 at 237.5 31
\pinlabel 4 at 332.5 133
\pinlabel 3 at 332.5 31

\pinlabel 1 at 440 133
\pinlabel 2 at 440 31
\pinlabel 4 at 535 133
\pinlabel 3 at 535 31

\small

\pinlabel $\bar y$ at 82 140
\pinlabel $\bar y'$ at 83 25

\pinlabel $\bar y$ at 487.5 140
\pinlabel $\bar y'$ at 488.5 25

\pinlabel $\bar b$ at 227 83
\pinlabel $\bar b'$ at 343 83

\pinlabel $\bar b$ at 429.5 83
\pinlabel $\bar b'$ at 545.5 83

\pinlabel $\bar z$ at 28 93
\pinlabel $\bar x'$ at 28 74
\pinlabel $\bar x$ at 137 93
\pinlabel $\bar z'$ at 137 74

\pinlabel $\bar a$ at 275 139
\pinlabel $\bar c'$ at 295 140
\pinlabel $\bar c$ at 275 25
\pinlabel $\bar a'$ at 295 26.5

\pinlabel $\bar s$ at 500 105
\pinlabel $\bar t$ at 500 58

\endlabellist
\includegraphics[scale = 0.8]{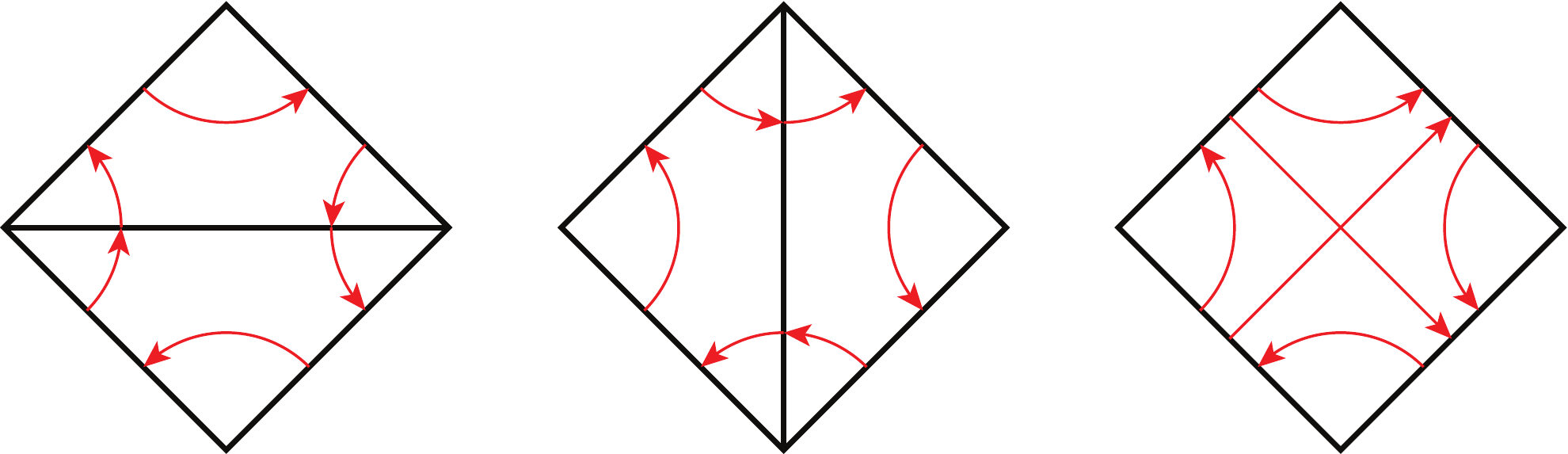}
\caption{Change in turning numbers under a 2--2 move}
\label{quad_paths}
\end{center}
\end{figure}

Then the corresponding turning numbers at the corners of the triangles in the left and centre of the figure are given by:
$$\bar a=\bar y+\bar t,~\bar c=\bar y'-\bar s,~\bar a'=\bar y'-\bar t,~\bar c'=\bar y+\bar s$$
and
$$\bar x=\bar b'-\bar s,~\bar z=\bar b-\bar t,~\bar x'=\bar b+\bar s,~\bar z'=\bar b'+\bar t.$$
For example, $a$ is the signed number of (anticlockwise) normal arcs going
from edge 1 to the vertical edge in the centre of the figure. But each such arc is a normal arc in the quadrilateral going from edge 1 to 3 or from edge 1 to edge 4, hence $\bar a=\bar y+\bar t$.\\

In particular, we have the following relations that will be used below:
\be
\label{obs1}
\bar a+\bar a'=\bar c+\bar c'=\bar y+\bar y' \text{ and } \bar x+\bar x'=\bar z+\bar z'=\bar b+\bar b'
\ee
\be
\label{obs2}
2\bar b+\bar c'-\bar a=2\bar b + \bar a'-\bar c= \bar x'+\bar z  \text{ and }  2\bar b' + \bar c - \bar a' = 2\bar b' + \bar a - \bar c' = \bar x + \bar z'.
\ee

Now $\varpi_{\wt\kk}$ is obtained from $\varpi$ by adding small vertex linking circles around the vertices of $\wt\CT_\bd$, and similarly  $\varpi_{\kk}$ is obtained from $\varpi$ by adding small vertex linking circles around the vertices of $\CT_\bd$. Note that we have a map from the turning numbers of $\varpi$ relative to $\CT$ to the turning numbers of $\varpi$ relative to $\wt\CT$. If we apply the same map to the turning numbers of $\varpi_{\kk}$ relative to $\CT$, we obtain the turning numbers of $\varpi_{(0,\kk)}$ relative to $\wt\CT$. 

\medskip
In the following, we will abbreviate our notation as follows.
Let 
\begin{align*}
\label{abc_abbeviations}
& \bar a_j=\bar a_j(\kk,\varpi),~ \bar b_j = \bar b_j(\kk,\varpi),~ \bar c_j = \bar c_j(\kk,\varpi)
\text { for } j=0,4 \\
& \bar a_j=\bar a_j((0,\kk),\varpi),~ \bar b_j = \bar b_j((0,\kk),\varpi),~ \bar c_j = \bar c_j((0,\kk),\varpi)
\text { for } j=1,2,3.
\end{align*}
Recall that $\bar a_j(\kk,\varpi)$ is the sum of $\bar a_{\varpi,j}$ and the weights given by $\kk$ on the two edges of 
tetrahedron $j$ labelled $\bar a_j$, and similarly for $\bar b_j(\kk,\varpi)$, $\bar c_j(\kk,\varpi)$.
Thus, replacing $\kk$ by $\wt \kk = (k_0,\kk)$ does not change $\bar a_j(\kk,\varpi)$
or $\bar c_j(\kk,\varpi)$ but we have $ \bar b_j ( \wt\kk,\varpi) = \bar b_j + k_0$ for $j=1,2,3$,
since the new edge $04$ 
is incident to angles labelled $\bar b_j$ in Figure \ref{fig.23}.\\

We also let $\bar A_j = \bar b_j-\bar c_j, \bar B_j = \bar c_j-\bar a_j, \bar C_j=\bar a_j-\bar b_j$, 
and note that each of these is a sum of contributions from the 4 triangular corners of the 
(truncated) tetrahedron $T_j$. On each of these triangles,  $\bar A_j,\bar B_j,\bar C_j$ represent inward intersection
numbers of oriented normal arcs with the sides of the triangles as shown in Figure \ref{turn_intersect}.

 \begin{figure}[htpb]
 \begin{center} 
 \labellist
\pinlabel $\bar a$ at 103 184
\pinlabel $\bar A$ at 103 7
\pinlabel $\bar b$ at 30 57
\pinlabel $\bar C$ at 25 144
\pinlabel $\bar c$ at 176 57
\pinlabel $\bar B$ at 183 144

\endlabellist
 \includegraphics[width=6cm]{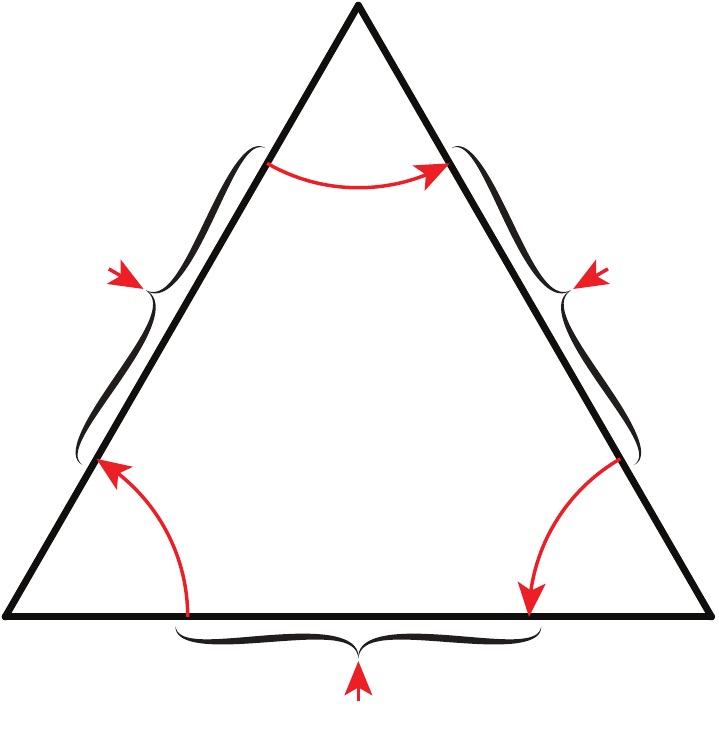}
\caption{Intersection numbers and turning numbers in a triangle.}
\label{turn_intersect}
  \end{center}
  \end{figure}
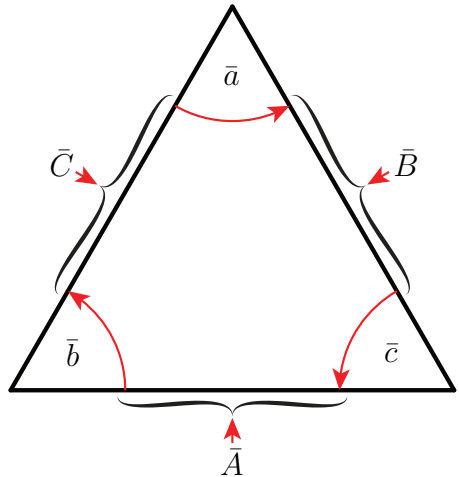

Then we have
\be
 \bar A_j + \bar B_j + \bar C_j =0 \text{ for each } j,
 \ee
 and 
\be
\bar B_1+\bar B_2+\bar B_3=0 .
\label{B_sum=0}
\ee
To prove the last equation, we look at the contribution from each of the 5 corners of the bipyramid
as shown in Figure \ref{all_corners_of_ bipyramid}.
At the polar vertices $0,4$ the contribution to 
$\sum_i \bar B_i= \sum_{i=1}^3 \bar a_i - \sum_{i=1}^3 \bar c_i$ is zero from Figure  \ref{1-3move},
while at the equatorial vertices $1,2,3$ this follows from (\ref{obs1}). Alternatively, one can look at the vertices $1,2,3$ as shown in Figure \ref{all_corners_of_ bipyramid}, and see that, for example, at vertex $1$ the contribution to $B_2$ is cancelled by the contribution to $B_3$. Similar cancellations happen at vertices $2$ and $3$.\\

Now consider the left hand side of (\ref{bipyramid-23-identity}). We have
$$J(T_1;(k_0,\kk),\varpi) = (-q^{\frac{1}{2}})^{-\bar c_1} \ID(\bar c_1-\bar a_1,\bar b_1+k_0-\bar c_1) =  
(-q^{\frac{1}{2}})^{-\bar c_1}\ID (\bar B_1,\bar A_1+k_0),$$
$$J(T_2;(k_0,\kk),\varpi) = (-q^{\frac{1}{2}})^{-\bar c_2} \ID(\bar c_2-\bar a_2,\bar b_2+k_0-\bar c_2) =  
(-q^{\frac{1}{2}})^{-\bar c_2}\ID (\bar B_2,\bar A_2+k_0),$$
and
\begin{align*}
J(T_3;(k_0,\kk),\varpi) &= (-q^{\frac{1}{2}})^{-\bar a_3}\ID (\bar a_3-\bar b_3-k_0,\bar c_3-\bar a_3) 
= (-q^{\frac{1}{2}})^{-\bar a_3}\ID (\bar C_3-k_0,\bar B_3) \\
&=(-q^{\frac{1}{2}})^{-\bar a_3}\ID (-\bar B_3,-\bar C_3+k_0) 
= (-q^{\frac{1}{2}})^{-\bar a_3}\ID (\bar B_1+\bar B_2,-\bar C_3+k_0)
\end{align*}
using  the duality identity (\ref{eq.Z2}) and (\ref{B_sum=0}).\\

The pentagon identity (\ref{eq.pentagon}) can be rewritten
(by replacing $e_3$ by $x_3 + e_0$)
in the form
\begin{align}
\sum_{e_0 \in \BZ}
& q^{e_0} \ID(m_1,x_1+e_0)\ID(m_2,x_2+e_0)\ID(m_1+m_2,x_3+e_0)   \label{pent_eq}\\
&=q^{-x_3} \ID(m_1-x_2+x_3,x_1-x_3)\ID(m_2-x_1+x_3,x_2-x_3) . \nonumber
\end{align}

We apply this with $e_0=k_{0}, m_1=\bar B_1,m_2=\bar B_2, x_1=\bar A_1,x_2=\bar A_2$ and
$x_3=-\bar C_3$.
Then direct calculations, using the observations above, show that 
\begin{enumerate}
\item[(i)] $m_1-x_2+x_3= \bar B_1-\bar A_2-\bar C_3=\bar A_0$, 
\item[(ii)] $x_1-x_3=\bar A_1+\bar C_3 =\bar C_0$, 
\item[(iii)] $m_2-x_1+x_3=\bar B_2-\bar A_1-\bar C_3=\bar B_4$, 
\item[(iv)] $x_2-x_3=\bar A_2+\bar C_3=\bar A_4$, and 
\item[(v)] $\bar c_1+\bar c_2+\bar a_3+2x_3 = \bar b_0+\bar c_4$ hence 
$(-q^{\frac{1}{2}})^{-\bar c_1-\bar c_2-\bar a_3} q^{-x_3}=  (-q^{\frac{1}{2}})^{ - \bar b_0-\bar c_4}.$
\end{enumerate}

Thus, multiplying (\ref{pent_eq}) by $(-q^{\frac{1}{2}})^{-\bar c_1-\bar c_2-\bar a_3}$ gives
\begin{align*}
\sum_{k_{0} \in \Z}   q^{k_{0}}   J(T_1;(k_0,\kk),\varpi) J(T_2;(k_0,\kk),\varpi) J(T_3;(k_0,\kk),\varpi) &=
(-q^{\frac{1}{2}})^{-\bar b_0}\ID(\bar A_0,\bar C_0) (-q^{\frac{1}{2}})^{-\bar c_4}\ID(\bar B_4,\bar A_4) \\
&= J(T_0;\kk,\varpi) J(T_4;\kk,\varpi),
\end{align*}
as desired.\\

 \begin{figure}[p]
 \begin{center} 
  \labellist
\pinlabel {Vertex 1} at -70 899  
\pinlabel {Vertex 2} at -70 714
\pinlabel {Vertex 3} at -70 529  
\pinlabel {Vertex 0} at -70 319 
\pinlabel {Vertex 4} at -70 119 
  
\footnotesize
\pinlabel $\bar C_4$ at 57.5 967
\pinlabel $\bar B_4$ at 191.5 967
\pinlabel $\bar B_0$ at 57.5 833
\pinlabel $\bar C_0$ at 191.5 833
 
\pinlabel $\bar C_2$ at 241 967
\pinlabel $\bar A_3$ at 375 967
\pinlabel $\bar A_2$ at 241 833
\pinlabel $\bar C_3$ at 375 833
 
\pinlabel $\bar A_4$ at 57.5 782.5
\pinlabel $\bar C_4$ at 191.5 782.5
\pinlabel $\bar A_0$ at 57.5 648.5
\pinlabel $\bar B_0$ at 191.5 648.5
 
\pinlabel $\bar C_3$ at 241 782.5
\pinlabel $\bar A_1$ at 375 782.5
\pinlabel $\bar A_3$ at 241 648.5
\pinlabel $\bar C_1$ at 375 648.5
 
\pinlabel $\bar B_4$ at 57.5 598
\pinlabel $\bar A_4$ at 191.5 598
\pinlabel $\bar C_0$ at 57.5 464
\pinlabel $\bar A_0$ at 191.5 464
 
\pinlabel $\bar C_1$ at 241 598
\pinlabel $\bar A_2$ at 375 598
\pinlabel $\bar A_1$ at 241 464
\pinlabel $\bar C_2$ at 375 464

\pinlabel $\bar A_0$ at 139 922
\pinlabel $\bar A_4$ at 139 874 

\pinlabel $\bar C_0$ at 139 737.5
\pinlabel $\bar B_4$ at 139 689.5

\pinlabel $\bar B_0$ at 139 553
\pinlabel $\bar C_4$ at 139 505

\pinlabel $\bar B_3$ at 286 884 
\pinlabel $\bar B_2$ at 327 884 

\pinlabel $\bar B_1$ at 286 699.5
\pinlabel $\bar B_3$ at 327 699.5

\pinlabel $\bar B_2$ at 286 515
\pinlabel $\bar B_1$ at 327 515

\pinlabel $\bar b_2$ at 251 899
\pinlabel $\bar b_3$ at 365 899

\pinlabel $\bar b_3$ at 251 714.5
\pinlabel $\bar b_1$ at 365 714.5

\pinlabel $\bar b_1$ at 251 530
\pinlabel $\bar b_2$ at 365 530

\pinlabel $\bar b_4$ at 71 910
\pinlabel $\bar c_0$ at 71 888
\pinlabel $\bar c_4$ at 177 910
\pinlabel $\bar b_0$ at 177 888

\pinlabel $\bar c_4$ at 71 725.5
\pinlabel $\bar b_0$ at 71 703.5
\pinlabel $\bar a_4$ at 177 725.5
\pinlabel $\bar a_0$ at 177 703.5

\pinlabel $\bar a_4$ at 71 543
\pinlabel $\bar a_0$ at 71 521
\pinlabel $\bar b_4$ at 177 543
\pinlabel $\bar c_0$ at 177 521

\pinlabel $\bar a_2$ at 297 954
\pinlabel $\bar c_3$ at 319 954
\pinlabel $\bar c_2$ at 297 846
\pinlabel $\bar a_3$ at 319 846

\pinlabel $\bar a_3$ at 297 769.5
\pinlabel $\bar c_1$ at 319 769.5
\pinlabel $\bar c_3$ at 297 661.5
\pinlabel $\bar a_1$ at 319 661.5

\pinlabel $\bar a_1$ at 297 585
\pinlabel $\bar c_2$ at 319 585
\pinlabel $\bar c_1$ at 297 477
\pinlabel $\bar a_2$ at 319 477

\pinlabel $\bar a_4$ at 126 954
\pinlabel $\bar a_0$ at 126 846

\pinlabel $\bar b_4$ at 126 769.5
\pinlabel $\bar c_0$ at 126 661.5

\pinlabel $\bar c_4$ at 126 585
\pinlabel $\bar b_0$ at 126 477

\pinlabel $\bar A_4$ at 104 434
\pinlabel $\bar B_1$ at 329 434
\pinlabel $\bar A_0$ at 104 -10
\pinlabel $\bar B_1$ at 329 -10

\pinlabel $\bar B_4$ at 20 295
\pinlabel $\bar C_4$ at 187 295
\pinlabel $\bar C_0$ at 20 131
\pinlabel $\bar B_0$ at 187 131

\pinlabel $\bar B_2$ at 245 295
\pinlabel $\bar B_3$ at 412 295
\pinlabel $\bar B_2$ at 245 131
\pinlabel $\bar B_3$ at 412 131

\pinlabel $\bar A_2$ at 310 390
\pinlabel $\bar C_3$ at 346 390
\pinlabel $\bar C_2$ at 310 36
\pinlabel $\bar A_3$ at 346 36

\pinlabel $\bar C_1$ at 280 335
\pinlabel $\bar A_1$ at 376 335
\pinlabel $\bar A_1$ at 280 92
\pinlabel $\bar C_1$ at 376 92

\pinlabel $\bar A_3$ at 298 306
\pinlabel $\bar C_2$ at 358 306
\pinlabel $\bar C_3$ at 298 122
\pinlabel $\bar A_2$ at 358 122

\pinlabel $\bar b_1$ at 329 358
\pinlabel $\bar b_2$ at 315 335
\pinlabel $\bar b_3$ at 343 335

\pinlabel $\bar b_1$ at 329 66
\pinlabel $\bar b_2$ at 315 91
\pinlabel $\bar b_3$ at 343 91

\pinlabel $\bar a_1$ at 266 391.5
\pinlabel $\bar c_1$ at 392 391.5

\pinlabel $\bar c_1$ at 266 34
\pinlabel $\bar a_1$ at 392 34

\pinlabel $\bar c_2$ at 255 373.5
\pinlabel $\bar a_3$ at 403 373.5

\pinlabel $\bar a_2$ at 255 52
\pinlabel $\bar c_3$ at 403 52

\pinlabel $\bar a_2$ at 319 263
\pinlabel $\bar c_3$ at 339 263
\pinlabel $\bar c_2$ at 319 163
\pinlabel $\bar a_3$ at 339 163

\pinlabel $\bar a_4$ at 104 259
\pinlabel $\bar a_0$ at 104 167

\pinlabel $\bar c_4$ at 30 385
\pinlabel $\bar b_4$ at 178 385

\pinlabel $\bar b_0$ at 30 42
\pinlabel $\bar c_0$ at 178 42

\endlabellist
\includegraphics[height=22cm]{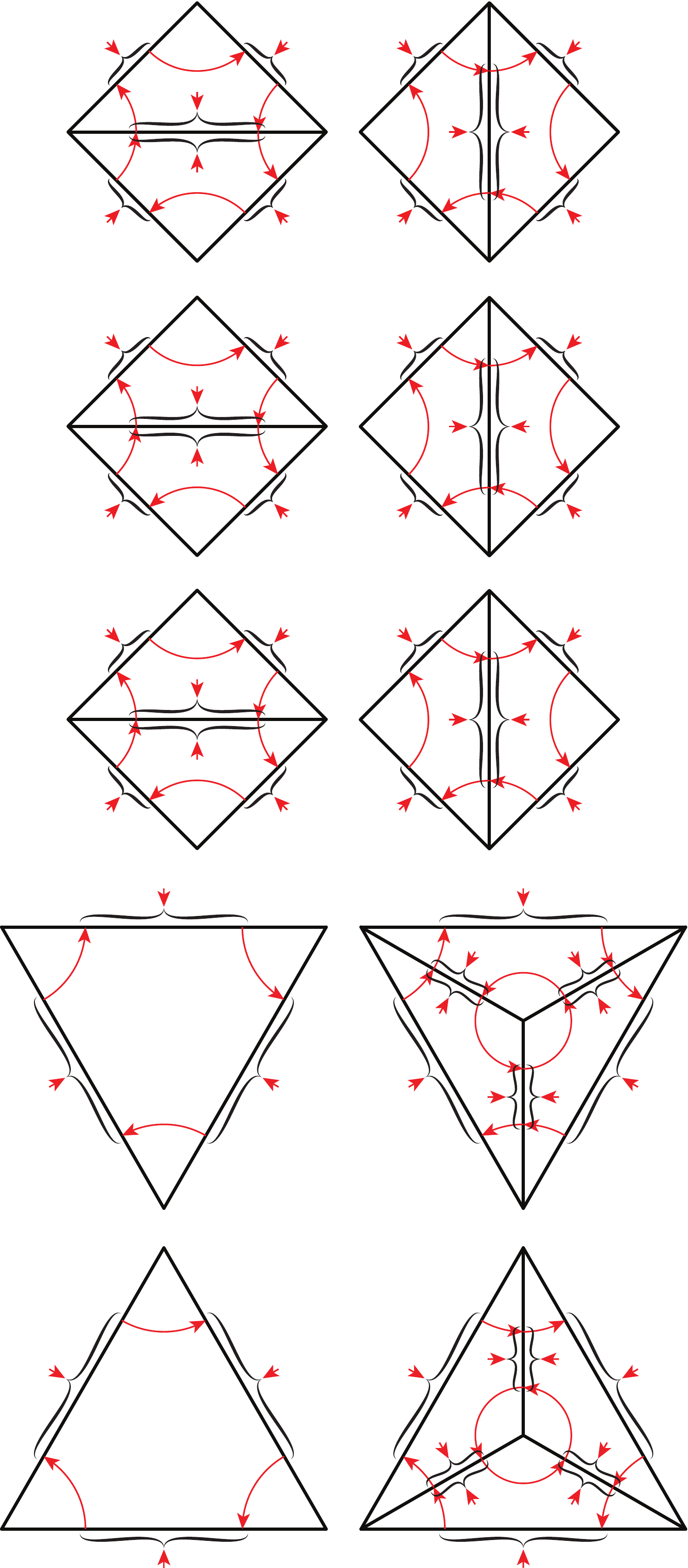}
 \caption{Turning numbers and intersection numbers at all corners of the bipyramid.}
  \label{all_corners_of_ bipyramid}
   \end{center}
 \end{figure}
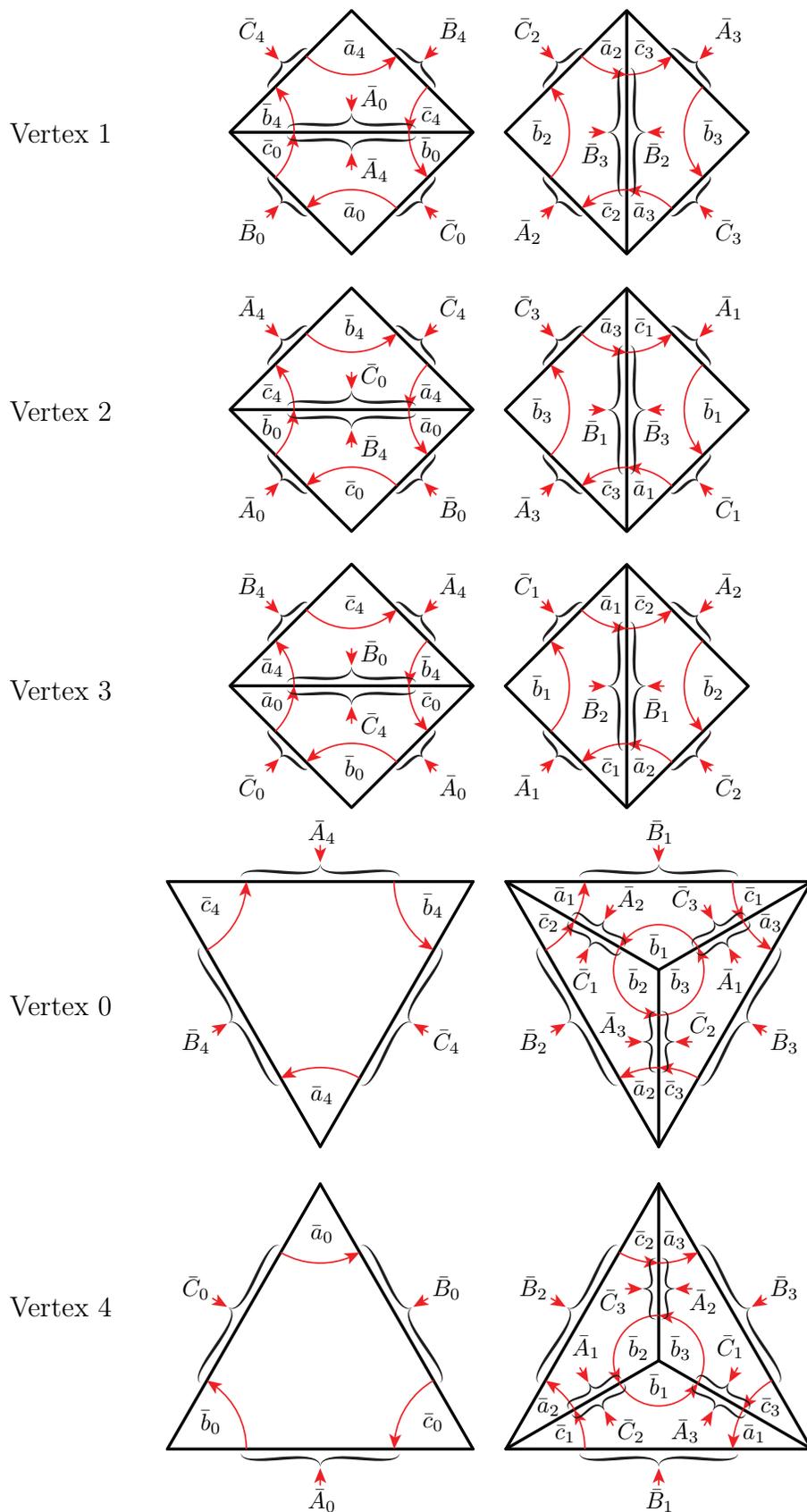

The identities (i)-(v) are proved by looking separately at the contributions
from the 5 corners of the bipyramid, and carefully studying Figure \ref{all_corners_of_ bipyramid}.

 For example, to prove (i), we note that
\begin{align*} 
&\text{ at vertex 1: }  \bar B_1-\bar A_2-\bar C_3 = 0 - \bar A_2-\bar C_3 = - \bar B_0 - \bar C_0 = \bar A_0  \\
&\text{ at vertex 2: }  \bar B_1-\bar A_2-\bar C_3 =  - \bar B_3 - 0 - \bar C_3 = \bar A_3 =  \bar A_0 \\
&\text{ at vertex 3: }  \bar B_1-\bar A_2-\bar C_3 = \bar B_1 - \bar A_2 - 0 = - \bar B_2 -\bar A_2 = \bar C_2 = \bar A_0 \\
&\text{ at vertex 0: } \bar B_1-\bar A_2-\bar C_3 = \bar B_1 +\bar C_1+\bar A_1 = 0 = \bar A_0\\
&\text{ at vertex 4: } \bar B_1-\bar A_2-\bar C_3 = \bar B_1 + 0 = \bar A_0
\end{align*}
For (ii), we note that
\begin{align*} 
&\text{ at vertex 1: }  \bar A_1+\bar C_3 = 0+\bar C_3 = \bar C_0  \\
&\text{ at vertex 2: }  \bar A_1+\bar C_3 = \bar C_4 +\bar A_4 = - \bar B_4 = \bar C_0 \\
&\text{ at vertex 3: }  \bar A_1+\bar C_3 = \bar A_1 = \bar C_0 \\
&\text{ at vertex 0: } \bar A_1+\bar C_3 = 0 = \bar C_0\\
&\text{ at vertex 4: } \bar A_1+\bar C_3 = -\bar C_2 - \bar A_2 = \bar B_2 = \bar C_0
\end{align*}
The equalities (iii), (iv) are verified similarly.\\

Finally, to prove (v),
first note that since $x_3 = -\bar C_3 = \bar b_3-\bar a_3$, we need to show that
$$\bar c_1+ \bar c_2 - \bar a_3 + 2 \bar b_3 = \bar b_0 + \bar c_4.$$
Now at vertices $0$ and $4$, 
 $\bar b_3 = \bar b_3(0,\kk),\varpi) = 0$,
 since we can choose $\varpi$ as in Figure \ref{1-3move}.
Further, $\bar a_j,\bar b_j,\bar c_j=0$ at vertex $j$. Thus, from Figure \ref{all_corners_of_ bipyramid},
\begin{align*} 
&\text{ at vertex 0: } LHS=\bar c_1+\bar c_2-\bar a_3=\bar b_4+\bar c_4-\bar b_4=\bar c_4 = RHS  \text{ from Figure \ref{1-3move}}\\
&\text{ at vertex 4: } LHS = \bar c_1+\bar c_2-\bar a_3=\bar b_0+\bar a_0-\bar a_0= \bar b_0= RHS  \text{ from Figure \ref{1-3move}} \\
&\text{ at vertex 1: }  LHS = \bar c_2 - \bar a_3 + 2 \bar b_3 = \bar b_0 + \bar c_4  \text{  by (\ref{obs2})}  \\
&\text{ at vertex 2: }  LHS = \bar c_1-\bar a_3+2\bar b_3=\bar b_0+\bar c_4  \text{  by (\ref{obs2})}  \\
&\text{ at vertex 3: }  LHS = \bar c_1 +\bar c_2 = \bar b_0+\bar c_4   \text{  by (\ref{obs1})}.
\end{align*}

This completes the proof of  Lemma \ref{lem_bipyramid}.
\end{proof}

%%%%%%%%%%%%%%%%%%%%%%%%%%%%%%%%%%%%%%%%%%%%%%%%%%%%%%%%%%%%%%%%%%%%%%%%%%%%
%%%%%%%%%%%%%%%%%%%%%%%%%%%%%%%%%%%%%%%%%%%%%%%%%%%%%%%%%%%%%%%%%%%%%%%%%%%%
\bibliographystyle{hamsalpha}
\bibliography{biblio}

\newcommand{\etalchar}[1]{$^{#1}$}
\def\cprime{$'$} \def\cprime{$'$}
\providecommand{\bysame}{\leavevmode\hbox to3em{\hrulefill}\thinspace}
\providecommand{\href}[2]{#2}
\providecommand{\eprint}{\begingroup \urlstyle{rm}\Url}
\begin{thebibliography}{CCG{\etalchar{+}}94}

\bibitem[AK]{AK}
J{\o}rgen~E. Andersen and Rinat~M. Kashaev, \emph{A {TQFT} from quantum
  {T}eichm\"uller theory}, \eprint{arXiv:1109.6295}, Preprint 2011.

\bibitem[Aki01]{Akiyoshi}
Hirotaka Akiyoshi, \emph{Finiteness of polyhedral decompositions of cusped
  hyperbolic manifolds obtained by the {E}pstein-{P}enner's method}, Proc.
  Amer. Math. Soc. \textbf{129} (2001), no.~8, 2431--2439 (electronic).

\bibitem[BB05]{BB1}
St{\'e}phane Baseilhac and Riccardo Benedetti, \emph{Classical and quantum
  dilogarithmic invariants of flat {${\rm PSL}(2,\Bbb C)$}-bundles over
  3-manifolds}, Geom. Topol. \textbf{9} (2005), 493--569 (electronic).

\bibitem[BB07]{BB2}
\bysame, \emph{Quantum hyperbolic geometry}, Algebr. Geom. Topol. \textbf{7}
  (2007), 845--917.

\bibitem[BDP]{holo-blocks}
Christopher Beem, Tudor Dimofte, and Sara Pasquetti, \emph{Holomorphic blocks
  in three dimensions}, \eprint{arXiv:1211.1986}, Preprint 2012.

\bibitem[BP97]{BP}
Riccardo Benedetti and Carlo Petronio, \emph{Branched standard spines of
  {$3$}-manifolds}, Lecture Notes in Mathematics, vol. 1653, Springer-Verlag,
  Berlin, 1997.

\bibitem[CCG{\etalchar{+}}94]{CCGLS}
Daryl Cooper, Marc Culler, Henry Gillet, Darren~D. Long, and Peter~B. Shalen,
  \emph{Plane curves associated to character varieties of {$3$}-manifolds},
  Invent. Math. \textbf{118} (1994), no.~1, 47--84.

\bibitem[CDW]{SnapPy}
Marc Culler, Nathan~M. Dunfield, and Jeffery~R. Weeks, \emph{{S}nap{P}y, a
  computer program for studying the geometry and topology of 3-manifolds},
  \url{http://snappy.computop.org/}.

\bibitem[DG]{DG}
Tudor Dimofte and Stavros Garoufalidis, \emph{The quantum content of the gluing
  equations}, \eprint{arXiv:1202.6268}, Preprint 2012.

\bibitem[DGGa]{DGG2}
Tudor Dimofte, Davide Gaiotto, and Sergei Gukov, \emph{3-manifolds and 3d
  indices}, \eprint{arXiv:1112.5179}, Preprint 2011.

\bibitem[DGGb]{DGG1}
\bysame, \emph{Gauge theories labelled by three-manifolds},
  \eprint{arXiv:1108.4389}, Preprint 2011.

\bibitem[DLRS10]{DeLoera:book}
Jes{\'u}s~A. De~Loera, J{\"o}rg Rambau, and Francisco Santos,
  \emph{Triangulations. structures for algorithms and applications}, Algorithms
  and Computation in Mathematics, vol.~25, Springer-Verlag, Berlin, 2010.

\bibitem[EP88]{EP}
D.~B.~A. Epstein and R.~C. Penner, \emph{Euclidean decompositions of noncompact
  hyperbolic manifolds}, J. Differential Geom. \textbf{27} (1988), no.~1,
  67--80.

\bibitem[FKB08]{FK}
Charles Frohman and Joanna Kania-Bartoszynska, \emph{The quantum content of the
  normal surfaces in a three-manifold}, J. Knot Theory Ramifications
  \textbf{17} (2008), no.~8, 1005--1033.

\bibitem[Gara]{Ga:data}
Stavros Garoufalidis, \emph{3d index data},
  \url{http://www.math.gatech.edu/~stavros/publications/3Dindex.data}.

\bibitem[Garb]{Ga:index}
\bysame, \emph{The {3D} index of an ideal triangulation and angle structures},
  \eprint{arXiv:1208.1663}, Preprint 2012.

\bibitem[GK]{GK}
Stavros Garoufalidis and Rinat Kashaev, \emph{The $q$-dilogarithm, its
  state-integrals and their $q$-series}, Preprint 2013.

\bibitem[GKZ94]{GKZ}
I.~M. Gel{\cprime}fand, M.~M. Kapranov, and A.~V. Zelevinsky,
  \emph{Discriminants, resultants, and multidimensional determinants},
  Mathematics: Theory \& Applications, Birkh\"auser Boston Inc., Boston, MA,
  1994.

\bibitem[GL]{GL2}
Stavros Garoufalidis and Thang T.~Q. L{\^e}, \emph{Nahm sums, stability and the
  colored {J}ones polynomial}, \eprint{arXiv:1112.3905}, Preprint 2011.

\bibitem[GS10]{GS}
Fran{\c{c}}ois Gu{\'e}ritaud and Saul Schleimer, \emph{Canonical triangulations
  of {D}ehn fillings}, Geom. Topol. \textbf{14} (2010), no.~1, 193--242.

\bibitem[GV]{GV}
Stavros Garoufalidis and Thao Vuong, \emph{Alternating knots, planar graphs and
  {$q$}-series}, Preprint 2013.

\bibitem[GZ]{GZ}
Stavros Garoufalidis and Don Zagier, \emph{Empirical relations between
  $q$-series and {K}ashaev's invariant of knots}, Preprint 2013.

\bibitem[Hat07]{Hat}
Allen Hatcher, \emph{Notes on basic 3-manifold topology}, 2007, Lecture Notes.

\bibitem[HRS12]{HRS}
Craig~D. Hodgson, J.~Hyam Rubinstein, and Henry Segerman, \emph{Triangulations
  of hyperbolic 3-manifolds admitting strict angle structures}, J. of Top.
  \textbf{5} (2012), no.~5, 887--908, 2012; doi: 10.1112/jtopol/jts022.

\bibitem[Jon87]{Jo}
Vaughan F.~R. Jones, \emph{Hecke algebra representations of braid groups and
  link polynomials}, Ann. of Math. (2) \textbf{126} (1987), no.~2, 335--388.

\bibitem[JR03]{JR}
William Jaco and J.~Hyam Rubinstein, \emph{{$0$}-efficient triangulations of
  3-manifolds}, J. Differential Geom. \textbf{65} (2003), no.~1, 61--168.

\bibitem[Kas97]{Ka}
R.~M. Kashaev, \emph{The hyperbolic volume of knots from the quantum
  dilogarithm}, Lett. Math. Phys. \textbf{39} (1997), no.~3, 269--275.

\bibitem[KLV12]{KLV}
Rinat~M. Kashaev, Feng Luo, and Grigory Vartanov, \emph{A {TQFT} of
  {T}uraev-{V}iro type on shaped triangulations}, 2012,
  \eprint{arXiv:1210.8393}, Preprint.

\bibitem[KR04]{KR1}
Ensil Kang and J.~Hyam Rubinstein, \emph{Ideal triangulations of 3-manifolds.
  {I}. {S}pun normal surface theory}, Proceedings of the {C}asson {F}est, Geom.
  Topol. Monogr., vol.~7, Geom. Topol. Publ., Coventry, 2004,
  \eprint{arXiv:math.GT/0410541}, pp.~235--265.

\bibitem[KR05]{KR2}
\bysame, \emph{Ideal triangulations of 3-manifolds. {II}. {T}aut and angle
  structures}, Algebr. Geom. Topol. \textbf{5} (2005), 1505--1533.

\bibitem[Lac00a]{Lackenby2}
Marc Lackenby, \emph{Taut ideal triangulations of 3-manifolds}, Geom. Topol.
  \textbf{4} (2000), 369--395 (electronic).

\bibitem[Lac00b]{Lackenby}
\bysame, \emph{Word hyperbolic {D}ehn surgery}, Invent. Math. \textbf{140}
  (2000), no.~2, 243--282.

\bibitem[Lod04]{Loday}
Jean-Louis Loday, \emph{Realization of the {S}tasheff polytope}, Arch. Math.
  (Basel) \textbf{83} (2004), no.~3, 267--278.

\bibitem[LT08]{LT}
Feng Luo and Stephan Tillmann, \emph{Angle structures and normal surfaces},
  Trans. Amer. Math. Soc. \textbf{360} (2008), no.~6, 2849--2866.

\bibitem[Mac71]{MacLane}
Saunders MacLane, \emph{Categories for the working mathematician},
  Springer-Verlag, New York, 1971, Graduate Texts in Mathematics, Vol. 5.

\bibitem[Mat87]{Ma1}
Sergei~V. Matveev, \emph{Transformations of special spines, and the {Z}eeman
  conjecture}, Izv. Akad. Nauk SSSR Ser. Mat. \textbf{51} (1987), no.~5,
  1104--1116, 1119.

\bibitem[Mat07]{Ma2}
\bysame, \emph{Algorithmic topology and classification of 3-manifolds}, second
  ed., Algorithms and Computation in Mathematics, vol.~9, Springer, Berlin,
  2007.

\bibitem[Neu92]{Neumann-combi}
Walter~D. Neumann, \emph{Combinatorics of triangulations and the
  {C}hern-{S}imons invariant for hyperbolic {$3$}-manifolds}, Topology '90
  ({C}olumbus, {OH}, 1990), Ohio State Univ. Math. Res. Inst. Publ., vol.~1, de
  Gruyter, Berlin, 1992, pp.~243--271.

\bibitem[NZ85]{NZ}
Walter~D. Neumann and Don Zagier, \emph{Volumes of hyperbolic three-manifolds},
  Topology \textbf{24} (1985), no.~3, 307--332.

\bibitem[Pet95]{Petronio:thesis}
Carlo Petronio, \emph{Standard spines and 3-manifolds}, Ph.D. thesis, Scuola
  Normale Superiore, Pisa, 1995.

\bibitem[Pie88]{Pi}
Riccardo Piergallini, \emph{Standard moves for standard polyhedra and spines},
  Rend. Circ. Mat. Palermo (2) Suppl. (1988), no.~18, 391--414, Third National
  Conference on Topology (Italian) (Trieste, 1986).

\bibitem[Rub97]{Rub}
J.~H. Rubinstein, \emph{Polyhedral minimal surfaces, {H}eegaard splittings and
  decision problems for {$3$}-dimensional manifolds}, Geometric topology
  ({A}thens, {GA}, 1993), AMS/IP Stud. Adv. Math., vol.~2, Amer. Math. Soc.,
  Providence, RI, 1997, pp.~1--20.

\bibitem[San06]{Santos:ICM}
Francisco Santos, \emph{Geometric bistellar flips: the setting, the context and
  a construction}, International {C}ongress of {M}athematicians. {V}ol. {III},
  Eur. Math. Soc., Z\"urich, 2006, pp.~931--962.

\bibitem[Sta63]{Stasheff:1963}
James~Dillon Stasheff, \emph{Homotopy associativity of {$H$}-spaces. {I},
  {II}}, Trans. Amer. Math. Soc. 108 (1963), 275-292; ibid. \textbf{108}
  (1963), 293--312.

\bibitem[Sta97]{Stasheff:Luminy}
Jim Stasheff, \emph{From operads to ``physically'' inspired theories}, Operads:
  {P}roceedings of {R}enaissance {C}onferences ({H}artford, {CT}/{L}uminy,
  1995), Contemp. Math., vol. 202, Amer. Math. Soc., Providence, RI, 1997,
  pp.~53--81.

\bibitem[Sto00]{Sto}
Michelle Stocking, \emph{Almost normal surfaces in {$3$}-manifolds}, Trans.
  Amer. Math. Soc. \textbf{352} (2000), no.~1, 171--207.

\bibitem[Thu77]{Th}
William Thurston, \emph{The geometry and topology of 3-manifolds}, 1977,
  Lecture notes, Princeton.

\bibitem[Wee85]{Weeks:thesis}
Jeffrey Weeks, \emph{Hyperbolic structures on three-manifolds}, Ph.D. thesis,
  Princeton University, 1985.

\end{thebibliography}
\end{document}